\documentclass[letter, 11pt, british]{amsart}

\usepackage{mathpazo}
\usepackage{libertine}
\usepackage[libertine]{newtxmath}

\usepackage{babel}
\usepackage{enumitem}
\usepackage{hyperref}
\usepackage[utf8]{inputenc}
\usepackage{newunicodechar}
\usepackage{mathtools}
\usepackage{varioref}
\usepackage[arrow,curve,matrix]{xy}

\usepackage{colortbl}
\usepackage{graphicx}
\usepackage{tikz}

\usepackage{mathrsfs}

\definecolor{linkred}{rgb}{0.7,0.2,0.2}
\definecolor{linkblue}{rgb}{0,0.2,0.6}

\numberwithin{figure}{section}

\usepackage[hyperpageref]{backref}

\setdescription{labelindent=\parindent, leftmargin=2\parindent}
\setitemize[1]{labelindent=\parindent, leftmargin=2\parindent}
\setenumerate[1]{labelindent=0cm, leftmargin=*, widest=iiii}

\newunicodechar{א}{\ensuremath{\aleph}}
\newunicodechar{α}{\ensuremath{\alpha}}
\newunicodechar{β}{\ensuremath{\beta}}
\newunicodechar{χ}{\ensuremath{\chi}}
\newunicodechar{δ}{\ensuremath{\delta}}
\newunicodechar{ε}{\ensuremath{\varepsilon}}
\newunicodechar{Δ}{\ensuremath{\Delta}}
\newunicodechar{η}{\ensuremath{\eta}}
\newunicodechar{γ}{\ensuremath{\gamma}}
\newunicodechar{Γ}{\ensuremath{\Gamma}}
\newunicodechar{ι}{\ensuremath{\iota}}
\newunicodechar{κ}{\ensuremath{\kappa}}
\newunicodechar{λ}{\ensuremath{\lambda}}
\newunicodechar{Λ}{\ensuremath{\Lambda}}
\newunicodechar{ν}{\ensuremath{\nu}}
\newunicodechar{μ}{\ensuremath{\mu}}
\newunicodechar{ω}{\ensuremath{\omega}}
\newunicodechar{Ω}{\ensuremath{\Omega}}
\newunicodechar{π}{\ensuremath{\pi}}
\newunicodechar{Π}{\ensuremath{\Pi}}
\newunicodechar{φ}{\ensuremath{\phi}}
\newunicodechar{Φ}{\ensuremath{\Phi}}
\newunicodechar{ψ}{\ensuremath{\psi}}
\newunicodechar{Ψ}{\ensuremath{\Psi}}
\newunicodechar{ρ}{\ensuremath{\rho}}
\newunicodechar{σ}{\ensuremath{\sigma}}
\newunicodechar{Σ}{\ensuremath{\Sigma}}
\newunicodechar{τ}{\ensuremath{\tau}}
\newunicodechar{θ}{\ensuremath{\theta}}
\newunicodechar{Θ}{\ensuremath{\Theta}}
\newunicodechar{ξ}{\ensuremath{\xi}}
\newunicodechar{Ξ}{\ensuremath{\Xi}}
\newunicodechar{ζ}{\ensuremath{\zeta}}

\newunicodechar{ℓ}{\ensuremath{\ell}}
\newunicodechar{ï}{\"{\i}}

\newunicodechar{𝔸}{\ensuremath{\bA}}
\newunicodechar{𝔹}{\ensuremath{\bB}}
\newunicodechar{ℂ}{\ensuremath{\bC}}
\newunicodechar{𝔻}{\ensuremath{\bD}}
\newunicodechar{𝔼}{\ensuremath{\bE}}
\newunicodechar{𝔽}{\ensuremath{\bF}}
\newunicodechar{𝔾}{\ensuremath{\bG}}
\newunicodechar{ℕ}{\ensuremath{\bN}}
\newunicodechar{ℙ}{\ensuremath{\bP}}
\newunicodechar{ℚ}{\ensuremath{\bQ}}
\newunicodechar{ℝ}{\ensuremath{\bR}}
\newunicodechar{𝕏}{\ensuremath{\bX}}
\newunicodechar{ℤ}{\ensuremath{\bZ}}
\newunicodechar{𝒜}{\ensuremath{\sA}}
\newunicodechar{ℬ}{\ensuremath{\sB}}
\newunicodechar{𝒞}{\ensuremath{\sC}}
\newunicodechar{𝒟}{\ensuremath{\sD}}
\newunicodechar{ℰ}{\ensuremath{\sE}}
\newunicodechar{ℱ}{\ensuremath{\sF}}
\newunicodechar{𝒢}{\ensuremath{\sG}}
\newunicodechar{ℋ}{\ensuremath{\sH}}
\newunicodechar{𝒥}{\ensuremath{\sJ}}
\newunicodechar{ℒ}{\ensuremath{\sL}}
\newunicodechar{ℳ}{\ensuremath{\sM}}
\newunicodechar{𝒪}{\ensuremath{\sO}}
\newunicodechar{𝒬}{\ensuremath{\sQ}}
\newunicodechar{𝒮}{\ensuremath{\sS}}
\newunicodechar{𝒯}{\ensuremath{\sT}}
\newunicodechar{𝒲}{\ensuremath{\sW}}

\newunicodechar{∂}{\ensuremath{\partial}}
\newunicodechar{∇}{\ensuremath{\nabla}}

\newunicodechar{↺}{\ensuremath{\circlearrowleft}}
\newunicodechar{∞}{\ensuremath{\infty}}
\newunicodechar{⊕}{\ensuremath{\oplus}}
\newunicodechar{⊗}{\ensuremath{\otimes}}
\newunicodechar{•}{\ensuremath{\bullet}}
\newunicodechar{Λ}{\ensuremath{\wedge}}
\newunicodechar{↪}{\ensuremath{\into}}
\newunicodechar{→}{\ensuremath{\to}}
\newunicodechar{↦}{\ensuremath{\mapsto}}
\newunicodechar{⨯}{\ensuremath{\times}}
\newunicodechar{∪}{\ensuremath{\cup}}
\newunicodechar{∩}{\ensuremath{\cap}}
\newunicodechar{⊋}{\ensuremath{\supsetneq}}
\newunicodechar{⊇}{\ensuremath{\supseteq}}
\newunicodechar{⊃}{\ensuremath{\supset}}
\newunicodechar{⊊}{\ensuremath{\subsetneq}}
\newunicodechar{⊆}{\ensuremath{\subseteq}}
\newunicodechar{⊂}{\ensuremath{\subset}}
\newunicodechar{⊄}{\ensuremath{\not \subset}}
\newunicodechar{≥}{\ensuremath{\geq}}
\newunicodechar{≠}{\ensuremath{\neq}}
\newunicodechar{≫}{\ensuremath{\gg}}
\newunicodechar{≪}{\ensuremath{\ll}}

\newunicodechar{≤}{\ensuremath{\leq}}
\newunicodechar{∈}{\ensuremath{\in}}
\newunicodechar{∉}{\ensuremath{\not \in}}
\newunicodechar{∖}{\ensuremath{\setminus}}
\newunicodechar{◦}{\ensuremath{\circ}}
\newunicodechar{°}{\ensuremath{^\circ}}
\newunicodechar{…}{\ifmmode\mathellipsis\else\textellipsis\fi}
\newunicodechar{·}{\ensuremath{\cdot}}
\newunicodechar{⋯}{\ensuremath{\cdots}}
\newunicodechar{∅}{\ensuremath{\emptyset}}
\newunicodechar{⇒}{\ensuremath{\Rightarrow}}

\newunicodechar{⁰}{\ensuremath{^0}}
\newunicodechar{¹}{\ensuremath{^1}}
\newunicodechar{²}{\ensuremath{^2}}
\newunicodechar{³}{\ensuremath{^3}}
\newunicodechar{⁴}{\ensuremath{^4}}
\newunicodechar{⁵}{\ensuremath{^5}}
\newunicodechar{⁶}{\ensuremath{^6}}
\newunicodechar{⁷}{\ensuremath{^7}}
\newunicodechar{⁸}{\ensuremath{^8}}
\newunicodechar{⁹}{\ensuremath{^9}}
\newunicodechar{ⁱ}{\ensuremath{^i}}

\newunicodechar{⌈}{\ensuremath{\lceil}}
\newunicodechar{⌉}{\ensuremath{\rceil}}
\newunicodechar{⌊}{\ensuremath{\lfloor}}
\newunicodechar{⌋}{\ensuremath{\rfloor}}

\newunicodechar{≅}{\ensuremath{\cong}}
\newunicodechar{⇔}{\ensuremath{\Leftrightarrow}}
\newunicodechar{∃}{\ensuremath{\exists}}
\newunicodechar{±}{\ensuremath{\pm}}

\DeclareFontFamily{OMS}{rsfs}{\skewchar\font'60}
\DeclareFontShape{OMS}{rsfs}{m}{n}{<-5>rsfs5 <5-7>rsfs7 <7->rsfs10 }{}
\DeclareSymbolFont{rsfs}{OMS}{rsfs}{m}{n}
\DeclareSymbolFontAlphabet{\scr}{rsfs}
\DeclareSymbolFontAlphabet{\scr}{rsfs}

\DeclareFontFamily{U}{mathx}{\hyphenchar\font45}
\DeclareFontShape{U}{mathx}{m}{n}{
      <5> <6> <7> <8> <9> <10>
      <10.95> <12> <14.4> <17.28> <20.74> <24.88>
      mathx10
      }{}
\DeclareSymbolFont{mathx}{U}{mathx}{m}{n}
\DeclareFontSubstitution{U}{mathx}{m}{n}
\DeclareMathAccent{\wcheck}{0}{mathx}{"71}

\DeclareMathOperator{\const}{const}

\DeclareMathOperator{\Hom}{Hom}
\DeclareMathOperator{\Id}{Id}
\DeclareMathOperator{\Image}{Image}
\DeclareMathOperator{\img}{img}
\DeclareMathOperator{\Pic}{Pic}
\DeclareMathOperator{\rank}{rank}
\DeclareMathOperator{\Ramification}{Ramification}
\DeclareMathOperator{\red}{red}
\DeclareMathOperator{\reg}{reg}

\DeclareMathOperator{\Spec}{Spec}
\DeclareMathOperator{\Sym}{Sym}
\DeclareMathOperator{\supp}{supp}

\newcommand{\sA}{\scr{A}}
\newcommand{\sB}{\scr{B}}
\newcommand{\sC}{\scr{C}}
\newcommand{\sD}{\scr{D}}
\newcommand{\sE}{\scr{E}}
\newcommand{\sF}{\scr{F}}
\newcommand{\sG}{\scr{G}}
\newcommand{\sH}{\scr{H}}
\newcommand{\sHom}{\scr{H}\negthinspace om}
\newcommand{\sI}{\scr{I}}
\newcommand{\sJ}{\scr{J}}

\newcommand{\sL}{\scr{L}}
\newcommand{\sM}{\scr{M}}

\newcommand{\sO}{\scr{O}}

\newcommand{\sQ}{\scr{Q}}

\newcommand{\sS}{\scr{S}}
\newcommand{\sT}{\scr{T}}

\newcommand{\sW}{\scr{W}}

\newcommand{\cC}{\mathcal C}

\newcommand{\bA}{\mathbb{A}}
\newcommand{\bB}{\mathbb{B}}
\newcommand{\bC}{\mathbb{C}}
\newcommand{\bD}{\mathbb{D}}
\newcommand{\bE}{\mathbb{E}}
\newcommand{\bF}{\mathbb{F}}
\newcommand{\bG}{\mathbb{G}}

\newcommand{\bN}{\mathbb{N}}

\newcommand{\bP}{\mathbb{P}}
\newcommand{\bQ}{\mathbb{Q}}
\newcommand{\bR}{\mathbb{R}}

\newcommand{\bX}{\mathbb{X}}

\newcommand{\bZ}{\mathbb{Z}}

\theoremstyle{plain}
\newtheorem{thm}{Theorem}[section]

\newtheorem{defn}[thm]{Definition}
\newtheorem{fact}[thm]{Fact}
\newtheorem{lem}[thm]{Lemma}

\newtheorem{prop}[thm]{Proposition}

\theoremstyle{remark}

\newtheorem{asswlog}[thm]{Assumption w.l.o.g.}
\newtheorem{claim}[thm]{Claim}
\newtheorem{c-n-d}[thm]{Claim and Definition}
\newtheorem{consequence}[thm]{Consequence}
\newtheorem{construction}[thm]{Construction}
\newtheorem{computation}[thm]{Computation}

\newtheorem{explanation}[thm]{Explanation}
\newtheorem{notation}[thm]{Notation}
\newtheorem{obs}[thm]{Observation}
\newtheorem{rem}[thm]{Remark}

\newtheorem*{rem-nonumber}{Remark}
\newtheorem{setting}[thm]{Setting}

\numberwithin{equation}{thm}

\setlist[enumerate]{label=(\thethm.\arabic*), before={\setcounter{enumi}{\value{equation}}}, after={\setcounter{equation}{\value{enumi}}}}

\newcommand{\into}{\hookrightarrow}

\newcommand{\wtilde}{\widetilde}
\newcommand{\what}{\widehat}

\newcommand\CounterStep{\addtocounter{thm}{1}\setcounter{equation}{0}}

\newcommand{\factor}[2]{\left. \raise 2pt\hbox{$#1$} \right/\hskip -2pt\raise -2pt\hbox{$#2$}}
\newcommand{\Preprint}[1]{}

\newcommand{\subversionInfo}{}
\newcommand{\svnid}[1]{}
\newcommand{\approvals}[2][Approval]{}

\renewcommand{\phi}{\varphi}

\author{Stefan Kebekus} %
\address{Stefan Kebebus, Mathematisches Institut, Albert-Ludwigs-Universität
  Freiburg, Ernst-Zermelo-Straße 1, 79104 Freiburg im Breisgau, Germany \&
  Freiburg Institute for Advanced Studies (FRIAS), Freiburg im Breisgau,
  Germany} %
\email{\href{mailto:stefan.kebekus@math.uni-freiburg.de}{stefan.kebekus@math.uni-freiburg.de}} %
\urladdr{\url{https://cplx.vm.uni-freiburg.de}}

\author{Jorge Vitório Pereira} %
\address{Jorge Vitório Pereira, IMPA, Estrada Dona Castorina 110, 22460-320, Rio
  de Janeiro, Brazil} %
\email{\href{mailto:jvp@impa.br}{jvp@impa.br}} %
\urladdr{\url{https://www.impa.br/~jvp}}

\author{Arne Smeets} %
\address{Arne Smeets, Department of Mathematics, KU Leuven, Celestijnenlaan 200B, 3001 Heverlee, Belgium} %
\email{\href{mailto:arne.smeets@kuleuven.be}{arne.smeets@kuleuven.be}}

\thanks{Stefan Kebekus gratefully acknowledges partial support through a
  fellowship of the Freiburg Institute of Advanced Studies (FRIAS).  This
  material is based in part upon work supported by the National Science
  Foundation under Grant No.~DMS-1440140 while Stefan Kebekus was in residence
  at the Mathematical Sciences Research Institute in Berkeley, California,
  during the Spring 2019 semester.  Jorge Vitório Pereira acknowledges support
  of CNPq, Faperj, and the Freiburg Institute for Advanced Studies (FRIAS).  The
  research leading to these results has received funding from the People
  Programme (Marie Curie Actions) of the European Union's Seventh Programme
  (FP7/2007-2013) under REA grant agreement nr.~[609305].  Arne Smeets
  acknowledges support from ERC Consolidator Grant MOTMELSUM (grant agreement
  $\sharp$615722), a postdoctoral fellowship of FWO (Research Foundation --
  Flanders) and a Veni fellowship from NWO (Netherlands Organisation for
  Scientific Research).}



\title[Failure of the Brauer--Manin principle for a simply connected fourfold]{Failure of the Brauer--Manin principle for a simply connected fourfold over a global function field, via orbifold Mordell}
\date{\today}

\makeatletter
\hypersetup{
  pdfauthor={\authors},
  pdftitle={\@title},
  pdfsubject={\@subjclass},
  pdfkeywords={\@keywords},
  pdfstartview={Fit},
  pdfpagelayout={TwoColumnRight},
  pdfpagemode={UseOutlines},
  bookmarks,
  colorlinks,
  linkcolor=linkblue,
  citecolor=linkred,
  urlcolor=linkred
}
\makeatother

\DeclareMathOperator{\Brauer}{Br}
\DeclareMathOperator{\Div}{Div}
\DeclareMathOperator{\etale}{\acute{e}t}
\DeclareMathOperator{\Etale}{\acute{E}t}

\DeclareMathOperator{\fract}{fract}
\DeclareMathOperator{\Hilb}{Hilb}
\DeclareMathOperator{\lcm}{lcm}
\DeclareMathOperator{\mult}{mult}
\DeclareMathOperator{\num}{num}
\DeclareMathOperator{\ord}{ord}
\DeclareMathOperator{\tame}{tame}

\DeclareMathOperator{\unit}{unit}
\DeclareMathOperator{\wild}{wild}

\theoremstyle{remark}

\newtheorem{o-and-n}[thm]{Observation and Notation}
\newtheorem{reminder}[thm]{Reminder}
\newtheorem{r-and-d}[thm]{Reminder and Notation}
\newtheorem{s-and-d}[thm]{Setting and Notation}
\newtheorem{situation}[thm]{Situation}

\begin{document}

\approvals[Abstract]{Arne & yes \\ Jorge & yes \\ Stefan & yes}
\begin{abstract}
Almost one decade ago, Poonen constructed the first examples of algebraic
varieties over global fields for which Skorobogatov’s étale Brauer–Manin
obstruction does not explain the failure of the Hasse principle.  By now,
several constructions are known, but they all share common geometric features
such as large fundamental groups.

In this paper, we construct simply connected fourfolds over global fields of
positive characteristic for which the Brauer–Manin machinery fails.  Contrary to
earlier work in this direction, our construction does not rely on major
conjectures.  Instead, we establish a new diophantine result of independent
interest: a Mordell-type theorem for Campana’s “geometric orbifolds” over
function fields of positive characteristic.  Along the way, we also construct
the first example of simply connected surface of general type over a global
field with a non-empty, but non-Zariski dense set of rational points.

\end{abstract}

\maketitle
\tableofcontents

%
%
\svnid{$Id: 01-intro.tex 603 2021-11-04 10:09:01Z kebekus $}

\section{Introduction}
\subversionInfo

\subsection{Insufficiency of the Brauer--Manin obstruction}
\approvals{Arne & yes \\ Jorge & yes \\ Stefan & yes}

Let $X$ be a smooth, projective variety over a global field $K$ with adèle ring
$𝔸_K$.  To decide whether the variety $X$ has a $K$-rational point, the
\emph{Brauer--Manin obstruction} and its refinements are often very useful
tools.  A conjecture of Colliot-Thélène asserts that if $X$ is rationally
connected and if $K$ is a number field, then the Brauer--Manin obstruction to
the existence of a rational point is the only one.  In other words: if the
Brauer--Manin set $X(𝔸_K)^{\Brauer}$ is non-empty, then $X(K)$ is non-empty.
This conjecture is now known in many special cases, including many geometrically
rational surfaces and many homogeneous spaces of linear algebraic groups.  On
the other side of the geometric spectrum, we know very little.  For example, we
do not have \emph{any} examples of smooth hypersurfaces of general type in
$ℙ^n_ℚ$, where $n ≥ 4$, for which the Hasse principle fails, even though one
should expect such examples to abound.  An interesting line of recent research
has tackled the problem of constructing varieties for which one can prove that
the Brauer--Manin machinery \emph{fails}.

\subsection{Main result}
\approvals{Arne & --- \\ Jorge & yes \\ Stefan & yes}

This paper addresses the challenge to construct the first example of a simply
connected variety over a global field for which the Brauer--Manin obstruction
does not explain the failure of the Hasse principle.  The following is our main
result.

\begin{thm}\label{maintheorem1}
  For every sufficiently large prime $p$, there exist a global function field
  $K$ of characteristic $p$ and a smooth, projective, geometrically integral
  fourfold $Z$ over $K$ such that $π_1^{\etale}(Z_{\overline{K}}) = 0$,
  $Z(𝔸_K)^{\Brauer} ≠ ∅$ and $Z(K) = ∅$.
\end{thm}

\begin{rem}
  Theorem~\ref{maintheorem1} is no abstract existence result.  The construction
  carried out in Section~\ref{sec:construction} is quite explicit.
\end{rem}

\begin{rem}
  To the best of our knowledge, Theorem~\ref{maintheorem1} provides the first
  example of a simply connected variety over a global field where Brauer--Manin
  fails.  Over number fields, simply connected examples have been constructed
  \emph{conditionally (assuming the Bombieri--Lang conjectures)} by Sarnak--Wang
  \cite{SW} and Poonen \cite{Poonenbis}, and \emph{conditionally (assuming the
    $abc$ conjecture)} by Smeets \cite[§4]{Smeets}.
\end{rem}

\subsubsection{Earlier results}
\approvals{Arne & yes \\ Jorge & yes \\ Stefan & yes}

The first example of a smooth, projective variety over a number field for which
the Brauer--Manin obstruction does not explain the failure of the Hasse
principle was a bi-elliptic surface constructed by Skorobogatov,
\cite[§2]{Skorobogatov}.  For this surface, however, the failure of the Hasse
principle could be explained by the (finer) étale Brauer--Manin obstruction
\cite[§3]{Skorobogatov}.  Next, Poonen \cite{Poonen0} found the first examples
of varieties for which this finer obstruction fails as well.  His examples are
threefolds fibred over a curve of genus at least one.  Soon after,
Harpaz--Skorobogatov \cite{HS} constructed surfaces with this property,
Colliot-Thélène--Pál--Skorobogatov \cite{CTPS} found examples of quadric
bundles, and Smeets \cite[§3]{Smeets} came up with the first examples with
trivial Albanese variety.

\subsection{A simply connected surface with a non-empty, but non-Zariski dense set of rational points}
\approvals{Arne & --- \\ Jorge & yes \\ Stefan & yes}

The proof of Theorem~\ref{maintheorem1} builds on an idea that goes back to the
work of Poonen \cite{Poonen0}: to construct a simply connected example over a
field $K$, it suffices to construct a simply connected $K$-surface $S$, equipped
with a fibration $f: S → ℙ¹_K$, such that only finitely many (but more than
zero) fibres of $f$ have $K$-rational points.  The following theorem claims the
existence of $S$ abstractly.  The construction given in
Section~\ref{section:fibrations} is however really quite explicit.

\begin{thm}\label{maintheorem2}
  For every sufficiently large prime $p$, there exist a global function field
  $K$ of characteristic $p$ and a smooth, projective, geometrically integral and
  geometrically simply connected $K$-surface $Y$ of general type, equipped with
  a dominant morphism $π: Y → ℙ¹_K$, such that $π \bigl( Y(K) \bigr)$ is finite
  and non-empty.
\end{thm}

To reach this goal, we adapt a cunning strategy devised by Campana in
\cite{Cam05}.  He constructed simply connected surfaces of general type over
$ℚ$ fibred over $ℙ¹_ℚ$ with a so-called ``orbifold base of general type'', and
argued that his ``orbifold Mordell conjecture'' would imply non-Zariski density
of the set of rational points on the surface.

We make his work unconditional in positive characteristic.  To achieve this, we
build on a construction of Stoppino \cite{Stoppino} which is simpler than the
one used by Campana in \cite[§5]{Cam05}, and more easily transportable to
positive characteristic.  We combine this construction with a new diophantine
ingredient: a version the Mordell conjecture for Campana's ``geometric
orbifolds'' over global function fields.  The orbifolds considered by Campana
are not stacks, but simply pairs consisting of a smooth variety and a certain
type of $ℚ$-divisor.  We call these \emph{$\cC$-pairs}, see
Section~\ref{subsection:orbifoldMordell} and Definition~\ref{def:pair} for more
details.

\begin{rem}
  It is a major open problem to construct such an $X$ over $ℚ$ for which $X(ℚ)$
  is both non-empty and not Zariski dense, see \cite[Rem.~1.4]{Poonen3}.  To the
  best of our knowledge, Theorem~\ref{maintheorem2} yields the very first
  example of a simply connected surface over a global field $K$ with $X(K) ≠ ∅$
  for which one can verify the non-Zariski density of the set rational points
  unconditionally, in the direction of the Bombieri--Lang conjecture.  Our
  methods are, however, restricted to global fields of positive characteristic.
  As will be clear from our proof of Theorem~\ref{maintheorem2}, the very same
  statement holds with $K$ replaced by a function field of characteristic zero,
  what was already known to Campana.  After the present work was made public,
  Carlo Gasbarri explained to us in private communication how to construct
  explicit examples of smooth projective varieties of arbitrary positive
  dimension, defined over a function field of characteristic zero, which are
  simply connected and whose sets of rational points are not Zariski dense.
\end{rem}

\begin{rem}
  In fact, the construction used in Section~\ref{section:fibrations} to prove
  Theorem \ref{maintheorem2} immediately yields a slightly stronger statement:
  for the field $K$ and surface $Y$ constructed in
  Section~\ref{section:fibrations}, we know that the set $π \bigl( Y(L) \bigr)$
  remains finite for any finite, separable field extension $L/K$.
 \end{rem}

\subsection{An orbifold version of the Mordell conjecture} \label{subsection:orbifoldMordell}
\approvals{Arne & yes \\ Jorge & yes \\ Stefan & yes}

Let $k$ be an algebraically closed field, and let $K$ be the function field of a
$k$-curve.  The classical Mordell conjecture over function fields, proven by
Grauert and Manin in characteristic zero and by Samuel in positive
characteristic, states that if $X$ is a smooth, projective curve of genus
$g(X) ≥ 2$ over $K$, then $X$ has finitely many $K$-rational points, provided
that it is \emph{non-isotrivial}.  In the orbifold setting, the curve $X$ can be
of arbitrary genus, but comes equipped with $ℚ$-divisor $D$, with coefficients
of the form $1-\frac{1}{m}$.  We refer the reader to Section~\ref{ssec:cpairs}
for a precise definition of the ``$\cC$-pair'' $(X,D)$.  Given a sufficiently
nice integral model of $(X,D)$, one can consider the set of \emph{$\cC$-integral
  points}.  One should think about these as interpolating between two classical
notions: $K$-rational points on $X$ on the one hand, and integral points on a
model of $(X, ⌈D⌉)$ on the other hand.  Campana proposed a version of Mordell's
conjecture for $\cC$-pairs over global fields of characteristic zero.  We prove
such a statement unconditionally over global fields of positive characteristic,
where the formulation becomes slightly more involved.  Stated in rather loose
terms, our orbifold Mordell-type theorem says the following.

\begin{thm}[Mordell-type theorem for $\cC$-integral points (= Theorem~\ref{thm:Mordell})]\label{maintheorem3}
  Let $K$ be the function field of a curve defined over an algebraically closed
  field $k$.  Let $(X,D)$ be a one--dimensional $\cC$-pair of general type over
  $K$, with non-vanishing Kodaira--Spencer class.  Then the set of
  $\cC$-integral points on any integral model of $(X,D)$ is finite.
\end{thm}

\begin{rem}
  The assumption that $(X,D)$ is \emph{of general type} simply means that the
  degree of $K_X + D$ is positive; this is an analogue of the assumption
  ``$g(X) ≥ 2$'' in the classical Mordell-type theorem.  The analogue of the
  condition on isotriviality is slightly more subtle.  One could declare a
  $\cC$-pair $(X,D)$ to be ``non-isotrivial'' if the associated ``logarithmic''
  pair $(X, ⌈D⌉)$ is non-isotrivial.  Unlike in the classical setting, this does
  however \emph{not} suffice to guarantee finiteness of the set of
  $\cC$-integral points, see Section~\ref{ssec:necessity mordell}.  We impose
  the stronger condition that the \emph{Kodaira--Spencer class} associated to
  $(X, ⌈D⌉)$ does not vanish, and it turns out that this condition does in fact
  suffice to guarantee finiteness of the set of $\cC$-integral points on any
  integral model of the $\cC$-pair $(X,D)$.
\end{rem}

\begin{rem}
  Theorem~\ref{maintheorem3} can be thought as a generalisation of the function
  field versions of the classical theorems of Mordell and Siegel, valid in
  arbitrary characteristic.  We want to stress that the case where $k = ℂ$ has
  already been treated by Campana \cite[§3]{Cam05}; we give an alternative
  treatment, but the real novelties lie in positive characteristic.  Campana
  also conjectured such a statement for number fields, see \cite[§4]{Cam05}.
  This conjecture is wide open, but it is known to follow from the $abc$
  conjecture thanks to observations of Colliot-Thélène and Abramovich (see the
  exposition in \cite[Appendix]{Smeets}).
\end{rem}

\subsection{Height bounds}
\approvals{Arne & yes \\ Jorge & yes \\ Stefan & yes}

Like other proofs of Mordell-type theorems, the proof of
Theorem~\ref{maintheorem3}, relies on \emph{height bounds} and on
\emph{rigidity} results for $\cC$-integral points; these are formulated in
Theorems~\ref{thm:height1a} and \ref{thm:rigidity}, respectively.  Establishing
the relevant height bounds is the main difficulty of this paper.  The proof
combines ideas pioneered by Grauert, Vojta and Kim \cite{MR0222087, MR1096125,
  MR1436743} that often allow one to restrict one's attention to integral points
that are tangent to suitable foliations, with ideas of Campana-Păun \cite{CP15}
who construct foliations on suitable ``adapted'' covers of the original space
$X$.

Two main difficulties arise and need to be overcome.  To begin, we need to
construct and discuss covering spaces and foliations in positive characteristic,
where covers might well be inseparable or wildly ramified, and where the
discussion of foliations becomes substantially more involved when compared with
the characteristic zero case.  Next, the comparison of height functions on $X$
and heights on an adapted cover turns out to be a second major issue.  One of
the main technical insights of this paper, hidden in Claim~\ref{claim:6-5} of
Section~\ref{sec:pf45} is the observation that a lift of a $\cC$-integral point
from $X$ to an adapted cover is often quite singular, and that the singularities
\emph{improve} the height bounds on the cover by exactly as much as is necessary
to establish bounds on the original space $X$.  We hope that this technique
might be of interest for others.

\subsection{Acknowledgements}
\approvals{Arne & yes \\ Jorge & yes \\ Stefan & yes}

We would like to thank numerous colleagues for discussions and invaluable help.
This includes Piotr Achinger, Frédéric Campana, Jean-Louis Colliot-Thélène,
Hélène Esnault, Andrea Fanelli, David Harbater, Annette Huber-Klawitter, Shane
Kelly, Qing Liu, Andrew Obus, Zsolt Patakfalvi, Rachel Pries, Erwan Rousseau,
Takeshi Saito, Angelo Vistoli, Felipe Voloch, Joe Waldron and Liang Xiao.

We would also like to thank the anonymous referees for patience, careful reading
and helpful advice.

%
%
\svnid{$Id: 02-notation.tex 600 2021-08-23 12:06:49Z kebekus $}

\section{Notation and global assumptions}
\subversionInfo

\subsection{Global assumptions}
\label{sec:globalAss}
\approvals{Arne & yes \\ Jorge & yes \\ Stefan & yes}

Throughout, the letter $k$ will always denote an algebraically closed field of
arbitrary characteristic.

\subsection{Varieties and pairs}
\approvals{Arne & --- \\ Jorge & yes \\ Stefan & yes}

We follow notation and conventions of Hartshorne's book \cite{Ha77}.  In
particular, varieties are always assumed to be irreducible.

\begin{defn}[Curves and surfaces]
  Let $k$ be an algebraically closed field.  A $k$-curve is a quasi-projective
  $k$-variety of dimension one that is smooth over $\Spec k$.  Analogously for
  $k$-surfaces.
\end{defn}

\begin{defn}[Pairs]\label{def:pair1}
  Let $k$ be an algebraically closed field.  A \emph{$k$-pair} is a pair
  $(X,D)$, consisting of a normal, quasi-projective $k$-variety $X$ and a Weil
  $ℚ$-divisor $D = δ_1·D_1 + ⋯ + δ_d·D_d$ on $X$, with coefficients $δ_i$ in the
  set $[0,1] ∩ ℚ$.
\end{defn}

\begin{notation}[Round-up, round-down and fractional part]
  In the setting of Definition~\ref{def:pair1}.  We denote the round-up and
  round-down of $D$ as $⌈D⌉ = \sum_{i=1}^{d} ⌈δ_i⌉·D_i$ and
  $⌊D⌋ = \sum_{i=1}^{d} ⌊δ_i⌋·D_i$.  The fractional part of $D$ will be written
  as $\{D\} := D - ⌊D⌋$.
\end{notation}

\begin{notation}[Intersection of boundary components]\label{not:ibc}
  In the setting of Definition~\ref{def:pair1}, if $I ⊆ \{1,… , d\}$ is any
  non-empty subset, consider the scheme-theoretic intersection
  $D_I := ∩_{i∈I} \supp D_i$.  If $I$ is empty, set $D_I := X$.
\end{notation}

The notion of relatively snc divisors has been used in the literature, but its
definition has not been discussed much.  For the reader’s convenience, we
reproduce the definition given in \cite[§~3.1]{MR3084424}.

\begin{defn}[\protect{SNC morphism, relatively snc divisor, \cite[Def.~2.1]{VZ02}}]\label{defn:sncmor}
  Let $k$ be an algebraically closed field, let $(X,D)$ be a $k$-pair and let
  $φ : X → Y$ be a surjective morphism of quasi-projective $k$-varieties.  We
  say that $D$ is \emph{relatively snc}, or that $φ$ is an \emph{snc morphism of
    the $k$-pair $(X, D)$} if for any set $I$ with $D_I ≠ ∅$, all restricted
  morphisms $φ|_{D_I} : D_I → Y$ are smooth of relative dimension
  $\dim X - \dim Y - |I|$.
\end{defn}

\begin{defn}[$k$-SNC]
  Let $k$ be an algebraically closed field.  A $k$-pair $(X, D)$ is called
  \emph{snc} if the morphism to $\Spec k$ is an snc morphism.
\end{defn}

\subsection{$\cC$-pairs}
\label{ssec:cpairs}
\approvals{Arne & yes \\ Jorge & yes \\ Stefan & yes}

The core notion of this paper is that of a ``$\cC$-pair''.  These pairs were
introduced under the name ``orbifoldes géométriques'' by Campana and feature
prominently in his work.  We briefly recall the main definition and refer the
reader to one of the many survey papers, including \cite{Cam04, Abramovich,
  CKT16}, for a more detailed introduction to Campana's ideas.

\begin{defn}[Pairs and $\cC$-pairs]\label{def:pair}
  Let $(X, D)$ be a $k$-pair, as in Definition~\ref{def:pair1}.  The pair
  $(X,D)$ is called a \emph{$\cC$-pair} if the coefficients $δ_i$ are contained
  in the set
  $$
  \bigl\{ {\textstyle 1- \frac{1}{m}} \, \bigl|\, m ∈ ℕ^+ \bigr\} ∪ \{ 1 \}.
  $$
  We follow the convention that $1 - \frac{1}{∞} = 1$ and write
  $δ_i = 1-\frac{1}{m_i}$ with $m_i ∈ ℕ^+ ∪ \{∞\}$.  We refer to the numbers
  $m_i$ as \emph{$\cC$-multiplicities}.
\end{defn}

\begin{defn}[$\cC$-curve]\label{def:ccurve}
  In the setting of Definition~\ref{def:pair}, assume that the $k$-pair $(X, D)$
  is snc.  A $\cC$-curve is a $k$-curve $T$ and a morphism of $k$-varieties
  $γ : T → X$ such that $\Image γ ⊄ \supp D$ and such that the following
  conditions hold for every index $i$.
  \begin{enumerate}
  \item If $m_i = ∞$, then $γ^* D_i = 0$.
  \item If $m_i < ∞$, then $γ^* D_i ≥ m_i·\supp γ^* D_i$.
  \end{enumerate}
\end{defn}

\begin{rem}
  Roughly speaking, $\cC$-curves avoid all integral components of $D$.  At
  points of intersection with one of the remaining $D_i$, the local intersection
  of every branch of $T$ with $D_i$ is at least $m_i$.
\end{rem}

\subsection{Projectivised bundles}
\approvals{Arne & --- \\ Jorge & yes \\ Stefan & yes}

Like earlier papers on the subject, we follow Grauert's approach,
\cite{MR0222087}, and study curves on $X$ by looking at their natural lifts to
the projectivised bundle $ℙ(Ω¹_X)$.  The following setting fixes assumptions and
notation.

\begin{setting}[Projectivized vector bundles]\label{setting:pbdl}
  Let $k$ be an algebraically closed field.  Given a smooth $k$-variety $X$ and
  a locally free sheaf $ℰ$ of $𝒪_X$-modules, consider the
  projectivisation\footnote{We follow Grothendieck's convention, where
    $ℙ_X(ℰ) := \operatorname{Proj} \bigl( ⊕_d \Sym^d ℰ \bigr)$.} $ℙ := ℙ_X(ℰ)$
  together with the natural projection morphism $π: ℙ → X$ and the Euler
  sequence
  \begin{equation}\label{eq:euler0}
    \xymatrix{ %
      0 \ar[r] & Ω¹_{ℙ/X}(1) \ar[r] & π^* ℰ \ar[r]^{τ} & 𝒪_{ℙ}(1)
      \ar[r] & 0.  %
    }
  \end{equation}
\end{setting}

The construction satisfies a number of universal properties, and the
``tautological sheaf morphism'' $τ$ appears prominently in all of them.  For the
reader's convenience, we recall three standard facts.

\begin{fact}[Universal property of projectivized bundles -- Quotients of $ℰ$]\label{f:P1}
  In Setting~\ref{setting:pbdl}, if $φ : Z → X$ is any morphism, then to give a
  morphism $Φ : Z → ℙ$ that makes following diagram commute,
  \begin{equation}\label{eq:bcmnyf0}
    \begin{gathered}
      \xymatrix{
        & ℙ \ar[d]^{π} \\
        Z \ar@/^2mm/[ur]^{Φ} \ar[r]_{φ} & X,
      }
    \end{gathered}
  \end{equation}
  is equivalent to give an invertible quotient of $φ^*ℰ$.  The relation
  between the morphisms $Φ$ and quotients is described as follows.
  \begin{enumerate}
  \item Given morphism $Φ$, then the associated quotient is given by the
    pull-back $Φ^*τ$, which maps $φ^* ℰ = Φ^*π^* ℰ$ to $Φ^* 𝒪_{ℙ}(1)$.

  \item Given quotient $q: φ^*ℰ → 𝒬$, then the associated morphism $Φ$ gives
    an isomorphism of sheaves $𝒬 ≅ Φ^* 𝒪_{ℙ}(1)$ that identifies $q$
    and $Φ^* τ$.  \qed
  \end{enumerate}
\end{fact}

\begin{fact}[Relatively ample hypersurfaces in $ℙ$ -- Subsheaves of $\Sym ℰ$]\label{f:P2}
  In Setting~\ref{setting:pbdl}, to give a relatively ample Cartier divisor
  $H ∈ \Div(ℙ)$ of relative degree $M ∈ ℕ^+$, it is equivalent to give an
  invertible subsheaf $ℒ ⊆ \Sym^M ℰ$.  The relation between subsheaves and
  divisors is described as follows.
  \begin{enumerate}
  \item Given a divisor $H ∈ \Div(ℙ)$ of relative degree $M ∈ ℕ^+$, observe that
    there exists an invertible $ℒ ∈ \Pic(X)$ such that
    $π^*ℒ ≅ 𝒪_{ℙ}(M)⊗ \sI_H ⊆ 𝒪_{ℙ}(M)$.  Push-down to $X$, in order
    to obtain the inclusion $ℒ ⊆ π_* 𝒪_{ℙ}(M) = \Sym^M ℰ$.

  \item Given an invertible $ℒ ⊆ \Sym^M ℰ$, then the following composed morphism of
    invertibles on $ℙ$ is non-trivial:
    $$
    \xymatrix{ %
      π^* ℒ \ar[r] & π^* \Sym^M ℰ \ar[r]^(.6){\Sym^M τ} & 𝒪_{ℙ}(M), %
    }
    $$
    The associated vanishing locus is a Cartier divisor $H ⊆ ℙ$ of relative
    degree $M$.  \qed
  \end{enumerate}
\end{fact}

The constructions described in Fact~\ref{f:P1} and \ref{f:P2} are of course
related.

\begin{fact}[Relation between Facts~\ref{f:P1} and \ref{f:P2}]\label{f:P3}
  In Setting~\ref{setting:pbdl}, assume we are given a diagram as in
  \eqref{eq:bcmnyf0}, with associated quotient $q: φ^*ℰ → 𝒬$, as well as a
  relatively ample Cartier divisor $H ∈ \Div(ℙ)$ of relative degree $M ∈ ℕ^+$,
  with associated subsheaf $ℒ ⊆ \Sym^M ℰ$.  Abusing notation slightly, the
  symbol $H$ will also be used to denote the associated complete intersection
  subscheme of $ℙ$.  Then, the morphism $Φ$ factors via $H ⊆ ℙ$ if and only if
  the following composed sheaf morphism vanishes identically:
  $$
  φ^* ℒ → φ^* \Sym^M ℰ = Φ^* π^* \Sym^M ℰ → Φ^* 𝒪_{ℙ}(M) ≅ 𝒬^{⊗
    M}.  \eqno\qed
  $$
\end{fact}

\part{The Mordell conjecture for integral points on orbifolds}
\label{part:orbimordell}

%
%
\svnid{$Id: 03-bounds-setup.tex 599 2021-08-23 12:06:26Z kebekus $}

\section{Main results}
\label{sec:height}
\subversionInfo
\approvals{Arne & yes \\ Jorge & yes \\ Stefan & yes}

This present section formulates the main results of Part~\ref{part:orbimordell}
of this paper: orbifold analogues of the classical theorems on boundedness and
rigidity for algebraic points over function fields, as well as an orbifold
version of the Mordell conjecture.  Proofs are given in
Sections~\ref{sec:kimsquare}--\ref{sec:orbimordell} below.

\subsection{Setup}
\label{subsec:setup}
\approvals{Arne & --- \\ Jorge & yes \\ Stefan & yes}

To begin, we specify the setup that is used throughout
Part~\ref{part:orbimordell} in some detail and fix notation.  The central
setting is that of a surface pair $(X,D)$ fibred over a curve $S$, where the
fibration is assumed to be an snc morphism, away from a finite set $Δ$ of
exceptional points in $S$.

\begin{setting}[Surface pair fibred over a curve]\label{setting:CKim}
  Let $k$ be an algebraically closed field of characteristic $p ≥ 0$.  Let
  $φ : X → S$ be a surjective morphism of smooth, projective $k$-varieties,
  where $X$ is of dimension two and $S$ is of dimension one.  Assume that $φ$
  has connected fibres.  Let $D = \sum_{i=1}^d δ_i · D_i$ be a $ℚ$-divisor on
  $X$ such that $(X,D)$ is an snc $\cC$-pair, and let $m_i ∈ ℕ ∪ \{∞\}$ denote
  the $\cC$-multiplicities of $D$.  Finally, let $S° ⊆ S$ be a dense open set
  such that the restriction $φ° := φ|_{X°}$ is an snc morphism for the pair
  $(X°,D°)$, where $X° := f^{-1}(S°)$ and $D° := D ∩ X°$.  We denote the generic
  point of $S$ by $η$ and write $X_η$ for the generic fibre.  Write $Δ := S∖ S°$
  and view $Δ$ as a reduced subscheme of $S$.
\end{setting}

\begin{notation}
  Throughout the text, we consider index sets
  \begin{align*}
    \log & := \{ i \::\: m_i = ∞\} & \fract & := \{ i \::\: m_i < ∞\} \\
    \wild & := \{ i ∈ \fract \::\: p ≠ 0 \text{ and } p|m_i\} & \tame & := \fract ∖ \wild
  \end{align*}
  and the associated reduced sub-divisors of $\supp D$,
  \begin{align*}
    D_{\log} & := ⌊D⌋, & D_{\fract} & := ⌈ D ⌉ - ⌊D⌋, \\
    D_{\wild} & := \sum_{i ∈ \wild} D_i, & D_{\tame} & := \sum_{i ∈ \tame} D_i.
  \end{align*}
  In addition, we write $\{D\} := D - ⌊D⌋$ for the fractional part of
  $D$.
\end{notation}

\begin{rem}[SNC morphism]
  The assumption that $φ°$ is an snc morphism implies that the induced morphism
  $D° → S°$ is étale.  In particular, every component of $D°$ is smooth over
  $\Spec k$.  The assumption also implies that there exists an exact sequence of
  logarithmic differentials, as follows:
  \begin{equation}\label{eq:log6-4}
    \xymatrix{ %
      0 \ar[r] & (φ°)^* Ω¹_{S°} \ar[r]^(.45){dφ°} & Ω¹_{X°} \bigl( \log ⌈ D°⌉ \bigr) \ar[r] & ω_{X°/S°} \bigl(\log ⌈ D° ⌉\bigr) \ar[r] & 0.
    }
  \end{equation}
\end{rem}

As pointed out in the introduction, we are mainly concerned with
\emph{$\cC$-integral points}.  These are parameterised curves $T → X$ that are
not contained in $\supp D$, that dominate $S$, and that intersect the components
$D_i$ of $D$ locally with $\cC$-multiplicity $m_i$ or more, at least away from
the exceptional set $Δ ⊂ S$.  Figure~\vref{fig:xcc} illustrates the somewhat
technical definition.
\begin{figure}
  \begin{tikzpicture}
    \draw plot [smooth] coordinates {(2.5-7.0, 2.0) (0.8-7.0, 0.8) (-0.8-7.0,1) (-1.3-7.0,-0.0) (-0.8-7.0,-1) (0.8-7.0, -0.8) (2.5-7.0, -2.0)};
    \draw (2.2-7.0, 2.0) node[left] {$T$};

    \draw [->] (-4.5, 0.0) -- node[above]{$γ$} (-2.5, 0);

    \fill[fill=gray!20!white] (-2.0, 1.5) node[above] {$X$} rectangle (2.5, -1.5);
    \draw[dotted] (-2.0, -1.5) rectangle (2.5, 1.5);

    \draw[color=gray] plot [smooth, tension=1] coordinates {(-1.5, 2.0) (-1.3,0.0) (-1.5,-2.0)};
    \draw[color=gray] plot [smooth, tension=1] coordinates {(-0.5, 2.0) (-0.2,0.0) (-0.5,-2.0)};
    \draw[color=gray] plot [smooth, tension=1] coordinates {( 0.5, 2.0) ( 0.8,0.0) ( 0.5,-2.0)};

    \draw[color=gray] (1.5, 2) to [out=-90, in=110] (2,-0.5);
    \draw[color=gray] (1.5, -2) to [out=90, in=-110] (2,0.5);

    \draw[dashed] plot [smooth, tension=1] coordinates {(-2.5, -1.0) (-1.3,0.0) (0.8, 0.79) (3, 1.0)} node[right] {$D_1$};
    \draw[dashed] plot [smooth, tension=1] coordinates {(-2.5, 1.0) (-1.3,0.0) (0.8,-0.79) (3,-1.0)} node[right] {$D_2$};

    \draw plot [smooth] coordinates {(2.5, 2.0) (0.8, 0.8) (-0.8,1) (-1.3,-0.0) (-0.8,-1) (0.8, -0.8) (2.5, -2.0)} node[right] {$Σ_T$};

    \draw [->] (0, -2) -- node[right]{$φ$} (0, -3.5);

    \draw (-2, -4.0) node[above]{$S$} -- (2.5, -4.0);
    \filldraw (-1.5, -4.0) circle(0.05) node[below]{$Δ_1$};
    \filldraw ( 1.5, -4.0) circle(0.05) node[below]{$Δ_2$};
  \end{tikzpicture}

  \parbox{\textwidth}{\footnotesize The figure shows a $\cC$-integral point in
    Setting~\ref{setting:CKim}, with a boundary divisor of the form
    $D = \frac{1}{2}·(D_1+D_2)$.  The set $Δ ⊂ S$ contains two points where $φ$
    fails to be an snc morphism.  In applications, the set $Δ$ might also
    contain points where $φ$ \emph{is} an snc morphism, but where $γ$ fails to
    be a $\cC$ curve.}

  \caption{$\cC$-integral points}
  \label{fig:xcc}
\end{figure}

\begin{defn}[$\cC$-integral points]\label{def:orbicurve}
  In Setting~\ref{setting:CKim}, a \emph{$\cC$-integral point} is a morphism
  $γ : T → X$, where $T$ is a smooth, projective $k$-curve satisfying the
  following conditions.
  \begin{enumerate}
  \item\label{il:3-3-1} The curve $T$ dominates $S$.  In particular,
    $T° := γ^{-1}(X°)$ is not empty.
  \item\label{il:3-3-2} The restriction $γ° : T° → X°$ is a $\cC$-curve for the
    pair $(X°, D°)$.
  \item\label{il:3-3-3} The induced morphism $γ : T → \Image(γ)$ is birational.
  \end{enumerate}
  If the divisor $D$ is empty, we refer to $\cC$-integral points simply as
  \emph{algebraic points}.
\end{defn}

\begin{rem}[Discussion of $\cC$-integral points]
  Item~\ref{il:3-3-2} implies in particular that $Σ_T := \Image(γ)$, which is a
  one-dimensional $k$-subvariety of $X$, is not contained in $\supp D$.
  Item~\ref{il:3-3-3} implies that the pull-back map of differentials,
  $dγ : γ^* Ω¹_X → ω_T$ is not the zero map.  The following diagrams summarise
  the objects and morphisms of Definition~\ref{def:orbicurve}.
  $$
  \begin{gathered}
    \xymatrix{ %
      T \ar@/_2mm/[rd] \ar[r]^(.3){γ} & Σ_T ⊂ X \hphantom{Σ_T ⊂ } \ar[d]^{φ} \\
      & S }
  \end{gathered}
  \quad\text{and}\qquad
  \begin{gathered}
    \xymatrix{ %
      T° \ar[rrr]^{γ°,\text{ $\cC$-curve for }(X°, D°)} \ar@/_2mm/[drrr] &&& X° \ar[d]^{φ°\text{ snc for }(X°, D°)} \\
      &&& S°.  %
    }
  \end{gathered}
  $$
\end{rem}

As in the classical setting, the two main invariants associated with
$\cC$-integral points are its \emph{height} and its \emph{discriminant}.  The
definitions below are identical to those found in the literature.  They need no
adjustment to work in the orbifold case.

\begin{defn}[Height and discriminant]\label{def:height}
  In Setting~\ref{setting:CKim}, let $B$ be any $ℚ$-divisor on $X$.  Given a
  $\cC$-integral point $γ : T → X$ as in Definition~\ref{def:orbicurve}, define
  the \emph{height} $h_B(γ)$ and the \emph{discriminant} $δ(γ)$ as
  $$
  h_B(γ) := \frac{\deg γ^* B}{[T : S]} \quad\text{and}\quad δ(γ) := \frac{\deg ω_T}{[T : S]}.
  $$
  If $ℬ$ is an invertible sheaf on $X$, we define $h_{ℬ}(γ)$ in the obvious
  fashion.  We simply write $h(γ)$ for the somewhat lengthy symbol
  $h_{K_{X/S}+D}(γ)$.  In order to avoid awkward case-by-case definition, we
  define $h_{ℬ}(γ) = 0$ for all $γ$ if $ℬ$ is the zero-sheaf.
\end{defn}

\subsection{Geometric height inequalities for $\cC$-integral points}
\approvals{Arne & yes \\ Jorge & yes \\ Stefan & yes}
\label{ssec:height}

The following geometric height inequality for $\cC$-integral points is the main
result of Part~\ref{part:orbimordell} in this paper.

\begin{thm}[Geometric height inequalities for $\cC$-integral points]\label{thm:height1a}
  In Setting~\ref{setting:CKim}, assume the following.
  \begin{enumerate}
  \item\label{il:3-6-1} The degree $d := \deg_{X_η} (K_X+D)$ is strictly
    positive.

  \item\label{il:3-6-2} Sequence~\eqref{eq:log6-4} does not split when
    restricted to the generic fibre $X_η$.
  \end{enumerate}
  Then, a height inequality of the following form holds for all $\cC$-integral
  points $γ$:
  $$
  h(γ) ≤ \const^+·\,δ(γ) + O\left( \textstyle{\sqrt{h(γ)}} \right).
  $$
\end{thm}

Theorem~\ref{thm:height1a} is in fact a simple corollary of the following more
precise result, which generalises works of Vojta, Kim and others.  A proof is
given in Sections~\ref{sec:kimsquare}--\ref{sec:pf45} below.

\begin{thm}[Geometric height inequalities for $\cC$-integral points]\label{thm:height1b}
  In the setting of Theorem~\ref{thm:height1a}, assume that the characteristic
  $p$ is zero, or that none of the finite $\cC$-multiplicities is a multiple of
  $p²$.  Write $d' := d·\lcm\, \{ m_i \,|\, m_i \ne ∞ \}$.  Then, given any
  number $ε ∈ ℚ^+$, a height inequality of the following form holds for all
  $\cC$-integral points $γ$:
  $$
  h(γ) ≤ \max\{d', \, 2+ε\}·δ(γ) + O\left( \textstyle{\sqrt{h(γ)}} \right).
  $$
\end{thm}

\subsubsection{Explanation of Assumption~\ref*{il:3-6-2}}
\label{ssec:explanation}
\approvals{Arne & yes \\ Jorge & yes \\ Stefan & yes}

In case where $\operatorname{char}(k) \ne 0$ and $D \ne 0$,
Assumption~\ref{il:3-6-2} can be interpreted in terms of the field of definition
for the affine curve $C_η := (X∖ \supp D)_η$ over $\Spec k(η)$, as follows.  The
affine curve $C_η$ is defined over $k(η)^{\operatorname{char}(k)}$ if and only
if the exact sequence \eqref{eq:log6-4} splits, see \cite[Lem.~1]{MR1106753}.
In particular, if $X = S ⨯ ℙ¹$ and $\supp D$ is the union of sections, then the
exact sequence \eqref{eq:log6-4} splits if and only if the cross ratios of any
$4$ sections in $D$ lie in $k(η)^{\operatorname{char}(k)}$, see
\cite[Sect.~7]{MR1815399}.

\subsubsection{Necessity of Assumption~\ref*{il:3-6-2}}
\label{ssec:necessity height}
\approvals{Arne & --- \\ Jorge & yes \\ Stefan & yes}

In case where $\operatorname{char}(k) \ne 0$, set $X := S ⨯ ℙ¹$ and let
$\supp D$ be a union of graphs of inseparable morphisms $S → ℙ¹$, taken with
$\cC$-multiplicities that are less than the characteristic.
Sequence~\eqref{eq:log6-4} will then split on $X_η$.  In this case, no matter
whether $\deg_{X_η} (K_X+D)$ is positive or not, the graph of every inseparable
morphism $S → ℙ¹$ is either contained in $\supp D$, or is a $\cC$-integral
point.  Therefore, without Assumption~\ref*{il:3-6-2}, the height of a
$\cC$-integral point cannot be bounded in terms of the discriminant.  At the
same time, we do not claim that our assumptions are optimal.  It is an
interesting problem to determine optimal assumptions in order to guarantee that
a height inequality of the form presented in Theorem~\ref{thm:height1a} holds
true.

\subsubsection{Improved height bounds in characteristic zero}
\approvals{Arne & --- \\ Jorge & yes \\ Stefan & yes}
\label{ssec:impbhcz}
\CounterStep

This paper is mainly concerned with height bounds for $\cC$-integral points over
fields of finite characteristic.  Still, we would like to remark that if
$\operatorname{char}(k) = 0$, then the height bound of
Theorem~\ref{thm:height1b} can easily be improved to
\begin{equation}\label{eq:dfgbns}
  h(γ) ≤ (2+ε)·δ(γ) + O\left( \textstyle{\sqrt{h(γ)}} \right).
\end{equation}
To keep the paper readable, we chose to not discuss the characteristic at every
single step of the proof.  Instead, we refer the reader to
Sections~\ref{ssec:notksq} and \ref{sssec:dfg} where the improvements in case
$\operatorname{char}(k) = 0$ are briefly explained.

\subsubsection{Sharpness of Theorem~\ref*{thm:height1b} and relation to earlier results in case $\operatorname{char}(k) = 0$}
\approvals{Arne & yes \\ Jorge & yes \\ Stefan & yes}

In the classical setting where $D=0$ and $\operatorname{char}(k) = 0$, the
height bound \eqref{eq:dfgbns} is due to Vojta, \cite[Thm.~0.2]{MR1096125}, and
is known to be nonoptimal.  Independent works by McQuillan and Yamanoi that are
specific to characteristic zero, \cite{MR3095099, MR2096455}, allow to replace
the constant $2+ε$ by $1+ε$.  This was previously conjectured by Vojta, see
\cite[Conj.~0.1]{MR1096125} and the survey \cite{MR2605324}.  If desired, the
results of McQuillan and Yamanoi may be applied to further improve
Theorem~\ref{thm:height1b} in case where $\operatorname{char}(k) = 0$.

\subsubsection{Sharpness of Theorem~\ref*{thm:height1b} and relation to earlier results in case $\operatorname{char}(k) \ne 0$}
\approvals{Arne & yes \\ Jorge & yes \\ Stefan & yes}
\label{sssec:B}

The situation in positive characteristic is less well understood.  In case where
$D=0$ and $\operatorname{char}(k) \ne 0$, the height bound of
Theorem~\ref{thm:height1b} is due to Kim, \cite[Thm.~1]{MR1436743}.  Voloch has
shown by way of example that one cannot hope for a better constant than
$\max\{d, \, 2+ε\}$ without further assumptions, \cite[p.~45--46]{MR1436743}.
Nevertheless, more stringent constraints on the behaviour of the Kodaira-Spencer
maps do yield better bounds, see \cite[Thm.~2]{MR1436743} and
\cite[Claim~2.2]{MR1815399}.  In the orbifold setting and in positive
characteristic, it is presently unclear to us if stronger assumptions on the
Kodaira-Spencer of the pair $(X°, ⌈D°⌉)$ might lead to a better constant.  We
will return to the subject in Section~\ref{ssec:doBetter}.

It is conceivable that the assumption ``none of the $\cC$-multiplicities $m_i$
is a multiple of $p²$'' is not necessary and that a more general statement can
be proven if one is willing to replace the Artin-Schreier covers that we use in
Section~\ref{sec:cover} by more complicated Artin-Schreier-Witt covers,
cf.~Remark~\vref{rem:5-4}.  Since this paper is rather long already and since
Theorem~\ref{thm:height1b} suffices for our applications, we have chosen to
avoid a detailed analysis of the algebra and combinatorics involved and to leave
this problem until later.

\subsection{Rigidity theorem for $\cC$-integral points}
\approvals{Arne & --- \\ Jorge & yes \\ Stefan & yes}

As a second step towards our Mordell-type theorem for $\cC$-integral points, we show
that $\cC$-integral points do not deform.  The following theorem makes this
statement precise.

\begin{thm}[Rigidity theorem for $\cC$-integral points]\label{thm:rigidity}
  In Setting~\ref{setting:CKim}, assume the following.
  \begin{enumerate}
  \item\label{il:3-6-1a} The degree $d := \deg_{X_η} (K_X+D)$ is strictly
    positive.

  \item\label{il:3-6-2a} Sequence~\eqref{eq:log6-4} does not split when
    restricted to the generic fibre $X_η$.
  \end{enumerate}
  If \, $T$ is any smooth, projective $k$-curve over $S$, if $H$ is any
  $k$-variety and if
  \[
    γ : T ⨯ H → X
  \]
  is any family of $S$-morphisms whose individual members
  $(γ_h : T → X)_{h ∈ H(k)}$ are $\cC$-integral points, then $γ$ is constant.
\end{thm}

\begin{rem}
  The assumption that $T$ is a $k$-curve over $S$ means that $T$ comes equipped
  with a surjective morphism $ζ : T → S$.  The assumption that $γ$ is a family
  of $S$-morphisms means that $φ◦γ_h = ζ$, for every $h ∈ H(k)$.  These
  assumptions are essential for the rigidity as, otherwise, finiteness would be
  impossible for curves with infinite automorphism group.  The conclusion ``$γ$
  is constant'' asserts that $γ_{h_1} = γ_{h_2}$, for all $h_1$, $h_2 ∈ H(k)$.
\end{rem}

\subsection{The Mordell conjecture for $\cC$-integral points}
\approvals{Arne & yes \\ Jorge & yes \\ Stefan & yes}

As a fairly immediate consequence of height bounds and of rigidity, we obtain
the following Mordell-type theorem, asserting the finiteness of $\cC$-integral
points.

\begin{thm}[Mordell-type theorem for $\cC$-integral points]\label{thm:Mordell}
  In Setting~\ref{setting:CKim}, assume the following.
  \begin{enumerate}
  \item\label{il:3-6-1b} The degree $d := \deg_{X_η} (K_X+D)$ is strictly
    positive.

  \item\label{il:3-6-2b} Sequence~\eqref{eq:log6-4} does not split when
    restricted to the generic fibre $X_η$.
  \end{enumerate}
  If $T$ is any smooth, projective $k$-curve over $S$, then the number of
  $\cC$-integral points $γ : T → X$ over $S$ is finite.
\end{thm}

Theorem~\ref{thm:Mordell} will be shown in Section~\ref{sec:orbimordell},
starting from Page~\pageref{sec:orbimordell} below.

\subsubsection{Necessity of Assumption ~\ref*{il:3-6-2b}}
\label{ssec:necessity mordell}
\approvals{Arne & yes \\ Jorge & yes \\ Stefan & yes}

The discussion in Section~\ref{ssec:necessity height} shows that without
Assumption~\ref*{il:3-6-2b}, finiteness of $\cC$-integral points does not hold
in general.

\subsubsection{Earlier results}
\approvals{Arne & yes \\ Jorge & yes \\ Stefan & yes}

In the classical situation, where $D = 0$, Theorem~\ref{thm:Mordell} was proven
in characteristic zero by Grauert and Manin independently, \cite{MR0157971,
  MR0222087}.  See also \cite{MR1096426}, which translates Manin's proof to
modern language and fixes a gap in it.  In positive characteristic, but still in
the classical situation, the Theorem is due to Samuel \cite{MR0204430,
  MR0222088}.  An alternative approach in characteristic zero was laid down by
Parshin in \cite{Parshin68}, and was later generalised to arbitrary
characteristic by Szpiro \cite[Sect.~8, Cor.~1]{Szpiro81}.  In positive
characteristic, there is yet another approach, due to Voloch \cite{MR1106753}.
For orbifolds in characteristic zero, Theorem \ref{thm:Mordell} is due to
Campana \cite[Thm.~3.8]{Cam05}.  To the best of our knowledge,
Theorem~\ref{thm:Mordell} is new for orbifolds in positive characteristic.

%
%
\svnid{$Id: 04-bounds-kim.tex 599 2021-08-23 12:06:26Z kebekus $}

\section{Geometric height bounds in a generalised logarithmic setting}
\subversionInfo
\label{sec:kimsquare}
\approvals{Arne & yes \\ Jorge & yes \\ Stefan & yes}

Theorem~\ref{thm:kimsquare} below is the technical core of
Part~\ref{part:orbimordell} of this paper.  It generalises the height bounds
found by Kim, \cite{MR1436743} and will be used in Section~\ref{sec:pf45} to
prove the height inequalities that were stated in Section~\ref{ssec:height}
above.  More specifically, we establish height bounds for $\cC$-integral points
on $(X,D)$ by applying the following Theorem~\ref{thm:kimsquare} to a suitable
adapted cover of $X$ and taking the sheaf of ``adapted differentials'' for $𝒜$.
For clarity of exposition, we specify our precise setting first.

\begin{setting}\label{set:4-1}
  In Setting~\ref{setting:CKim}, assume that the divisor $D$ is reduced.  In
  other words, assume that all coefficients $δ_i$ are equal to $1$.  Also,
  assume that there exists a sequence of inclusions $φ^*Ω¹_S ⊊ 𝒜 ⊆ Ω¹_X(\log D)$
  where $𝒜$ is locally free of rank two and where $\factor{𝒜}{φ^*Ω¹_S}$ is
  invertible over $X°$.  Setting $ℬ := \bigl( \factor{𝒜}{φ^*Ω¹_S} \bigr)^{**}$,
  we obtain a complex of sheaves on $X$,
  $$
  \xymatrix{
    φ^*Ω¹_S \ar@{^(->}[r]^(.55){dφ} & 𝒜 \ar[r] & ℬ,
    }
  $$
  and an exact sequence of locally free sheaves on $X°$,
  \begin{equation}\label{eq:x1}
    \xymatrix{%
      0 \ar[r] & (φ°)^*Ω¹_{S°} \ar[r]^(.55){dφ°} & 𝒜|_{X°} \ar[r] & ℬ|_{X°} \ar[r] & 0.
    }
  \end{equation}
  We follow the notation introduced in Definition~\ref{def:orbicurve} concerning
  $\cC$-integral points $γ : T → X$ and write $dγ_𝒜$ for the composed morphism,
  \begin{equation}\label{eq:dgA}
    γ^*𝒜 → γ^*Ω¹_X(\log D) \xrightarrow{dγ} ω_T\bigl(\log (γ^*D)_{\red}\bigr).
  \end{equation}
\end{setting}

\begin{thm}[Height bounds for subsheaves of $ω_{X/S}$]\label{thm:kimsquare}
  In Setting~\ref{set:4-1}, assume that the degree $d := \deg_{X_η} ℬ$ is
  positive and that Sequence~\eqref{eq:x1} does not split when restricted to
  $X_η$.  Then, given any number $ε ∈ ℚ^+$, a height inequality of the following
  form will hold for all $\cC$-integral points $γ : T → X$,
  \begin{equation}\label{eq:hI2}
    h_ℬ (γ) ≤ \max \{ d, 2+ε \}·\frac{\deg_T \Image dγ_𝒜}{[T:S]}
    + O \left( \textstyle{\sqrt{h_ℬ (γ)}}\right).
  \end{equation}
\end{thm}

\begin{notation}
  To avoid overly verbose notation, we write $ω_T(\log γ^*D)$ instead of the
  more correct $ω_T\bigl(\log (γ^*D)_{\red}\bigr)$ throughout the rest of the
  present Section~\ref{sec:kimsquare}.
\end{notation}

\begin{rem}[Relation to Theorem~\ref{thm:height1b}]
  No matter what sheaf $𝒜$ is chosen, the assumption that $γ°$ is a $\cC$-curve,
  as spelled out in \ref{il:3-3-2}, implies that $\supp γ^*D$ lies over
  $Δ = S ∖ S°$.  In particular, $\deg_T (γ^*D)_{\red} ≤ [T:S]·\#Δ$ and
  $$
  \frac{\deg_T \Image dγ_𝒜}{[T:S]} ≤ \frac{\deg_T ω_T(\log γ^*D)}{[T:S]} ≤ δ(γ) +
  \#Δ
  $$
  Theorem~\ref{thm:kimsquare} therefore implies Theorem~\ref{thm:height1b} in
  the special case where $D$ is reduced.
\end{rem}

\begin{rem}[Relation to results in the literature]
  In the special case where $D = 0$ and $𝒜 = Ω¹_X$, Theorem~\ref{thm:kimsquare}
  replaces the number $\deg_T ω_T$ which appears in Kim's work \cite[Thms.~1 and
  2]{MR1436743} by $\deg_T \Image(dγ_𝒜)$, which gets smaller in comparison to
  $\deg_T ω_T$ the more singular the curve $Σ_T$ is.  This improvement will be
  of critical importance for our applications.  In fact, the whole proof of
  Theorem~\ref{thm:height1b} hinges on this observation.
\end{rem}

The proof of Theorem~\ref{thm:kimsquare} follows ideas of \cite{MR1436743,
  MR1757883} and extends them to our generalised setting.  We present a detailed
and self-contained proof in the rest of Section~\ref{sec:kimsquare}.

\subsection{Preparation for the proof: elementary fact}
\label{sec:1a}
\approvals{Arne & --- \\ Jorge & yes \\ Stefan & yes}

The proof of Theorem~\ref{thm:kimsquare} considers the following situation more
than once.

\begin{situation}\label{sit:1a}
  Let $k$ be an algebraically closed field.  Let $ψ : Z → B$ be a surjective
  morphism with connected fibres between smooth, projective $k$-varieties, where $Z$
  is of dimension two, and $B$ of dimension one.  We will also consider a chain
  $0 ⊊ ℱ ⊊ 𝒜$ of coherent sheaves on $Z$, where $ℱ$ is invertible, where $𝒜$
  is locally free of rank two, and where the quotient by $𝒬 := \factor{𝒜}{ℱ}$ is
  torsion free.
\end{situation}

\begin{reminder}[Saturated subsheaf]\label{remi:sat}
  The assumption that $𝒬$ is torsion free is often formulated by saying that
  \emph{$ℱ$ is a saturated subsheaf of $𝒜$}.  We will use this language in the
  sequel and refer to \cite[Sect.~1.1]{MR2665168} for a more detailed discussion
  of saturation.
\end{reminder}

\begin{lem}\label{lem:1a}
  In Situation~\ref{sit:1a}, there exists an invertible sheaf $𝒢 ⊆ 𝒬$ such that
  the support of the quotient $\factor{𝒬}{𝒢}$ does not dominate $B$.  If $Z_η$
  denotes the generic $ψ$-fibre, this implies in particular that
  $𝒢|_{Z_η} ≅ 𝒬|_{Z_η}$.
\end{lem}
\begin{proof}
  The quotient $𝒬$ is torsion free and of rank one, but need not be locally
  free.  However, since $Z$ is regular, its reflexive hull $𝒬^{**}$ will be
  invertible, \cite[Cor.~1.4 or Prop.~1.9]{MR597077}.  Also, recall from
  \cite[Ex.~1.1.16]{MR2665168}, that there exists a dimension-zero (and hence
  finite) subscheme $A ⊊ Z$ such that the natural inclusion $𝒬 ↪ 𝒬^{**}$
  identifies $𝒬$ with $𝒥_A ⊗ 𝒬^{**}$.  Consider the image $C := ψ(A)$, equipped
  with its natural structure as a reduced subscheme of $B$.  If $N ≫ 0$, then
  $𝒥_{ψ^{-1}(C)}^{⊗ N} ⊊ 𝒥_A$, and we obtain the following chain of inclusions
  which finishes the proof:
  $$
  𝒢 := 𝒥_{ψ^{-1}(C)}^{⊗ N} ⊗ 𝒬^{**} ⊊ 𝒬 ⊆ 𝒬^{**}.  \eqno \qedhere
  $$
\end{proof}

Lemma~\ref{lem:1a} will be used in the following form.

\begin{obs}\label{obs:l1.2}
  In Situation~\ref{sit:1a}, let $𝒢 ⊆ 𝒬$ be one of the subsheaves given by
  Lemma~\ref{lem:1a}.  If $γ : T → Z$ is a morphism from a smooth, projective
  $k$-curve $T$ such that the induced map $T → B$ is dominant, if $ℒ ∈ \Pic(T)$
  and if $ρ : γ^*𝒜 → ℒ$ is any non-trivial morphism such that the composition
  $γ^*ℱ → γ^*𝒜 → ℒ$ vanishes, then there exists a non-trivial morphism
  $γ^*𝒢 → \Image(ρ)$, and we obtain a diagram as follows.
  $$
  \xymatrix@C=1.2cm{
    &&& γ^*ℱ \ar[d] \ar@/^3mm/[drr]^{\qquad\txt{\scriptsize composition vanishes\\\scriptsize by assumption}} \\
    &&& γ^*𝒜 \ar[rr]^{ρ} \ar@{->>}[d] && \Image(ρ) ⊆ ℒ \\
    γ^*𝒢 \ar[rrr]^{\text{generically an isomorphism}}_{\text{since $ψ◦ γ$ is dominant}} &&& γ^*𝒬 \ar@{..>}@/_3mm/[urr]_{\quad\text{map exists}} }
  $$
\end{obs}

\subsection{Preparation for the proof: height bounds for curves tangent to a foliation}
\label{ssec:x1}
\approvals{Arne & --- \\ Jorge & yes \\ Stefan & yes}

To prepare for the proof of Theorem~\ref{thm:kimsquare}, we will first prove a
height bound for $\cC$-integral points $γ$ tangent to a given foliation, or
which satisfy a ``Pfaffian equation'' in the language of \cite{MR1436743}.
While such bounds are easy to obtain in characteristic zero, where tangent
curves always form a bounded family, the proof below is somewhat more involved.
The following observation will be used.

\begin{obs}\label{obs:g3xa}
  In the setting of Theorem~\ref{thm:kimsquare}, assume that $k$ is
  algebraically closed.  If $γ : T → X$ is a $\cC$-integral point, then $dγ_𝒜$
  \emph{can} vanish identically, but only if the image curve $γ(T)$ is contained
  in the support of $Ω¹_X(\log D)/𝒜$.  To see why, recall from
  Item~\ref{il:3-3-3} of Definition~\ref{def:orbicurve} that $γ$ maps $T$
  birationally onto its image.  The differential
  \[
    γ^*Ω¹_X(\log D) \xrightarrow{dγ} ω_T\bigl(\log (γ^*D)_{\red}\bigr)
  \]
  will therefore \emph{not} vanish.  The composed map $dγ_𝒜$ of \eqref{eq:dgA}
  will then likewise \emph{not} vanish if $γ(T)$ intersects the open set of $X$
  where the sheaves $𝒜$ and $Ω¹_X(\log D)$ agree.  If $dγ_𝒜$ does vanish, then
  the image $γ(T)$ must therefore be contained in the closed set where $𝒜$ and
  $Ω¹_X(\log D)$ differ; this is however exactly the support of the quotient
  $Ω¹_X(\log D)/𝒜$.

  It follows that there are at most finitely many $\cC$-integral points $γ$
  where $dγ_𝒜 = 0$, up to reparametrisation of the morphism $γ : T → X$.  So,
  there exists a number $\const_1 ∈ ℝ$ such that $h_ℬ(γ) ≤ \const_1$ for all
  $\cC$-integral points $γ$ with $dγ_𝒜 = 0$.  \qed
\end{obs}

\begin{notation}[Foliation]
  Let us adopt the setting of Theorem~\ref{thm:kimsquare}.  By an
  \emph{$𝒜$-foliation} on $X$, we mean a saturated, invertible subsheaf $ℱ ⊆ 𝒜$.
  A $\cC$-integral point $γ$ is said to be \emph{tangent to the $𝒜$-foliation
    $ℱ$} if the following composed map $dγ_ℱ$ vanishes identically:
  $$
  γ^*ℱ → γ^*𝒜 \xrightarrow{dγ_𝒜} ω_T(\log γ^*D).
  $$
\end{notation}

\begin{prop}[\protect{Height bound for curves tangent to foliation, compare \cite[p.~53]{MR1436743}}]\label{prop:hbfol}
  If, in the setting of Theorem~\ref{thm:kimsquare}, $ℱ ⊆ 𝒜$ is an $𝒜$-foliation
  on $X$, then a height inequality of the following form will hold for all
  $\cC$-integral points $γ : T → X$ that are tangent to $ℱ$:
  $$
  h_ℬ (γ) ≤ d·\frac{\deg_T \Image dγ_𝒜}{[T:S]} + O \left( \textstyle{\sqrt{h_ℬ
        (γ)}}\right).
  $$
\end{prop}
\begin{proof}
  The generic fibre $X_η$ is a smooth $k(η)$-curve, the sheaf $𝒜|_{X_η}$ is a
  vector bundle on that curve, and Sequence~\eqref{eq:x1} presents $𝒜|_{X_η}$ as
  an extension of the trivial line bundle $φ^*Ω¹_{S°}|_{X_η} ≅ 𝒪_{X_η}$ and the
  ample line bundle $ℬ|_{X_η}$.  This extension is not split by assumption.
  Therefore, any invertible quotient of $𝒜|_{X_η}$ will be of positive degree.

  To apply this observation, consider the cokernel $𝒬 := \factor{𝒜}{ℱ}$.  Let
  $𝒢 ⊆ 𝒬$ be the invertible subsheaf given by Lemma~\ref{lem:1a}, so that
  $𝒢|_{X_η} = 𝒬|_{X_η}$.  The sheaf $𝒢|_{X_η}$ is then an invertible quotient of
  $𝒜|_{X_η}$, so that $\deg_{X_η} 𝒢 > 0$.  Néron's theorem,
  \cite[Thm.~2.11]{MR1002324}, now yields the following inequality for all $γ$:
  \begin{equation}\label{eq:keyinequality}
    h_ℬ(γ) ≤ \underbrace{\bigl(\deg_{X_η} 𝒢\bigr)}_{> 0}\,·\,h_ℬ(γ) + O(1)
    \overset{\text{Néron}}{≤} \underbrace{\bigl(\deg_{X_η} ℬ\bigr)}_{=d}\,·\,h_{𝒢}(γ) + O
    \left( \textstyle{\sqrt{h_ℬ(γ)}} \right).
  \end{equation}
  To end the proof, Observation~\ref{obs:g3xa} implies that it suffices to
  establish the inequality $\deg_T γ^*𝒢 ≤ \deg_T \Image(dγ_𝒜)$ for all $γ$ whose
  associated morphism $dγ_𝒜$ does not vanish identically.  This, however,
  follows because Observation~\ref{obs:l1.2} yields the existence of a
  non-trivial morphism $γ^*𝒢 → \Image(dγ_𝒜) ⊆ ω_T(\log γ^*D)$.
\end{proof}

\subsection{Proof of Theorem~\ref*{thm:kimsquare}}
\label{sec:kimsquarepf}
\approvals{Arne & yes \\ Jorge & yes \\ Stefan & yes}

In this paragraph we adopt the setting of Theorem~\ref{thm:kimsquare} and the
notation introduced in Section~\ref{ssec:x1} above, and we assume that a number
$ε ∈ ℚ^+$ is given.  Let $n_ε$ be the smallest positive integer such that
$n_ε·ε$ is integral.  For the reader's convenience, we subdivided the proof of
Theorem~\ref{thm:kimsquare} into a number of relatively independent steps.

\subsection*{Step 1, the projectivisation of $𝒜$}
\approvals{Arne & --- \\ Jorge & yes \\ Stefan & yes}

Continuing along the ideas of Grauert, Vojta and Kim, \cite{MR0222087,
  MR1096125, MR1436743}, we consider the projectivisation $ℙ := ℙ_X(𝒜)$ together
with the natural projection morphism $π: ℙ → X$ and the standard ``Euler
sequence'' \CounterStep
\begin{equation}\label{eq:eulerx}
  0 → Ω¹_{ℙ/X}(1) → π^*𝒜 \xrightarrow{τ} 𝒪_{ℙ}(1) → 0.
\end{equation}
If $γ : T → X$ is a $\cC$-integral point, the image of the morphism $dγ_𝒜$ is
torsion free, hence either zero or invertible.  If non-zero, we have seen in
Fact~\ref{f:P1} that it defines a lifting to $ℙ$, which will always be denoted
by $Γ$,
$$
\xymatrix@C=1.5cm{
  & ℙ \ar[d]^{π} \\
  T \ar@/^2mm/[ru]^{Γ} \ar[r]_{γ} & X
}
$$
The lifting has the property that $Γ^*τ ≅ dγ_𝒜$, where $τ$ is the morphism that
appears in the Euler sequence \eqref{eq:eulerx} above.  In particular, the
following holds.
\begin{enumerate}
\item\label{il:vcbm.y} We have $Γ^*𝒪_ℙ(1) ≅ \Image(dγ_𝒜) ⊆ ω_T(\log γ^*D)$.

\item\label{il:vcbm.x} The composition
  \[
    Γ^*Ω¹_{ℙ/X}(1) → Γ^*π^*𝒜 = γ^*𝒜 \xrightarrow{dγ_𝒜} ω_T(\log γ^*D)
  \]
  vanishes.
\end{enumerate}

\subsection*{Step 2, symmetric differentials, degenerate and nondegenerate points}
\approvals{Arne & --- \\ Jorge & yes \\ Stefan & yes}

Next, we adapt some arguments from \cite[Sect.~3]{MR1436743} to our setting.

\begin{claim}\label{claim:A}
  If $m$ is a sufficiently large positive integer, then there exists an ample
  invertible sheaf $ℋ_m ∈ \Pic(S)$ and a non-trivial morphism
  $$
  σ : φ^*(ℋ_m^\vee) ⊗ ℬ^{⊗ mn_ε} → \Sym^{(2+ε)·mn_ε} 𝒜.
  $$
\end{claim}
\begin{proof}[Proof of Claim~\ref{claim:A}]
  Writing $N := (2+ε)·mn_ε$ for brevity and
  \[
    \sHom_m := \sHom \bigl( ℬ^{⊗ mn_ε},\: \Sym^N 𝒜 \bigr),
  \]
  it suffices to show that the push-forward sheaf $φ_*\sHom_m$ is not the zero
  sheaf, for sufficiently large $m$.  Choose a closed point $s ∈ S°$ with fibre
  $X_s$.  Standard computations show that
  \begin{small}
    \begin{align*}
      \rank \sHom_m & = (2+ε)·mn_ε+1 \\
      \deg \sHom_m|_{X_s} & = \deg \Bigl( (ℬ^{⊗ mn_ε})^\vee ⊗ \Sym^N 𝒜 \Bigr)\Bigr|_{X_s} & \text{Defn.~of }\sHom_m \\
                    & = \deg \Bigl( (ℬ^{⊗ mn_ε})^\vee ⊗ \Sym^N (φ^* Ω¹_S ⊕ ℬ) \Bigr)\Bigr|_{X_s} & \text{Sequence~\eqref{eq:x1}}\\
                    & = \deg \Bigl( (ℬ^{⊗ mn_ε})^\vee ⊗ \bigoplus_{i=0}^N ℬ^{⊗ i} \Bigr)\Bigr|_{X_s} & \deg (φ^* Ω¹_S)\bigr|_{X_s} = 0\\
                    & = \sum_{i=0}^N (i-mn_ε)·d \\
                    & = \const^+·\,m² + \text{(lower order terms in $m$)} \\
      \intertext{By Riemann-Roch, the holomorphic Euler characteristic of $\sHom_m|_{X_s}$ grows quadratically in $m$:}
      χ(\sHom_m|_{X_s}) & = \deg \sHom_m|_{X_s} + \bigl(1-g(X_s) \bigr)·\rank \sHom_m|_{X_s}\\
                    & = \const^+·\,m² + \text{(lower order terms in $m$)}.
    \end{align*}
  \end{small}
  It follows that the vector bundles $\sHom_m|_{X_s}$ have sections when $m$ is
  sufficiently large, and hence that $φ_*\sHom_m$ is indeed non-zero.
  \qedhere~(Claim~\ref{claim:A})
\end{proof}

We fix a positive integer number $m$, an ample line bundle $ℋ_m ∈ \Pic(S)$ and a
non-trivial morphism $σ : φ^*(ℋ_m^\vee) ⊗ ℬ^{⊗ mn_ε} → \Sym^{(2+ε)·mn_ε} 𝒜$, and
maintain this choice for the remainder of the present proof.  For brevity of
notation, write
$$
M := mn_ε \quad \text{and} \quad ℒ := φ^*(ℋ_m^\vee) ⊗ ℬ^{⊗ M}.
$$
The sheaf morphism $σ$ identifies $ℒ$ with a subsheaf of $\Sym^{(2+ε)M} 𝒜$.  As
we have seen in Fact~\ref{f:P2}, this defines a divisor $H ∈ \Div(ℙ)$.

\begin{notation}[Degenerate points]
  Let $γ$ be a $\cC$-integral point with $dγ_𝒜 \ne 0$.  Following Kim, we call
  $γ$ \emph{degenerate with respect to $σ$} if $Γ$ factors via $\supp H$.  If
  $Y ⊆ \supp H$ is an irreducible component, we say that $γ$ is \emph{degenerate
    with respect to $Y$} if $Γ$ factors via $Y$.
\end{notation}

Let $γ: T → X$ be a $\cC$-integral point with $dγ_𝒜 \ne 0$.  Fact~\ref{f:P3}
asserts that $γ$ is degenerate with respect to $σ$ if and only if the composed
map
$$
γ^*ℒ \xrightarrow{γ^*σ} γ^*\Sym^{(2+ε)M} 𝒜 \xrightarrow{\Sym^{(2+ε)M} dγ_𝒜}
ω_T(\log γ^*D)^{⊗(2+ε)M}
$$
vanishes identically.  Nondegenerate points therefore satisfy the following
height bound, which allows us to concentrate on degenerate points for the
remainder of the proof.

\begin{obs}\label{obs:4-16}
  If a $\cC$-integral point $γ: T → X$ is nondegenerate with respect to $σ$,
  then
  $$
  h_ℬ(γ) ≤ (2+ε)·\frac{\deg \Image dγ_𝒜}{[T:S]} + \frac{(2+ε)M+1}{M}·\deg_S ℋ_m .
  \eqno\qed \quad\text{(Observation~\ref{obs:4-16})}
  $$
\end{obs}

\subsection*{Step 3, height bounds for degenerate points -- setup and simplifications}
\approvals{Arne & yes \\ Jorge & yes \\ Stefan & yes}

Since the divisor $H ∈ \Div(ℙ)$ has only finitely many components, it suffices
to prove our height bounds for algebraic points that are degenerate with respect
to any one given irreducible component $Y ⊆ \supp H$.  We fix one component $Y$
for the remainder of the proof.  If $π(Y)$ is a proper subset of $X$, then there
are (up to reparametrisation) at most finitely $\cC$-integral points that are
degenerate with respect to $Y$, and there is nothing to prove.  We will
therefore assume the following.

\begin{asswlog}
  The component $Y$ dominates $X$.
\end{asswlog}

To set the stage for the next steps in the proof, we use resolution of
singularities in dimension two, \cite{MR0491722}, in order to find $k$-smooth
varieties $\wtilde{Y}$, $\what{Y}$, $\what{S}$ and a diagram,
$$
\xymatrix{ %
  &&&& &&&& Y \ar[d]^{π|_Y} \\
  \wtilde{Y} \ar@/^3mm/[urrrrrrrr]^{ρ\text{, desingularisation}} \ar[rrrr]_{\txt{\scriptsize $α$, generically finite, \\ \scriptsize isomorphism or purely inseparable}} &&&& \what{Y} \ar[rrrr]_{β\text{, generically finite, separable}} \ar[d]^{\what{φ}\text{, connected fibres}} &&&& X \ar[d]^{φ\text{, connected fibres}} \\
  &&&& \what{S} \ar[rrrr]_{\text{finite}} &&&& S,
}
$$
where $\supp β^*D ⊊ \what{Y}$ and $\supp α^*β^*D ⊊ \wtilde{Y}$ are strict normal
crossings divisors.

\begin{notation}
  Write $\what{η}$ for the generic point of $\what{S}$ and $\what{Y}_{\what{η}}$
  for the generic fibre.
\end{notation}

\begin{o-and-n}
  If $γ$ is degenerate with respect to $Y$ and if $Γ(T)$ is \emph{not} contained
  in the set of fundamental points of the birational map $ρ^{-1}$, then the
  morphism $Γ$ factorises via $\wtilde{Y}$.  In other words, there exist
  morphisms $\wtilde{Γ}$, $\what{Γ}$ that make the following diagram commute:
  $$
  \xymatrix@C=1.2cm{ %
    T \ar[d]^{\wtilde{Γ}} \ar@{=}[r] & T \ar[d]^{\what{Γ}} \ar@{=}[r] & T \ar[d]^{γ} \ar@{=}[r] & T \ar[d]^{Γ} \\
    \wtilde{Y} \ar[r]^{α} \ar@/_3mm/[rrr]_{ρ} & \what{Y} \ar[r]^{β} & X & Y.  \ar[l]_{π} }
  $$
  We say these points \emph{are liftable to $\wtilde{Y}$}, or \emph{lift to
    $\wtilde{Y}$}.
\end{o-and-n}

The fundamental points of $ρ^{-1}$ form a proper subset of $Y$ that does not
dominate $X$.  In particular, there are (up to reparametrisation) at most
finitely many $\cC$-integral points $γ$ that are degenerate with respect to $Y$,
but do not lift to $\wtilde{Y}$.  This yields the following claim, which allows
us to concentrate on liftable curves for the remainder of the proof.

\begin{claim}\label{claim:g3x}
  An inequality of the form $h_ℬ(γ) ≤ O(1)$ holds for all $\cC$-integral points
  $γ$ that are degenerate with respect to $Y$ but do not lift to $\wtilde{Y}$.
  \qed~\quad(Claim~\ref{claim:g3x})
\end{claim}

\subsection*{Step 4, height bounds for degenerate points -- generalised foliation}
\approvals{Arne & yes \\ Jorge & yes \\ Stefan & yes}

We will see in this step that the surface $\what{Y}$ carries a generalised
foliation, to which almost all liftings $\what{Γ}$ are tangent.  The arguments
are similar in spirit to those of Section~\ref{ssec:x1} above.  The existence of
the foliation comes out of the following claim.

\begin{claim}\label{claim:iop}
  There exists a saturated, invertible sheaf $\what{ℱ} ⊊ β^*𝒜$ such that the
  composed map
  \begin{equation}\label{eq:vcmb}
    \what{Γ}^*\what{ℱ} → \what{Γ}^*β^*𝒜 = γ^*𝒜 \xrightarrow{dγ_𝒜} ω_T(\log γ^*D)
  \end{equation}
  vanishes, up to reparametrisation for almost all $Y$-degenerate $\cC$-integral
  points $γ$ that are liftable to $\wtilde{Y}$.
\end{claim}
\begin{proof}[Proof of Claim~\ref{claim:iop}]
  The construction depends on whether the restricted morphism $π|_Y : Y → X$ is
  separable or not.

  \subsubsection*{Construction in the case where $π|_Y$ is separable}

  If $π|_Y$ is separable and $α$ is therefore an isomorphism, the discussion of
  the Euler sequence (Item~\vref{il:vcbm.x}) allows one to choose $\what{ℱ}$ as
  the saturation of $\what{ℱ}'$ inside $β^*𝒜$, where
  $$
  \what{ℱ}' := \Image \bigl( (α^{-1})^*ρ^* \bigl( Ω¹_{ℙ/X}(1)|_Y \bigr) →
  (α^{-1})^* ρ^*π^*𝒜 \bigr) ⊆ β^*𝒜.
  $$
  As a reflexive sheaf of rank one on a regular, two-dimensional scheme,
  $\what{ℱ}$ will then be invertible, as required.

  \subsubsection*{Construction in the case where $π|_Y$ is inseparable}

  In this case, we follow \cite{MR1757883} and recall that the natural morphism
  $Ω¹_{\what{Y}} → α_*Ω¹_{\wtilde{Y}}$ has a non-trivial kernel: it suffices to
  check this at the generic points of $\what{Y}$ and $\wtilde{Y}$, where this is
  a well-known property of Kähler differentials for inseparable field
  extensions, cf.\ \cite[Sect.~6.1, Exerc.~1.6]{MR1917232}.  The same statement
  therefore holds for the associated morphism of logarithmic differentials.  To
  be more precise, we use separability of $β$ to consider the inclusions
  $$
  \xymatrix{ %
    β^*𝒜 \ar@{^(->}[r] & β^* Ω¹_X( \log D ) \ar@{^(->}[r]^(.45){dβ} & Ω¹_{\what{Y}}
    \bigl(\log (β^* D)_{\red} \bigr), %
  }
  $$
  and use these to view $β^*𝒜$ as a subsheaf of
  $Ω¹_{\what{Y}} \bigl( \log (β^*D)_{\red} \bigr)$.  Then, we choose an invertible
  subsheaf
  $$
  \what{ℱ}' ⊆ \ker \Bigl( Ω¹_{\what{Y}} \bigl(\log (β^*D)_{\red} \bigr) → α_*
  Ω¹_{\wtilde{Y}} \bigl( \log (α^*β^*D)_{\red} \bigr) \Bigr) ∩ β^*𝒜
  $$
  and let $\what{ℱ}$ again be the saturation of $\what{ℱ}'$ inside $β^*𝒜$.
  Since $\what{Γ}$ factors via $\wtilde{Γ}$, it is clear that the composed map
  $\what{Γ}^*\what{ℱ}' → ω_T(\log γ^*D)$ vanishes.  It is then clear that the
  composed map \eqref{eq:vcmb} vanishes, except perhaps if $\what{Γ}(T)$ is
  contained in the support of $\factor{\what{ℱ}}{\what{ℱ}'}$.  There are,
  however, up to reparametrisation only finitely many $\cC$-integral points with
  that property.  \qedhere~(Claim~\ref{claim:iop})
\end{proof}

For the remainder of the proof, we choose one saturated, invertible sheaf
$\what{ℱ} ⊊ β^*𝒜$ as given by Claim~\ref{claim:iop}.  Consider the quotient
$\what{𝒬} := \factor{β^*𝒜}{\what{ℱ}}$, apply Lemma~\ref{lem:1a} to the morphism
$\what{φ}$ and let $\what{𝒢} ⊆ \what{𝒬}$ be one of the invertible sheaves given
by that lemma.  Following \cite[Sect.~4]{MR1436743}, we will now prove the
desired height bound under the assumption that a certain numerical inequality
holds.

\begin{claim}\label{claim:gfh1x}
  If
  $(2+ε)·\deg_{\what{Y}_{\what{η}}} \what{𝒢} ≥ \deg_{\what{Y}_{\what{η}}} β^*ℬ$,
  then a height inequality of the form~\eqref{eq:hI2} holds for all
  $\cC$-integral points $γ$ that are degenerate with respect to $Y$.
\end{claim}
\begin{proof}[Proof of Claim~\ref{claim:gfh1x}]
  In view of Claim~\ref{claim:g3x}, we are interested in $\cC$-integral points
  $γ$ that are degenerate with respect to $Y$ and liftable to $\wtilde{Y}$.  For
  all such $γ$, Néron's theorem yields the inequalities
  \begin{align*}
    h_ℬ(γ) = \frac{1}{[\what{S}:S]}·h_{β^*ℬ}(\what{Γ}) %
    & \overset{\text{Néron}}{≤} \frac{1}{[\what{S}:S]}·\frac{\deg_{\what{Y}_{\what{η}}} β^*ℬ}{\deg_{\what{Y}_{\what{η}}} \what{𝒢}}·h_{\what{𝒢}}(\what{Γ}) + O\Bigl( \textstyle{\sqrt{h_{β^*ℬ}(\what{Γ})}} \Bigr) \\
    & \overset{\hphantom{\text{Néron}}}{≤} \frac{1}{[\what{S}:S]}· (2+ε)·h_{\what{𝒢}}(\what{Γ}) + O\bigl( \textstyle{\sqrt{h_ℬ(γ)}} \bigr).
      \intertext{But then, Claim~\ref{claim:iop} asserts that the composed map \eqref{eq:vcmb}
      vanishes, up to reparametrisation for all but finitely many $\cC$-integral points $γ$.  This in turn implies by
      Observation~\ref{obs:l1.2} that these $γ$ satisfy}
      \frac{1}{[\what{S}:S]}·h_{\what{𝒢}}(\what{Γ}) & \overset{\hphantom{\text{Néron}}}{≤} \frac{1}{[\what{S}:S]}·\frac{\deg_T \Image dγ_𝒜}{[T:\what{S}]} = \frac{\deg_T \Image dγ_𝒜}{[T:S]}.
  \end{align*}
  Hence, the claim.  \qedhere~(Claim~\ref{claim:gfh1x})
\end{proof}

\subsection*{Step 5, height bounds for degenerate points -- end of proof}
\approvals{Arne & yes \\ Jorge & yes \\ Stefan & yes}

Claim~\ref{claim:gfh1x} finishes the proof of Theorem~\ref{thm:kimsquare} in one
special case.  We will therefore assume for the remainder of the proof that the
assumption of Claim~\ref{claim:gfh1x} does \emph{not} hold.

\begin{asswlog}\label{asswlog:r3}
  We have
  $(2+ε)·\deg_{\what{Y}_{\what{η}}} \what{𝒢} < \deg_{\what{Y}_{\what{η}}} β^*ℬ$.
\end{asswlog}

A short computation using Assumption~\ref{asswlog:r3} shows that
$β^*𝒜|_{\wtilde{Y}_{η'}}$ is slope-unstable.  The fact that $β$ is separable
immediately implies that the morphism $\wtilde{Y}_{\what{η}} → X_η$ is
separable, too.  In particular, \cite[Prop.~3.2]{Miyaoka87} implies that the
sheaf $𝒜|_{X_η}$ is slope-unstable.  In other words, there exists a saturated
(hence reflexive, hence invertible) subsheaf $ℱ ⊆ 𝒜$ such that
\begin{equation}\label{eq:vnh}
  \deg_{X_η} ℱ > \frac{1}{2}·\deg_{X_η} 𝒜 = \frac{1}{2}·\deg_{X_η} ℬ.
\end{equation}
Recalling from Proposition~\ref{prop:hbfol} that a height bound has already been
established for all points $γ$ tangent to the $𝒜$-foliation $ℱ$, it remains to
consider those points $γ$ for which the associated morphism
$$
γ^*ℱ → γ^*𝒜 \xrightarrow{dγ_𝒜} ω_T(\log γ^*D)
$$
does not vanish identically, so that $[T:S]·h_{ℱ} (γ) ≤ \deg_T \Image dγ_𝒜$.  But
then again Néron's theorem and Inequality~\eqref{eq:vnh} give a height
inequality of the form
$$
h_ℬ(•) \overset{\text{Néron}}{≤} 2·h_{ℱ}(•) + O \left( \textstyle{\sqrt{h_ℬ(•)}}
\right).
$$
This finishes the proof of Theorem~\ref{thm:kimsquare}.  \qed

\subsection{Improved height bounds in characteristic zero}
\approvals{Arne & yes \\ Jorge & yes \\ Stefan & yes}
\label{ssec:notksq}

If $\operatorname{char}(k) = 0$, then Jouanolou's theorem implies that the
family of curves on $X$ with general member irreducible and with every member
tangent to a given foliation is always bounded, \cite{Jouanolou, Ghys00}.  As a
result, we may replace the complicated height estimate of
Proposition~\ref{prop:hbfol} by the much stronger estimate $h_ℬ (γ) ≤ O (1)$.
With this improvement, the proof of Theorem~\ref{thm:kimsquare}, which uses
Proposition~\ref{prop:hbfol} only in its last step, goes through without many
changes, and allows us to replace Inequality~\eqref{eq:hI2} of
Theorem~\ref{thm:kimsquare} by the stronger bound
\begin{equation}\label{eq:ccjk}
  h_ℬ (γ) ≤ (2+ε)·\frac{\deg_T \Image dγ_𝒜}{[T:S]} + O \left( \textstyle{\sqrt{h_ℬ
        (γ)}}\right).
\end{equation}

%
%
\svnid{$Id: 05-coverings.tex 599 2021-08-23 12:06:26Z kebekus $}

\section{Covering constructions}
\label{sec:cover}
\subversionInfo
\approvals{Arne & --- \\ Jorge & yes \\ Stefan & yes}

Following ideas that originate from the work of Campana \cite{MR2831280} and
Campana-Păun, \cite{CP15}, we show Theorem~\ref{thm:height1b} by passing to a
\emph{strongly adapted} cover.  For varieties over $ℂ$, adapted covers are
discussed in Campana's work, but also in the survey article \cite{CKT16} or
\cite[Sect.~2.4]{MR2976311}.  In the present setup, where the characteristic
might be positive, more care needs to be taken.  The following claim summarises
the relevant properties.

\begin{prop}[Existence of a strongly adapted cover]\label{prop:B-1}
  In Setting~\ref{setting:CKim}, assume that the characteristic
  $p := \operatorname{char}(k)$ is either zero, or else that none of the
  $\cC$-multiplicities $m_i$ is a multiple of $p²$.  Then, there exists a
  smooth, projective $k$-variety $\what{X}$, a generically finite and separable
  surjection $c: \what{X} → X$ and a dense open set $S^{◦◦} ⊆ S°$ with preimages
  $\wcheck{X}^{◦◦} ⊆ \wcheck{X}$ and $\what{X}^{◦◦} ⊆ \what{X}$ forming a
  commutative diagram as follows,
  $$
  \xymatrix@R=1.5cm{ %
    \what{X} \ar[rrr]_{\what{c}\text{, genl.\ finite, separable}} \ar@/^5mm/[rrrrrr]^{c} \ar@/^10mm/[rrrrrrrrr]^{\what{φ}\text{, conn.\ fibres}} &&& \wcheck{X} \ar[rrr]_{\wcheck{c}\text{, genl.\ finite, separable}} &&& X \ar[rrr]_{φ\text{, conn.\ fibres}} &&& S \\
    \what{X}^{◦◦} \ar@{^(->}[u] \ar[rrr]^{\txt{\scriptsize $\what{c}^{\:◦◦}$, finite\\\scriptsize Galois with group $\what{G}$}} \ar@/_5mm/[rrrrrr]_{c^{◦◦}} \ar@/_10mm/[rrrrrrrrr]_{\what{φ}^{◦◦}\text{, conn.\ fibres}} &&& \wcheck{X}^{◦◦} \ar[rrr]^{\txt{\scriptsize $\wcheck{c}^{\:◦◦}$, finite\\\scriptsize Galois with group $\wcheck{G}$}} \ar@{^(->}[u] &&& X^{◦◦} \ar[rrr]^{φ^{◦◦}\text{, conn.\ fibres}} \ar@{^(->}[u] &&& S^{◦◦}, \ar@{^(->}[u]
  }
  $$
  such that the following holds.
  \begin{enumerate}
  \item\label{il:6a-1-1} We have
    $[\what{X}^{◦◦}:X^{◦◦}] = \lcm\, \{ m_i \,|\, m_i \ne ∞ \}$.

  \item\label{il:6a-1-1a} If $p = 0$, then $\what{c}$ is the identity morphism.
    If $p > 0$, then $\deg \what{c} ∈ \{1, p\}$, and $\deg \wcheck{c}$ is
    coprime to $p$.

  \item\label{il:6a-1-2} The morphism $\what{φ}^{\:◦◦}$ is snc for the pair
    $\bigl(\what{X}^{◦◦}, (c^{◦◦})^* D \bigr)$.

  \item\label{il:6a-1-2a} If $\what{R}^{◦◦} ⊂ \what{X}^{◦◦}$ is an irreducible
    component of the ramification divisor for $c^{◦◦}$, then $\what{R}^{◦◦}$ is
    either contained in $\supp (c^{◦◦})^* D$, or disjoint from that support.

  \item\label{il:6a-1-3} If $i$ is any index for which $m_i < ∞$, then all
    irreducible components $Γ$ of $(c^{◦◦})^* D_i$ have multiplicity
    $\mult_Γ (c^{◦◦})^* D_i = m_i$.

  \item\label{il:6a-1-4} If $i$ is any index with $m_i = ∞$, then all
    irreducible components $Γ$ of $(c^{◦◦})^* D_i$ have multiplicity
    $\mult_Γ (c^{◦◦})^* D_i = 1$.
  \end{enumerate}
\end{prop}

\begin{rem}[Smoothness]
  Item~\ref{il:6a-1-2} implies that $\what{X}^{◦◦}$ is smooth over $\Spec k$,
  that the components of $(c^{◦◦})^* D$ are smooth, and that no two of them
  intersect.
\end{rem}

\begin{rem}[Strongly adapted covers]
  In the language of earlier papers, Items~\ref{il:6a-1-3} and \ref{il:6a-1-4}
  are summarised by saying that the morphism $c^{◦◦}$ is \emph{strongly adapted}
  with respect to $(X^{◦◦}, D|_{X^{◦◦}})$, see for instance
  \cite[Sect.~2.6]{CKT16}.
\end{rem}

\begin{rem}[Optimality and generalisations]\label{rem:5-4}
  It is conceivable that the assumption ``none of the $\cC$-multiplicities $m_i$
  is a multiple of $p²$'' of Proposition~\ref{prop:B-1} is not necessary and
  that a more general statement can be proven if one is willing to replace the
  Artin-Schreier covers (that we discuss below) by more complicated
  Artin-Schreier-Witt covers.
\end{rem}

\subsection{Notation}
\label{ssec:covNot}
\approvals{Arne & yes \\ Jorge & yes \\ Stefan & yes}

The setup of Proposition~\ref{prop:B-1} will be maintained throughout
Sections~\ref{sec:adapted} and Section~\ref{sec:pf45}, and the following
notation will be used.

\begin{notation}[Strict transforms, restriction of divisors to open sets]\label{not:sT}
  Setting as in Proposition~\ref{prop:B-1}.  If $H ∈ \Div(X)$ is any reduced
  divisor, let $\what{H} ∈ \Div(\what{X})$ be the largest divisor in
  $(c^* H)_{\red}$ with the property that every component $\what{H}'$ of
  $\what{H}$ is generically finite over the image component of $H$.  By minor
  abuse of notation, we refer to $\what{H}$ as the \emph{strict transform}.
  Ditto for $\wcheck{H}$.  Restrictions of divisors $A ∈ ℚ\Div(X)$ to the open
  set $X^{◦◦}$ will as always be denoted by $A^{◦◦}$, ditto for divisors on
  $\wcheck{X}$ and $\what{X}$.  Since no confusion is likely to arise, we abuse
  Notation~\ref{not:sT} slightly, by writing $\what{D}^{◦◦}_{\fract}$ instead of
  the more correct, but somewhat clumsy $\what{D_{\fract}}^{◦◦}$.  Ditto for
  other divisors, and ditto for divisors on $\wcheck{X}^{◦◦}$.
\end{notation}

\subsection{Artin--Schreier covers}
\label{sec:asw}
\approvals{Arne & ---\\ Jorge & yes \\ Stefan & yes}

The proof of Proposition~\ref{prop:B-1} uses Artin--Schreier covers which we
briefly recall here.

\begin{construction}[Artin-Schreier cover]\label{cons:ASCover}
  Let $K$ be a field of positive characteristic $p$, not necessarily perfect or
  algebraically closed.  Let $C$ be a proper, smooth curve over $K$ and let
  $f ∈ K(C)$ be a non-constant rational function on $C$ which is regular on
  $U ⊊ C$ and has poles along $D = C ∖ U$.  Assume that $f$ is not of the form
  $g^p - g$ for some $g ∈ K(C)$.  Consider the following curve over $U$:
  $$
  A_U = \bigl\{ (c, y) ∈ U ⨯ 𝔸¹_K \:\Bigl|\: y^p - y = f\bigr\}.
  $$
  The classical theory of Artin-Schreier coverings (see \cite[Chapt.~VI,
  Thm.~6.4]{MR1878556} or \cite[Sect.~4.10]{MR3887555}) says that $A_U$ is
  irreducible, and the projection $A_U → U$ has degree $p$, is Galois and étale.
  To wit, identifying $𝔽_p$ with the prime field of $K$, the group $(𝔽_p, +)$
  acts on $A_U$ by
  $$
  \underline{n} : (c, y) ↦ (c, y+n).
  $$

  Then $A_U → U$ extends to a cyclic degree-$p$ cover $A → C$, where $A$ is the
  normal, proper $K$-curve constructed as follows: take the Zariski closure of
  $A_U$ in $C ⨯ 𝔸¹_K$, and normalise.
\end{construction}

\begin{prop}\label{prop:harbater}
  In the setting of Construction~\ref{cons:ASCover}, assume that the following
  holds.
  \begin{enumerate}
  \item All reduced, irreducible components of $D$ are étale over $\Spec K$.

  \item The pole orders of $f$ along the components of $D$ are prime to $p$.
  \end{enumerate}
  Then, the following will hold.
  \begin{enumerate}
  \item\label{il:1-2-1} The curve $A$ is smooth over $K$.

  \item\label{il:1-2-2} The morphism $A → C$ is étale over $U$ and totally
    (wildly) ramified over $D$.
  \end{enumerate}
\end{prop}

\begin{rem}
  This result is well-known if $K$ is perfect, but caution is advised when $K$
  is imperfect (and normality does not necessarily imply geometric normality).
\end{rem}

\begin{proof}[Proof of Proposition~\ref{prop:harbater}]
  Since the problem at hand is local on $C$, we may assume that $D$ consists of
  a unique closed point $c$, with residue field $κ(c)$ separable over $K$.

  Let $a$ be a point of $A$ lying over $c$ and view $y$ as a rational function
  on $A$.  Then, $y$ has a pole at $a$.  Writing $\mult_c$ and $\mult_a$ for the
  valuations induced by $c$ and $a$ on $K(C)$ and $K(A)$, respectively, we find
  that
  $$
  p \mult_a(y) = \mult_a(y^p) = \mult_a(y^p - y) = \mult_a(f) = e_{a/c} \mult_c(f)
  $$
  where $e_{a/c}$ denotes the ramification index of the extension
  $𝒪_{C,c} ⊆ 𝒪_{A,a}$.  Since $\mult_c(f)$ is prime to $p$ by assumption, we
  have $p \mid e_{a/c}$.  Since
  $$
  e_{a/c} ≤ \deg(A → C) = p,
  $$
  we obtain that $e_{a/c} = p$.  It follows that $a$ is the only point of $A$
  lying over $c$ and that $𝒪_{C,c} ⊆ 𝒪_{A,a}$ is a totally ramified extension of
  discrete valuation rings.  This proves \ref{il:1-2-2}.

  For \ref{il:1-2-1}, it remains to prove that $A$ is smooth over $K$ at the
  closed point $a$.  Indeed, $A_U → U$ is étale and $U$ is smooth over $K$, so
  the only remaining point is $a$.  Recall from
  \cite[\href{https://stacks.math.columbia.edu/tag/038S}{Tag 038S}, Lem.~33.12.2
  and 33.12.6]{stacks-project} that it is sufficient to check that if $L/K$ is a
  purely inseparable field extension, then $A_L$ is regular at $a_L$.

  Since the extension $κ(c)/K$ is separable, we know that $𝒪_{C,c} ⊗_K L$ is
  again a discrete valuation ring and the extension $𝒪_{C,c} ⊆ 𝒪_{C,c} ⊗_K L$ is
  weakly unramified, in the sense of
  \cite[\href{https://stacks.math.columbia.edu/tag/0EXQ}{Tag 0EXQ},
  Def.~15.104.1]{stacks-project}.  It then follows from
  \cite[\href{https://stacks.math.columbia.edu/tag/09ER}{Tag
    09ER}]{stacks-project} that
  $$
  𝒪_{A,a} ⊗_K L = 𝒪_{A,a} ⊗_{𝒪_{C,c}} (𝒪_{C,c} ⊗_K L)
  $$
  is again a discrete valuation ring, which is what we wanted to prove.
\end{proof}

\subsection{Taking roots out of a section}
\approvals{Arne & --- \\ Jorge & yes \\ Stefan & yes}

For the reader's convenience, we will briefly recall (in our special setting)
the construction of cyclic covers by ``taking roots out of a section''.  We
refer the reader to \cite[Sect.~3.5]{EV92} for a thorough, more general
discussion and to \cite[Lem.~2.3]{Viehweg95} for a summary.

\begin{construction}[Taking the $N^\text{th}$ root out of a divisor]\label{cons:tnrooB}
  Let $k$ be an algebraically closed field of arbitrary characteristic $p$.  Let
  $X$ be a smooth $k$-variety, equipped with an effective divisor
  $B = \sum a_i· B_i$ with snc support.  Assume that at least one of the numbers
  $a_i$ equals one.  Further, let $N ∈ ℕ$ be a number; if $p>0$ assume that $N$
  is prime to $p$.  Finally, let $ℒ ∈ \Pic(X)$ be any invertible sheaf whose
  $N^\text{th}$ tensor power is $ℒ^{⊗ N} ≅ 𝒪_X(B)$.

  Since $B$ is effective, there exists a section
  $σ ∈ H⁰\bigl( X,\, 𝒪_X(B)\bigr)$ whose zero-divisor is exactly $B$.  The dual
  of $σ$ gives a morphism $σ^* : (ℒ^*)^{⊗ N} → 𝒪_X$ and equips the direct sum
  \[
    𝒜 := \bigoplus_{i=0}^{N-1} \left(ℒ^{⊗ i}\right)^*
  \]
  with the structure of a sheaf of $𝒪_X$-algebras.  We consider the associated
  space $\Spec 𝒜$ and write $\overline{X}$ for its normalisation.  We call the
  natural morphism $γ : \overline{X} → X$ the \emph{covering obtained by taking
    the $N^\text{th}$ root out of $B$.}
\end{construction}

\begin{rem}[Alternate description]
  In more geometric terms, the space $\Spec 𝒜$ of Construction~\ref{cons:tnrooB}
  can also be described as follows.  The total spaces of the relevant sheaves,
  $|ℒ|$ and $|ℒ^{⊗ N}|$ are closely related: the group of $N^\text{th}$ unit
  roots acts on $|ℒ|$ by homotheties in fibre direction, and the space
  $|ℒ^{⊗ N}|$ is the quotient of $|ℒ|$ under this action.  If $Σ ⊊ |ℒ^{⊗ N}|$
  denotes the graph of the section $σ$, then $\Spec 𝒜$ is exactly the preimage
  of $Σ$ under the quotient map.
\end{rem}

We summarise the main properties of Construction~\ref{cons:tnrooB} in brief.

\begin{prop}[Properties of the covering construction]\label{prop:sx}
  Assume the setting of Construction~\ref{cons:tnrooB}.  Then, the following
  holds.
  \begin{enumerate}
  \item\label{il:s1a} The space $\overline{X}$ is irreducible.

  \item\label{il:s2a} The morphism $γ$ is finite, separable of degree $N$,
    Galois with cyclic Galois group, and étale away from $\supp B$.

  \item\label{il:s3a} If $i$ is any index and if $Γ ⊂ \supp γ^* B_i$ is any
    prime divisor, then $Γ$ appears in $γ^* B_i$ with multiplicity
    \[
      \mult_Γ γ^* B_i = \frac{N}{\gcd(N,a_i)}.
    \]

  \item\label{il:s4a} If $Y ⊂ X$ is any smooth curve that intersects the snc
    divisor $B$ transversely, and if $Y$ intersects at least one component of
    $B$ that has multiplicity one, then $\overline{Y} := γ^{-1}(Y)$ is normal
    and $γ|_{\overline{Y}} : \overline{Y} → Y$ is the covering obtained by
    taking the $N^\text{th}$ root out of $B|_Y$.
  \end{enumerate}
\end{prop}
\begin{proof}
  Items~\ref{il:s1a} is \cite[Lem.~2.3.d]{Viehweg95}, using the assumption that
  at least one of the numbers $a_i$ equals one.  Items \ref{il:s2a} and
  \ref{il:s3a} are \cite[Lem.~2.3.c and e]{Viehweg95}.  For Item~\ref{il:s4a},
  recall from \cite[Claim~3.8 and 3.12]{EV92} that the algebra structure of
  $𝒜$ extends to an algebra structure on the larger sheaf
  \[
    ℬ := \bigoplus_{i=0}^{N-1} \left(ℒ^{⊗ i} ⊗ 𝒪_X \Bigl( \left⌊
        \textstyle{\frac{-i}{N}}· B \right⌋ \Bigr) \right)^*,
  \]
  and that $\Spec ℬ$ is normal, hence isomorphic to $\overline{X}$.  The
  assumption that $Y$ intersects $B$ transversely guarantees that round-down
  commutes with restriction,
  \[
    ℬ|_Y = \bigoplus_{i=0}^{N-1} \left(ℒ|_Y^{⊗ i} ⊗ 𝒪_X \Bigl(
      \left⌊ \textstyle{\frac{-i}{N}}· B|_Y \right⌋ \Bigr)
    \right)^*,
  \]
  so that $γ^{-1}(Y) = \Spec ℬ|_Y$ is, again by \cite[Claim~3.8 and
  3.12]{EV92}, itself normal and in fact the covering obtained by taking the
  $N^\text{th}$ root out of $B|_Y$.
\end{proof}

\subsection{Proof of Proposition~\ref*{prop:B-1}}
\approvals{Arne & yes \\ Jorge & yes \\ Stefan & yes}

Assuming Setting~\ref{setting:CKim}, we will construct the diagram of
Proposition~\ref{prop:B-1} step-by-step and then prove that the construction has
all the desired properties.  Recall from Setting~\ref{setting:CKim} that
$D = \sum_{i=1}^d \frac{m_i-1}{m_i} · D_i$, where $m_i ∈ ℕ ∪ \{∞\}$ denote the
$\cC$-multiplicities of $D$.  Write
\[
  \fract := \{ i \::\: m_i < ∞\} \quad\text{and}\quad n := \lcm\, \{ m_i \,|\, i
  ∈ \fract \}.
\]

\subsection*{Step 1: Construction of $\wcheck{X}$}
\approvals{Arne & --- \\ Jorge & yes \\ Stefan & yes}

For any index $i ∈ \fract$, set
\[
  m'_i :=
  \begin{cases}
    m_i & \text{if $p = 0$,} \\
    \text{prime-to-}p\text{-part of }m_i & \text{otherwise}.
  \end{cases}
\]
Write $N := \lcm\, \{ m'_i \,|\, i ∈ \fract \}$.  Choose a very ample line
bundle $ℒ ∈ \Pic(X)$ such that
\[
  ℋ := ℒ^{⊗ N} ⊗ 𝒪_X\left( -\sum_{i ∈ \fract} \frac{N}{m'_i}·D_i \right) ∈
  \Pic(X)
\]
is still very ample, and choose a general element of the linear system,
$H ∈ |ℋ|$.  By Bertini's theorem, $H$ is an ample prime divisor that is not
contained in $\supp D$, the morphism $\supp H → S$ is separable and the divisor
$H + D$ has snc support.  By construction, there exists an isomorphism of line
bundle, $ℒ^{⊗ N} ≅ 𝒪_X(B)$, where
\[
  B := H + \sum_{i ∈ \fract} \frac{N}{m'_i}·D_i ∈ \Div(X).
\]
Finally, let $c : \overline{X} → X$ be the covering obtained by taking the
$N^\text{th}$ root out of $B$, as summarised in Construction~\vref{cons:tnrooB}.
Let $\wcheck{X}$ be a desingularisation of $\overline{X}$ that is isomorphic
over $\overline{X}_{\reg}$.  Desingularisation of this kind exist, because
$\overline{X}$ has dimension two, cf.~\cite[Thm.~1.2]{MR4167468} for a
convenient reference.  Let $\wcheck{c} : \wcheck{X} → X$ and
$\wcheck{φ} : \wcheck{X} → S$ be the composed morphisms.  We summarise the main
properties of our construction.

\begin{obs}\label{obs:y1z}
  The degree of $\wcheck{c}$ is given as
  \[
    [\wcheck{X} : X ] = N =
    \begin{cases}
      n & \text{if $p = 0$ or none of the finite $m_i$ are multiples of $p$,} \\
      n/p & \text{otherwise}.
    \end{cases}
  \]
\end{obs}

\begin{claim}\label{claim:y1a}
  There exists a dense open subset of $S$ over which $\wcheck{φ}$ is smooth.
\end{claim}
\begin{proof}[Proof of Claim~\ref{claim:y1a}]
  Let $s ∈ S$ be a general, closed point and consider the scheme-theoretic
  fibres $X_s$, $\overline{X}_s$ and $\wcheck{X}_s$.  We need to show that
  $\wcheck{X}_s$ is smooth.  But since the resolution morphism
  $\wcheck{X} → \overline{X}$ is an isomorphism away from a finite subset of
  $\overline{X}$, we find that the fibres $\overline{X}_s$ and $\wcheck{X}_s$
  agree, and it suffices to show that $\overline{X}_s$ is smooth.  But then, we
  have seen in Item~\ref{il:s4a} of Proposition~\ref{prop:sx} that the covering
  $c|_{\overline{X}_s} : \overline{X}_s → X_s$ is the covering obtained by
  taking the $N^\text{th}$ root out of $B|_{X_s}$.  In particular,
  $\overline{X}_s$ is an irreducible and normal $k$-curve, hence smooth over
  $\Spec k$ because $k$ is perfect.  \qedhere~(Claim~\ref{claim:y1a})
\end{proof}

\begin{claim}\label{claim:y1b}
  There exists is a dense open subset of $S$ over which $\wcheck{φ}$ is an snc
  morphism for the pair $\bigl(\wcheck{X}, \wcheck{D}+\wcheck{H} \bigr)$, where
  $\wcheck{D}$ and $\wcheck{H}$ are the strict transforms introduced in
  Notation~\ref{not:sT}.
\end{claim}
\begin{proof}[Proof of Claim~\ref{claim:y1b}]
  If $k$ is of characteristic zero, this is clear by generic smoothness.  If the
  characteristic of $k$ is finite, we already know that $\wcheck{φ}$ is
  generically smooth.  The covering degree $N = \deg \overline{c}$ is coprime to
  the characteristic by assumption.  It follows that the restriction of
  $\overline{c}$ to any component of $\overline{c}^* (D+H)$ is separable, and
  hence that the restriction of $\wcheck{c}$ to any component of
  $\wcheck{D} + \wcheck{H}$ is separable.  \qedhere~(Claim~\ref{claim:y1b})
\end{proof}

\subsection*{Step 2: Construction of $\what{X}$}
\approvals{Arne & --- \\ Jorge & yes \\ Stefan & yes}

Let $\wcheck{D}'$ denote the reduced divisor on $\wcheck{X}$, given as the
strict transforms of those divisors whose $\cC$-multiplicities $m_i$ are finite
multiples of $p$,
\[
  \wcheck{D}' := \sum_{i ∈ \fract} (1 - δ_{m_i,m'_i})·\wcheck{D}_i,
  \quad\text{where}\quad
  δ_{m_i,m'_i} := \left\{
    \begin{matrix}
      1 & \text{if } m_i = m'_i \\
      0 & \text{if } m_i ≠ m'_i.
    \end{matrix}
  \right.
\]
If $\wcheck{D}' = 0$, set $\what{X} := \wcheck{X}$, and let
$\what{c} : \what{X} → \wcheck{X}$ be the identity morphism.  Otherwise, write
$K := k(S)$ for the function field of $S$.  We have seen in
Claim~\ref{claim:y1a} that the generic fibre $\wcheck{X}_η$ is then a
$K$-smooth, projective $K$-curve.  Choose a rational function
$f ∈ K(\wcheck{X}_η)$ that has poles along $\wcheck{D}'_{η}$ of order prime to
$p$, and no other poles elsewhere.  The existence of such a function is a
straightforward consequence of the Riemann--Roch theorem.
Construction~\ref{cons:ASCover} will then give a cover
$\what{X}_η → \wcheck{X}_η$.  Choosing a suitable model, we obtain a smooth
variety $\what{X}$ and a generically finite morphism
$\what{c} : \what{X} → \wcheck{X}$ of degree $p$ as in
Proposition~\ref{prop:B-1}.

As before, let $\what{φ} : \what{X} → S$ and $c : \what{X} → X$ be the composed
morphism.  We summarise the main properties of our construction.

\begin{obs}\label{obs:z1z}
  By choice of $\wcheck{D}'$, the degree of $\what{c}$ is given as
  \[
    [\what{X} : \wcheck{X}] =
    \begin{cases}
      1 & \text{if $p = 0$ or none of the finite $m_i$ are multiples of $p$,} \\
      p & \text{otherwise}.
    \end{cases}
  \]
\end{obs}

\begin{obs}\label{obs:z1b}
  It follows from Claim~\ref{claim:y1b} and from Proposition~\ref{prop:harbater}
  that there exists a dense open subset of $S$ over which $\what{φ}$ is an snc
  morphism for the pair $\bigl(\what{X}, \what{D}+\what{H} \bigr)$.
\end{obs}

To end Step~2, choose one dense open subset $S^{◦◦} ⊆ S$ over which $\wcheck{φ}$
is an snc morphism and maintain this choice throughout.  We follow the notation
from the diagram of Proposition~\ref{prop:B-1} to denote the preimages of
$S^{◦◦}$ and to denote the restrictions of the relevant morphisms.

\subsection*{Step 3: End of proof}
\approvals{Arne & --- \\ Jorge & yes \\ Stefan & yes}

We need to check that our construction satisfies all properties spelled out in
Proposition~\ref{prop:B-1}.  We go through the list.

\begin{description}
\item[The morphisms $\what{c}^{\:◦◦}$ and $\wcheck{c}^{\:◦◦}$ are Galois] It is
  clear from construction that $\what{c}^{\:◦◦}$ and $\overline{c}^{\:◦◦}$ are
  Galois.  But $\overline{c}^{\:◦◦}$ and $\wcheck{c}^{\:◦◦}$ agree over the open
  subset $S^{◦◦}$ where $\wcheck{φ}$ is smooth and where the resolution morphism
  is therefore an isomorphism.

\item[Properties~\ref{il:6a-1-1} and \ref{il:6a-1-1a}] Follows from
  Observations~\ref{obs:y1z} and \ref{obs:z1z}.

\item[Property~\ref{il:6a-1-2}] Follows by choice of $S^{◦◦}$ and
  Observation~\ref{obs:z1b}.

\item[Property~\ref{il:6a-1-2a}] Recall from Item~\ref{il:s2a} of
  Proposition~\ref{prop:sx} and from Item~\ref{il:1-2-2} of
  Proposition~\ref{prop:harbater} that the ramification loci relate to the
  strict transforms of $D$ and $H$ as follows,
  \begin{align*}
    \supp \Ramification \wcheck{c}^{\:◦◦} & ⊆ \wcheck{B}^{◦◦} ⊆ \wcheck{D}^{◦◦} ∪ \wcheck{H}^{◦◦} \\
    \supp \Ramification \what{c}^{\:◦◦} & = \supp (\what{c}^{\:◦◦})^* \wcheck{D}' ⊆ \what{D}^{◦◦} \\
    \intertext{and therefore}
    \supp \Ramification c^{◦◦} & ⊆ \what{D}^{◦◦} ∪ \what{H}^{◦◦}
  \end{align*}
  Property~\ref{il:6a-1-2a} thus follows from the choice of $S^{◦◦}$ and
  Observation~\ref{obs:z1b}.

\item[Property~\ref{il:6a-1-3}] If $i$ is any index for which $m_i < ∞$, it
  follows from Item~\ref{il:s3a} of Proposition~\ref{prop:sx} that every
  irreducible component $\wcheck{Γ}$ of $\wcheck{D}_i$ has multiplicity equal to
  \begin{align*}
    \mult_{\wcheck{Γ}} \wcheck{c}^{\:*} D_i & = \frac{N}{\gcd\left(N,\factor{N}{m'_i}\right)} = \frac{N}{\factor{N}{m'_i}} = m'_i.  \\
    \intertext{Next, it follows from Item~\ref{il:1-2-2} of Proposition~\ref{prop:harbater}
    that every irreducible component $\what{Γ}$ of $\what{D}_i$ has multiplicity
    equal to}
    \mult_{\what{Γ}} \what{c}^{\:*} \wcheck{D}_i & =
    \begin{cases}
      1 & \text{if } m_i = m'_i, \\
      p & \text{if $m_i ≠ m'_i$ (and hence $m_i = m'_i·p$).}
    \end{cases}
  \end{align*}
  Property~\ref{il:6a-1-3} then follows.

\item[Property~\ref{il:6a-1-4}] If $i$ is any index for which $m_i = ∞$, it
  follows from Item~\ref{il:s2a} of Proposition~\ref{prop:sx} that every that
  $\wcheck{c}$ is étale over the generic point of $D_i$.  Likewise, if
  $\wcheck{Γ}$ is any irreducible component of the strict transform
  $\wcheck{D}_i$, then $\wcheck{Γ}$ is not contained in $\wcheck{D}'$, and
  Item~\ref{il:1-2-2} of Proposition~\ref{prop:harbater} shows that $\what{c}$
  is étale over the generic point of $\wcheck{Γ}$.  Property~\ref{il:6a-1-4}
  follows.
\end{description}
In summary, we checked that Properties~\ref{il:6a-1-1}--\ref{il:6a-1-4} all
hold.  This finishes the proof of Proposition~\ref{prop:B-1}.  \qed

%
%
\svnid{$Id: 06-adaptedDiffs.tex 599 2021-08-23 12:06:26Z kebekus $}

\section{The sheaf of adapted differentials}
\label{sec:adapted}
\subversionInfo

\subsection{Construction}
\label{ssec:adapted1}
\approvals{Arne & --- \\ Jorge & yes \\ Stefan & yes}

We maintain the setting and assumptions of Proposition~\ref{prop:B-1} in this
section.  Following ideas of Campana, we will consider a sheaf on
$\what{X}^{◦◦}$ called \emph{sheaf of adapted differentials}, and written as
$\what{Ω}¹_{X^{◦◦}}(\log \what{D}^{◦◦}_{\log})$.  In characteristic zero, this sheaf
is introduced and discussed at great length in \cite[Sect.~3]{CKT16}, though the
definition given there looks a little different from ours.  On a technical
level, the sheaf of adapted differentials is defined as follows.

\begin{defn}\label{defn:soad}
  Maintaining the setting of Proposition~\ref{prop:B-1}, and using
  Notation~\vref{not:sT} for strict transforms, we define the \emph{sheaf of
    adapted differentials} as
  \begin{equation}\label{eq:dad}
    \what{Ω}¹_{X^{◦◦}}(\log \what{D}^{◦◦}_{\log}) := \Bigl( 𝒥_{\what{D}^{◦◦}_{\fract}} ⊗ (c^{◦◦})^* Ω¹_X\bigl(\log ⌈D⌉ \bigr) \Bigr) \:+\:
    (\what{φ}^{◦◦})^* Ω¹_{\what{S}^{◦◦}},
  \end{equation}
  where the sum is the sum of subsheaves in
  $(c^{◦◦})^* Ω¹_X\bigl(\log ⌈D⌉ \bigr)$.  Its dual is called \emph{sheaf of
    adapted tangents} and will be denoted by
  $\what{𝒯}_{X^{◦◦}}(-\log \what{D}^{◦◦}_{\log})$.
\end{defn}

To justify the notation $\what{Ω}¹_{X^{◦◦}}(\log \what{D}^{◦◦}_{\log})$ in
Definition~\ref{defn:soad}, observe that sections in
$(c^{◦◦})^* Ω¹_X\bigl(\log ⌈D⌉ \bigr)$ are differential forms with logarithmic
poles along the support of $(c^{◦◦})^* ⌈D⌉$.  Taking the tensor product with the
ideal sheaf $𝒥_{\what{D}^{◦◦}_{\fract}}$ however cancels some of these poles, so
that sections in the tensor product are differential forms, with logarithmic
poles along the support of $(c^{◦◦})^* D_{\log}$ only.  This observation will
become important in Proposition~\vref{prop:soad1}.

\subsection{Main properties}
\label{ssec:adapted2}
\approvals{Arne & yes \\ Jorge & yes \\ Stefan & yes}

The following three propositions summarise the main properties of
$\what{Ω}¹_{X^{◦◦}}(\log \what{D}^{◦◦}_{\log})$ that will be relevant in the
sequel.  While Proposition~\ref{prop:soad1} and \ref{prop:soad2} are rather
elementary, the proof of Proposition~\ref{prop:soapSplit} requires some effort
and is not nearly as straightforward as one might wish.

\begin{prop}[Containment in $Ω¹_{\what{X}^{◦◦}}\bigl(\log D^{◦◦}_{\log} \bigr)$]\label{prop:soad1}
  Maintaining the setting of Proposition~\ref{prop:B-1} and the notation
  introduced in this section, $\what{Ω}¹_{X^{◦◦}}(\log \what{D}^{◦◦}_{\log})$ is
  a subsheaf of $Ω¹_{\what{X}^{◦◦}}\bigl(\log \what{D}^{◦◦}_{\log} \bigr)$.
  More precisely, there exists a commutative diagram of injective sheaf
  morphisms,
  $$
  \xymatrix{ %
    \what{Ω}¹_{X^{◦◦}}(\log \what{D}^{◦◦}_{\log}) \ar[r]^(.45){ι^{◦◦}_1} \ar@/_5mm/[rrd] & (c^{◦◦})^* Ω¹_X\bigl(\log ⌈D⌉ \bigr) \ar[r]^{ι^{◦◦}_2} & Ω¹_{\what{X}^{◦◦}} \bigl(\log (c^*D)_{\red} \bigr) \\
    & & Ω¹_{\what{X}^{◦◦}}\bigl(\log \what{D}^{◦◦}_{\log} \bigr), \ar[u]_{ι^{◦◦}_3} %
  }
  $$
  where $ι^{◦◦}_1$ and $ι^{◦◦}_3$ are the obvious inclusions, and where
  $ι^{◦◦}_2$ is the standard pull-back map for logarithmic differential forms.
\end{prop}
\begin{proof}
  The restriction of $ι^{◦◦}_2◦ι^{◦◦}_1$ to each of the two summands in
  \eqref{eq:dad} factorises via $ι^{◦◦}_3$.
\end{proof}

\begin{prop}[Quotient by $(\what{φ}^{◦◦})^* Ω¹_S$]\label{prop:soad2}
  Maintaining the setting of Proposition~\ref{prop:B-1} and the notation
  introduced in this section, the sheaf
  $\what{Ω}¹_{X^{◦◦}}(\log \what{D}^{◦◦}_{\log})$ contains
  $(\what{φ}^{◦◦})^* Ω¹_S$, and the quotient is isomorphic to
  $𝒥_{\what{D}^{◦◦}_{\fract}} ⊗ (c^{◦◦})^* ω_{X/S}(\log ⌈D⌉)$.  In particular,
  $\what{Ω}¹_{X^{◦◦}}(\log \what{D}^{◦◦}_{\log})$ is an extension of two locally
  frees, hence locally free.
\end{prop}
\begin{proof}
  Immediate from \eqref{eq:dad}.
\end{proof}

\begin{notation}[Sequence of relative adapted differentials]
  We refer to the quotient sequence
  \begin{equation}\label{eq:sorad}
    0 %
    → (\what{φ}^{◦◦})^* Ω¹_S %
    → \what{Ω}¹_{X^{◦◦}}(\log \what{D}^{◦◦}_{\log}) %
    → 𝒥_{\what{D}^{◦◦}_{\fract}} ⊗ (c^{◦◦})^* ω_{X/S}(\log ⌈D⌉) %
    → 0
  \end{equation}
  as the \emph{sequence of relative adapted differentials}.
\end{notation}

\begin{rem}\label{rem:soad}
  Items~\ref{il:6a-1-3} and \ref{il:6a-1-4} of Proposition~\ref{prop:B-1}
  immediately imply that the restriction of the quotient to the generic fibre of
  $\what{φ}$ is an invertible sheaf of degree
  $$
  \deg_{\what{X}_η} \Bigl( 𝒥_{\what{D}^{◦◦}_{\fract}} ⊗ (c^{◦◦})^* ω_{X/S}(\log ⌈D⌉) \Bigr)%
  = [\what{X}^{◦◦}:X^{◦◦}]·\deg_{X_η} \bigl(K_X+D\bigr).
  $$
\end{rem}

In order to apply the construction of adapted differentials to the problem of
finding geometric height inequalities for $\cC$-integral points, we need to
relate the splitting behaviour of Sequence~\eqref{eq:log6-4} to that of
\eqref{eq:sorad}.  The following proposition compares the two.

\begin{prop}[Splitting of the sequence of relative adapted differentials]\label{prop:soapSplit}
  Maintaining the setting of Proposition~\ref{prop:B-1} and the notation
  introduced in this section, assume that
  \begin{enumerate}
  \item\label{il:i1} the degree $d := \deg_{X_η} (K_X+D)$ is strictly positive,
    and

  \item\label{il:i2} the sequence of relative adapted differentials,
    Sequence~\eqref{eq:sorad}, splits when restricted to the generic fibre
    $\what{X}_η$.
  \end{enumerate}
  Then, Sequence~\eqref{eq:log6-4} splits when restricted to the generic fibre
  $X_η$.
\end{prop}

Proposition~\ref{prop:soapSplit} is shown in the subsequent
Sections~\ref{ssec:popsoapSplitprep} and \ref{ssec:popsoapSplit}.

\subsection{Preparation for the proof of Proposition~\ref*{prop:soapSplit}}
\label{ssec:popsoapSplitprep}
\approvals{Arne & yes \\ Jorge & yes \\ Stefan & yes}

We aim to relate Sequences~\eqref{eq:sorad} and \eqref{eq:log6-4} via
equivariant push-forward.  We refer the reader to \cite[Sect.~5.1]{MR0102537}
for an overview of elementary facts concerning $G$-sheaves and their
$G$-invariant push forwards; see also \cite[App.~A]{GKKP11} and
references therein.  The following elementary lemma turns out to be key.

\begin{lem}\label{lem:x2}
  Let $k$ be an algebraically closed field, and let $λ : A → B$ be a finite
  Galois morphism between smooth $k$-varieties, with Galois group $G$.  Let
  $Δ_B$ be a reduced divisor on $B$ and consider a $G$-invariant divisor $Δ_A$
  on $A$ with the following properties.
  \begin{enumerate}
  \item The divisors $Δ_A$ and $λ^* Δ_B - Δ_A$ are effective, so
    $0 ≤ Δ_A ≤ λ^* Δ_B$.
  \item Everywhere along the support of $λ^* Δ_B$, the divisor $Δ_A$ is strictly
    smaller than $λ^* Δ_B$.  In other words,
    $\supp (λ^*Δ_B-Δ_A) = \supp λ^* Δ_B$.
  \end{enumerate}
  Equip $𝒪_A(Δ_A)$ with the obvious structure of a $G$-subsheaf of
  $λ^* 𝒪_B(Δ_B)$.  If $ℰ$ is any locally free sheaf of $𝒪_B$-modules, then the
  canonical morphism
  $$
  ℰ ↪ λ_* \bigl( 𝒪_A(Δ_A) ⊗ λ^* ℰ \bigr)^G
  $$
  is an isomorphism.
\end{lem}

\begin{rem}\label{rem:xz}
  The morphism $λ$ of Lemma~\ref{lem:x2} is Galois, which is to say that $B$ is
  the quotient variety for the $G$-action on $A$.  In particular, regular
  functions on $G$-invariant open subsets of $A$ come from $B$ if and only if
  they are $G$-invariant.  The natural morphism $𝒪_B → λ_*(𝒪_A)^G$ is therefore
  an isomorphism, and then so are the natural morphisms $ℱ → λ_*(λ^* ℱ)^G$, for all
  locally free sheaves $ℱ$ on $B$.
\end{rem}

\begin{proof}[Proof of Lemma~\ref{lem:x2}]
  The problem is local on $B$ and respects direct sums.  We may therefore assume
  without loss of generality that $ℰ = 𝒪_B$.  We have inclusions of $G$-sheaves
  on $X$:
  $$
  λ^* 𝒪_B ⊆ 𝒪_A(Δ_A) ⊆ λ^* 𝒪_B(Δ_B).
  $$
  Using the fact that equivariant push-forward is left-exact, \cite[p.~197f]{MR0102537},
  we obtain a commutative diagram as follows:
  $$
  \xymatrix{ %
    𝒪_B \ar[d]_{\substack{\text{natl.~morphism, }n_1\\≅\text{ by Rem.~\ref{rem:xz}}}} \ar@{^(->}[rr]^{\text{inclusion}} && 𝒪_B(Δ_B) \ar[d]^{\substack{\text{natl.~morphism, }n_2\\≅\text{ by Rem.~\ref{rem:xz}}}}\\
    λ_*(λ^* 𝒪_B)^G \ar@{^(->}[r]_{α} & λ_*(𝒪_A(Δ_A))^G \ar@{^(->}[r]_{β} & λ_*(λ^* 𝒪_B(Δ_B))^G.
  }
  $$
  With the identifications indicated by the vertical arrows, a section in
  $λ_*(𝒪_A(Δ_A))^G$ is seen as a rational function $f$ on $B$, satisfying the
  following properties.
  \begin{enumerate}
  \item The function $f$ has at most simple poles along $Δ_B$, and no poles
    elsewhere.

  \item Along any component $δ ⊂ \supp λ^*Δ_B$, the pole order of the pull-back
    $λ^*f$ is required to satisfy
    $$
    \operatorname{poleOrder}_δ λ^*f ≤ \mult_δ Δ_A < \mult_δ λ^* Δ_B.
    $$
  \end{enumerate}
  Such a function is necessarily regular, which shows that the composition
  $α◦ n_1$ is an isomorphism, as required.
\end{proof}

\subsection{Proof of Proposition~\ref*{prop:soapSplit}}
\label{ssec:popsoapSplit}
\approvals{Arne & yes \\ Jorge & yes \\ Stefan & yes}

We assume that the sequence of relative adapted differentials,
Sequence~\eqref{eq:sorad}, splits when restricted to the generic fibre
$\what{X}_η$.  Equivalently, there exists an open subset $S^{◦◦◦} ⊆ S^{◦◦}$ such
that \eqref{eq:sorad} splits over the preimage of $S^{◦◦◦}$.  To keep the text
readable, we assume that $S^{◦◦◦} = S^{◦◦}$, so that there exists a morphism
$$
s : 𝒥_{\what{D}^{◦◦}_{\fract}} ⊗ (c^{◦◦})^* ω_{X/S}(\log ⌈D⌉) →
\what{Ω}¹_{X^{◦◦}}(\log \what{D}^{◦◦}_{\log})
$$
that splits Sequence~\eqref{eq:sorad}.

\begin{obs}\label{obs:1}
  Remark~\ref{rem:soad} and Assumption~\ref{il:i1} guarantee that $\Image(s)$
  equals the maximal destabilising subsheaf of
  $\what{Ω}¹_{X^{◦◦}}(\log \what{D}^{◦◦}_{\log})$ on $\what{X}_η$.  It follows
  that the splitting morphism $s$ is unique.
\end{obs}

\subsection*{Step 1: Embedding the sequence of relative adapted differentials}
\approvals{Arne & --- \\ Jorge & yes \\ Stefan & yes}

By definition, the sheaf of adapted differentials contains
$𝒥_{\what{D}^{◦◦}_{\fract}} ⊗ (c^{◦◦})^* Ω¹_X\bigl(\log ⌈D⌉ \bigr)$ and is
contained in $(c^{◦◦})^* Ω¹_X\bigl(\log ⌈D⌉ \bigr)$.  As a consequence, we find
that the sequence of relative adapted differentials, whose splitting behaviour
needs to be understood, is sandwiched between two sequences that relate to the
splitting behaviour of Sequence~\eqref{eq:log6-4}.

\begin{claim}[Embedding the sequence of relative adapted differentials]\label{claim:A15}
  The sequence of relative adapted differentials fits into a commutative diagram
  of sheaf morphisms on $\what{X}^{◦◦}$ with exact rows, as follows:
  $$
  \small \xymatrix@C=.5cm{ %
    (\what{φ}^{◦◦})^* Ω¹_S \ar@{^(->}[r] & (c^{◦◦})^* Ω¹_{X}(\log ⌈D⌉) \ar@{->>}[r] & (c^{◦◦})^* ω_{X/S}(\log ⌈D⌉) \\
    (\what{φ}^{◦◦})^* Ω¹_S \ar@{^(->}[r]^{\what{α}} \ar@{=}[u] & \what{Ω}¹_{X^{◦◦}}(\log \what{D}^{◦◦}_{\log}) \ar@{->>}[r] \ar@{^(->}[u] & 𝒥_{\what{D}^{◦◦}_{\fract}} ⊗ (c^{◦◦})^* ω_{X/S}(\log ⌈D⌉) \ar@{^(->}[u] \ar@/_5mm/[l]_{\what{s}\text{, splitting}} \\
    𝒥_{\what{D}^{◦◦}_{\fract}} ⊗ (\what{φ}^{◦◦})^* Ω¹_S \ar@{^(->}[r]_(.4){\what{δ}} \ar@{^(->}[u]_{\what{β}} & 𝒥_{\what{D}^{◦◦}_{\fract}} ⊗ (c^{◦◦})^* Ω¹_{X}(\log ⌈D⌉) \ar@{->>}[r] \ar@{^(->}[u]_{\what{γ}} & 𝒥_{\what{D}^{◦◦}_{\fract}} ⊗ (c^{◦◦})^* ω_{X/S}(\log ⌈D⌉) \ar@{^(->}[u].
  }
  $$
  All sheaves that appear in the diagram carry natural structures of
  $\what{G}$-sheaves, and all morphisms are morphisms of $\what{G}$-sheaves.
\end{claim}
\begin{proof}[Proof of Claim~\ref{claim:A15}]
  All assertions are clear from the construction, except perhaps the
  $\what{G}$-equivariance of the splitting morphism $\what{s}$.  The
  equivariance of $\what{s}$ follows from the uniqueness pointed out in
  Observation~\ref{obs:1}.  \qedhere~(Claim~\ref{claim:A15})
\end{proof}

\begin{notation}
  To keep this proof readable, we denote the entries in the
  diagram of Claim~\ref{claim:A15} by
  $$
  \begin{matrix}
    \what{A} & \what{B} & \what{C} \\
    \what{D} & \what{E} & \what{F} \\
    \what{G} & \what{H} & \what{I}.
  \end{matrix}
  $$
\end{notation}

\subsection*{Step 2: Dualisation}
\approvals{Arne & --- \\ Jorge & yes \\ Stefan & yes}

To see how the sequence of relative adapted differentials relates to
Sequence~\eqref{eq:log6-4}, one might be tempted to consider
$\what{G}$-invariant push-forward of the diagram from Claim~\ref{claim:A15} at
this point.  This will, however, not give the sheaves we are interested in.
Instead, we need to dualise first.

\begin{claim}[Dualisation]\label{claim:A15a}
  Dualising the diagram of Claim~\ref{claim:A15}, we obtain a commutative
  diagram with exact rows of $\what{G}$-sheaves as follows:
  $$
  \footnotesize \xymatrix@C=.5cm{ %
    (\what{φ}^{◦◦})^* 𝒯_S & (c^{◦◦})^* 𝒯_{X}(-\log ⌈D⌉) \ar@{->>}[l] \ar@{^(->}[d] & (c^{◦◦})^* 𝒯_{X/S}(-\log ⌈D⌉) \ar@{_(->}[l] \ar@{^(->}[d] \\
    (\what{φ}^{◦◦})^* 𝒯_S \ar@{=}[u] \ar@{^(->}[d] & \what{𝒯}_{X^{◦◦}}(-\log \what{D}^{◦◦}_{\log}) \ar@{->>}[l] \ar@{^(->}[d] \ar@/^5mm/[r]^{\what{s}^*\text{, splitting}} & \Bigl[ (c^{◦◦})^* 𝒯_{X/S}(-\log ⌈D⌉)\Bigr](\what{D}^{◦◦}_{\fract}) \ar@{_(->}[l] \ar@{^(->}[d] \\
    \Bigl[(\what{φ}^{◦◦})^* 𝒯_S \Bigr](\what{D}^{◦◦}_{\fract}) & \Bigl[ (c^{◦◦})^* 𝒯_{X}(-\log ⌈D⌉)\Bigr](\what{D}^{◦◦}_{\fract}) \ar@{->>}[l] & \Bigl[(c^{◦◦})^* 𝒯_{X/S}(-\log ⌈D⌉) \Bigr](\what{D}^{◦◦}_{\fract}) \ar@{_(->}[l].
  }
  $$
  As a subsheaf of
  \[
    \what{H}^* = \Bigl[ (c^{◦◦})^* 𝒯_{X}(-\log
    ⌈D⌉)\Bigr](\what{D}^{◦◦}_{\fract}),
  \]
  the sheaf
  \[
    \what{E}^* = \what{𝒯}_{X^{◦◦}}(-\log \what{D}^{◦◦}_{\log})
  \]
  is described as
  \begin{equation}\label{eq:A15a}
    \what{E}^* = \ker \Bigl(\what{H}^* → \factor{\what{G}^*}{\what{D}^*} \Bigr).
  \end{equation}
\end{claim}
\begin{proof}[Proof of Claim~\ref{claim:A15a}]
  Only the last line of the claim needs to be shown.  To begin, it follows
  directly from the definition of the sheaf of adapted differentials in
  \eqref{eq:dad} that we have an exact sequence of locally free sheaves
  $$
  \xymatrix@C=15mm{ %
    0 \ar[r] & \what{G} \ar[r]^(.4){\what{β} ⊕ \what{δ}} & \what{D}⊕\what{H} \ar[r]^(.6){\what{α} ◦ π_1-\what{γ} ◦ π_2} & \what{E} \ar[r] & 0.
  }
  $$
  Dualising, this give a commutative diagram with exact rows,
  $$
  \xymatrix@C=15mm{ %
    0 \ar[r] & 0 \ar[r] \ar[d] & \what{D}^* \ar[r]^{\Id} \ar[d] & \what{D}^* \ar[r] \ar@{^(->}[d]^{\what{β}^*} & 0 \\
    0 \ar[r] & \what{E}^* \ar[r]_(.4){\what{α}^*⊕(-\what{γ}^*)} & \what{D}^*⊕\what{H}^* \ar[r]_(.6){\what{β}^*π_1-\what{δ}^*π_2} & \what{G}^* \ar[r] & 0.
  }
  $$
  and the assertion then
 follows
  from the snake lemma.  \qedhere~(Claim~\ref{claim:A15a})
\end{proof}

\subsection*{Step 3: $\what{G}$-invariant push-forward}
\approvals{Arne & --- \\ Jorge & yes \\ Stefan & yes}

We consider the $\what{G}$-invariant push-forward of the diagram found above.
In other words, we apply the (left-exact) functor
$(\what{c}^{◦◦})_*(•)^{\what{G}}$ to all sheaves and sheaf morphisms involved.
For convenience of notation, write
$$
\wcheck{𝒯}_{X^{◦◦}}(-\log \what{D}^{◦◦}_{\log}) := (\what{c}^{◦◦})_* \Bigl( \what{𝒯}_{X^{◦◦}}\bigl(-\log \what{D}^{◦◦}_{\log}\bigr) \Bigr)^{\what{G}},
$$

\begin{claim}[$\what{G}$-invariant push-forward]\label{claim:A16}
  The $\what{G}$-invariant push-forward of the diagram in Claim~\ref{claim:A15a}
  is a commutative diagram with exact rows that reads as follows:
  $$
  \footnotesize \xymatrix@C=.5cm{ %
    (\wcheck{φ}^{◦◦})^* 𝒯_S & (\wcheck{c}^{◦◦})^* 𝒯_{X}(-\log ⌈D⌉) \ar@{->>}[l] \ar@{^(->}[d] & (\wcheck{c}^{◦◦})^* 𝒯_{X/S}(-\log ⌈D⌉) \ar@{_(->}[l] \ar@{^(->}[d] \\
    (\wcheck{φ}^{◦◦})^* 𝒯_S \ar@{=}[u] \ar@{^(->}[d] & \wcheck{𝒯}_{X^{◦◦}}(-\log \what{D}^{◦◦}_{\log}) \ar@{->>}[l] \ar@{^(->}[d] \ar@/^5mm/[r]^{\wcheck{s}\text{, splitting}} & \Bigl[ (\wcheck{c}^{◦◦})^* 𝒯_{X/S}(-\log ⌈D⌉)\Bigr](\wcheck{D}^{◦◦}_{\tame}) \ar@{_(->}[l] \ar@{^(->}[d] \\
    \Bigl[(\wcheck{φ}^{◦◦})^* 𝒯_S \Bigr](\wcheck{D}^{◦◦}_{\tame}) & \Bigl[ (\wcheck{c}^{◦◦})^* 𝒯_{X}(-\log ⌈D⌉)\Bigr](\wcheck{D}^{◦◦}_{\tame}) \ar@{->>}[l] & \Bigl[(\wcheck{c}^{◦◦})^* 𝒯_{X/S}(-\log ⌈D⌉) \Bigr](\wcheck{D}^{◦◦}_{\tame}).  \ar@{_(->}[l]
  }
  $$
\end{claim}
\begin{proof}[Proof of Claim~\ref{claim:A16}]
  The identification of the sheaves in the top row is clear by
  Remark~\ref{rem:xz}.  As for the identification of the other sheaves, observe
  that
  $$
  \Bigl[ (c^{◦◦})^* 𝒯_{X}(-\log ⌈D⌉)\Bigr](\what{D}^{◦◦}_{\fract}) = \biggl[ (\what{c}^{◦◦})^*\Bigl[\Bigl[ (\wcheck{c}^{◦◦})^* 𝒯_X(-\log ⌈D⌉) \Bigr](\wcheck{D}^{◦◦}_{\tame})\Bigr] \biggr](\what{D}^{◦◦}_{\wild}).
  $$
  The identification
  $$
  (\what{c}^{◦◦})_* \biggl( \Bigl[ (c^{◦◦})^* 𝒯_{X}(-\log ⌈D⌉)\Bigr](\what{D}^{◦◦}_{\fract}) \biggr)^{\what{G}} = (\what{c}^{◦◦})^*\Bigl[ (\wcheck{c}^{◦◦})^* 𝒯_X(-\log ⌈D⌉) \Bigr](\wcheck{D}^{◦◦}_{\tame})
  $$
  is thus an immediate consequence of Lemma~\ref{lem:x2}; ditto for the
  identifications of all the other sheaves.  It remains to prove surjectivity of
  the horizontal arrows towards the left column.  For the top arrow, this is
  clear.  Since the bottom row equals the top row tensored with the locally free
  $𝒪_{\wcheck{X}}(\wcheck{D}^{◦◦}_{\tame})$, surjectivity is also clear for the
  horizontal arrow in the bottom row.  Surjectivity of the middle arrow follows
  from the commutativity of the upper left square.
  \qedhere~(Claim~\ref{claim:A16})
\end{proof}

\begin{claim}[$\wcheck{G}$-sheaves in the diagram of Claim~\ref{claim:A16}]\label{claim:A17}
  All sheaves that appear in the diagram of Claim~\ref{claim:A16} carry natural
  structures of $\wcheck{G}$-sheaves, and all morphisms except possibly
  $\wcheck{s}$ are morphisms of $\wcheck{G}$-sheaves.
\end{claim}
\begin{proof}[Proof of Claim~\ref{claim:A17}]
  Exception for $\wcheck{𝒯}_{X^{◦◦}}(-\log \what{D}^{◦◦}_{\log})$ it is clear
  that all sheaves in the diagram of Claim~\ref{claim:A16} are
  $\wcheck{G}$-sheaves, and exception for the morphisms pointing to/from
  $\wcheck{𝒯}_{X^{◦◦}}(-\log \what{D}^{◦◦}_{\log})$, all morphisms are morphisms
  of $\wcheck{G}$-sheaves.  To prove the assertion, it will therefore suffice to
  show that the sheaf $\wcheck{𝒯}_{X^{◦◦}}(-\log \what{D}^{◦◦}_{\log})$ is
  stable under the action of $\wcheck{G}$, as a subsheaf of the
  $\wcheck{G}$-sheaf
  $\Bigl[ (\wcheck{c}^{◦◦})^* 𝒯_{X}(-\log ⌈D⌉)\Bigr](\wcheck{D}^{◦◦}_{\tame})$.

  However, using that the fact that the $\what{G}$-invariant push-forward
  functor $(\what{c}^{◦◦})_*(•)^{\what{G}}$ is left-exact, we know from
  Equation~\eqref{eq:A15a} of Claim~\ref{claim:A15a} that
  $$
  (\what{c}^{◦◦})_*(\what{E}^*)^{\what{G}} = \ker \Bigl((\what{c}^{◦◦})_*(\what{H}^*)^{\what{G}} → \factor{(\what{c}^{◦◦})_*(\what{G}^*)^{\what{G}}}{(\what{c}^{◦◦})_*(\what{D}^*)^{\what{G}}} \Bigr)
  $$
  The desired stability under the action of $\wcheck{G}$ follows from the
  observation that all morphisms to the right of the equality sign are
  morphisms of $\what{G}$-sheaves.  \qedhere~(Claim~\ref{claim:A17})
\end{proof}

\subsection*{Step 4: $\wcheck{G}$-invariant push-forward}
\approvals{Arne & yes\\ Jorge & yes \\ Stefan & yes}

In our discussion of the $\what{G}$-invariant push-forward we used the fact that
the splitting $\what{s}$ in the diagram of Claim~\ref{claim:A15} was unique for
numerical reasons, hence $\what{G}$-invariant.  These arguments do not
necessarily apply to the splitting $\wcheck{s}$, and we do not see why it should
be $\wcheck{G}$-invariant in general.  Using the special situation at hand, we
can however always find another $\wcheck{G}$-invariant splitting.

\begin{claim}[$\wcheck{G}$-invariant splitting]\label{claim:A18}
  There exists a $\wcheck{G}$-invariant morphism
  $$
  \wcheck{s}' : \wcheck{𝒯}_{X^{◦◦}}(-\log \what{D}^{◦◦}_{\log}) → \Bigl[ (\wcheck{c}^{◦◦})^* 𝒯_{X/S}(-\log ⌈D⌉)\Bigr](\wcheck{D}^{◦◦}_{\tame})
  $$
  that splits the middle row in the diagram of Claim~\ref{claim:A16}.
\end{claim}
\begin{proof}[Proof of Claim~\ref{claim:A18}]
  The order of the group $\wcheck{G}$ equals the degree
  $[\wcheck{X}^{◦◦} : X^{◦◦}]$ and is therefore coprime to the characteristic.
  Setting
  $\wcheck{s}' := \frac{1}{\#\wcheck{G}} · \sum_{g ∈ \wcheck{G}} g^*
  \wcheck{s}$ therefore yields the desired $\wcheck{G}$-invariant splitting.  \qedhere~(Claim~\ref{claim:A18})
\end{proof}

\begin{notation}
  Again we abuse notation slightly and assume without loss of generality that
  $\wcheck{s} = \wcheck{s}'$, so that all morphisms in Diagram~\ref{claim:A16}
  are in fact morphisms of $\wcheck{G}$-sheaves.
\end{notation}

\begin{claim}[$\wcheck{G}$-invariant splitting]\label{claim:A20}
  The $\wcheck{G}$-invariant push-forward of the diagram in
  Claim~\ref{claim:A16} is a commutative diagram with exact rows that reads as
  follows:
  $$
  \xymatrix@C=.5cm{ %
    𝒯_{S^{◦◦}} & 𝒯_{X^{◦◦}}(-\log ⌈D⌉) \ar@{->>}[l] \ar@{^(->}[d] & 𝒯_{X^{◦◦}/S^{◦◦}}(-\log ⌈D⌉) \ar@{_(->}[l] \ar@{^(->}[d] \\
    𝒯_{S^{◦◦}} \ar@{=}[u] \ar@{^(->}[d] & (\wcheck{c}^{◦◦})_* \Bigl( \what{𝒯}_{X^{◦◦}}(-\log \what{D}^{◦◦}_{\log})\Bigr)^{\wcheck{G}} \ar@{->>}[l] \ar@{^(->}[d] \ar@/_5mm/[r]_{s\text{, splitting}}& 𝒯_{X^{◦◦}/S^{◦◦}}(-\log ⌈D⌉) \ar@{_(->}[l] \ar@{^(->}[d] \\
    𝒯_{S^{◦◦}} & 𝒯_{X^{◦◦}}(-\log ⌈D⌉) \ar@{->>}[l] & 𝒯_{X^{◦◦}/S^{◦◦}}(-\log ⌈D⌉).  \ar@{_(->}[l]
  }
  $$
\end{claim}
\begin{proof}[Proof of Claim~\ref{claim:A20}]
  The identifications of the sheaves follow again from Lemma~\ref{lem:x2}.
  Surjectivity of the leftmost horizontal arrows follows as in the proof of
  Claim~\ref{claim:A16}.  \qedhere~(Claim~\ref{claim:A20})
\end{proof}

\subsection*{Step 5: End of proof}
\approvals{Arne & yes \\ Jorge & yes \\ Stefan & yes}

Observing that the top and bottom rows in the diagram of
Claim~\ref{claim:A20} agree with Sequence~\eqref{eq:log6-4}, we found the
desired splitting.  Proposition~\ref{prop:soapSplit} is thus shown.  \qed

%
%
\svnid{$Id: 07-height-bounds1.tex 599 2021-08-23 12:06:26Z kebekus $}

\section{Geometric height bounds --- proof of Theorem~\ref*{thm:height1b}}
\subversionInfo
\label{sec:pf45}
\approvals{Arne & yes \\ Jorge & yes \\ Stefan & yes}

We work in the setting of Theorem~\ref{thm:height1b} and assume that a number
$ε ∈ ℚ^+$ is given.  We apply Proposition~\ref{prop:B-1} and use the notation
introduced there as well as in Sections~\ref{ssec:covNot}, \ref{ssec:adapted1}
and \ref{ssec:adapted2}.  In particular, we consider the sequence of relative
adapted differentials
\begin{equation}\label{eq:6-4a}
  0 %
  → (\what{φ}^{◦◦})^* Ω¹_S %
  → \what{Ω}¹_{X^{◦◦}}(\log \what{D}^{◦◦}_{\log}) %
  → 𝒥_{\what{D}^{◦◦}_{\fract}} ⊗ (c^{◦◦})^* ω_{X/S} \bigl(\log ⌈D⌉ \bigr) %
  → 0.
\end{equation}

\begin{obs}[Splitting of Sequence~\ref{eq:6-4a}]\label{obs:6-4a}
  Using Assumptions~\ref{il:3-6-1} and \ref{il:3-6-2},
  Proposition~\ref{prop:soapSplit} implies that Sequence~\eqref{eq:6-4a} does
  not split when restricted to $\what{X}_η$.
\end{obs}

\subsection*{Step 1, extension of sheaves to $\what{X}$}
\approvals{Arne & --- \\ Jorge & yes \\ Stefan & yes}

We aim to apply Theorem~\ref{thm:kimsquare} to the surface $\what{X}$.  To this
end, we need to extend $\what{Ω}¹_{X^{◦◦}}(\log \what{D}^{◦◦}_{\log})$ to a
sheaf $𝒜$ that is defined on all of $\what{X}$.

\begin{construction}\label{cons:ssgg}
  Recalling from Proposition~\ref{prop:soad1} that there exists a natural
  diagram of inclusions,
  $$
  \xymatrix{ %
    \what{Ω}¹_{X^{◦◦}}(\log \what{D}^{◦◦}_{\log}) \ar[r]^(.45){ι_1^{◦◦}} \ar@/_5mm/[rrd] & (c^{◦◦})^* Ω¹_X\bigl(\log ⌈D⌉ \bigr) \ar[r]^{ι_2^{◦◦}} & Ω¹_{\what{X}^{◦◦}} \bigl(\log (c^*D)_{\red} \bigr) \\
    & & Ω¹_{\what{X}^{◦◦}}\bigl(\log \what{D}^{◦◦}_{\log} \bigr), \ar[u]_{ι_3^{◦◦}} %
  }
  $$
  we consider the extended morphisms
  $$
  \begin{matrix}
    ι_2 : & c^* Ω¹_X\bigl(\log ⌈D⌉ \bigr) & → & Ω¹_{\what{X}} \bigl(\log (c^*D)_{\red} \bigr) \\
    ι_3 : & Ω¹_{\what{X}}\bigl(\log \what{D}_{\log} \bigr) & → & Ω¹_{\what{X}} \bigl(\log (c^*D)_{\red} \bigr)
  \end{matrix}
  $$
  and let
  $$
  𝒜 ⊆ \Image(ι_2) ∩ \Image(ι_3) ⊆ Ω¹_{\what{X}} \bigl(\log (c^*D)_{\red} \bigr)
  $$
  be the largest subsheaf whose restriction to $\what{X}^{◦◦}$ agrees with
  $\Image(ι_2^{◦◦} ◦ ι_1^{◦◦})$.
\end{construction}

\begin{claim}\label{claim:6-2}
  The sheaf $𝒜$ is locally free and contains $\what{φ}^* Ω¹_S$.  Writing
  $ℬ := \bigl( \factor{𝒜}{\what{φ}^* Ω¹_S} \bigr)^{**}$, we have
  \begin{equation}\label{eq:b1}
    \deg_{\what{X}_η} ℬ = [\what{X}:X]·\deg_{X_η} \bigl(K_X+D\bigr).
  \end{equation}
\end{claim}
\begin{proof}[Proof of Claim~\ref{claim:6-2}]
  Both claims about $𝒜$ are consequences of its definition as ``the largest
  subsheaf …''.  First, it follows that $𝒜$ is reflexive, and hence locally free
  since $\what{X}$ is smooth of dimension two, \cite[Cor.~1.4]{MR597077}.
  Second, the inclusion
  $(\what{φ}^{◦◦})^* Ω¹_S ⊆ \what{Ω}¹_{X^{◦◦}}(\log \what{D}^{◦◦}_{\log})$
  implies that $𝒜$ must contain $\what{φ}^* Ω¹_S$.  The description of $ℬ$
  follows from Proposition~\ref{prop:soad2} and
  Remark~\ref{rem:soad}.  \qedhere~(Claim~\ref{claim:6-2})
\end{proof}

Claim~\ref{claim:6-2} yields a complex of sheaves, $\what{φ}^* Ω¹_S ↪ 𝒜 → ℬ$
that agrees over $\what{X}^{◦◦}$ with Sequence~\eqref{eq:6-4a}, the sequence of
relative adapted differentials.  The description of $ℬ$ given in
Claim~\ref{claim:6-2} has two immediate consequences which we note for future
reference.

\begin{consequence}\label{cons:59x}
  We have inequalities of height functions for $\what{X}/S$, applicable to all
  $\cC$-integral points $γ: \what{T} → \what{X}$, namely
  $$
  h_ℬ(•) - \const ≤ h_{c^* (K_{X/S}+D)}(•) ≤ h_ℬ(•) + \const.
  $$
  In other words, $h_{c^* (K_{X/S}+D)} = h_ℬ + O(1)$.
\end{consequence}
\begin{proof}[Proof of Consequence~\ref{cons:59x}]
  Properties~\ref{il:6a-1-3} and \ref{il:6a-1-4} from Proposition~\ref{prop:B-1}
  imply the equality of divisors
  $$
  (c^{◦◦})^* ⌈D⌉ - \what{D}^{◦◦}_{\fract} = (c^{◦◦})^* D.
  $$
  The restriction of $ℬ$ to $\what{X}^{◦◦}$ is therefore isomorphic to
  $𝒪_{\what{X}^{◦◦}} \bigl( c^* (K_{X/S}+D) \bigr)$.  Standard arguments,
  cf.~\cite[Sect.~2]{MR1436743}, will then give the relation between the two
  height functions.
\end{proof}

\begin{consequence}\label{cons:59y}
  We can express the number $d'$, which appears in the formulation of
  Theorem~\ref{thm:height1b} as
  $$
  d' \overset{\ref{il:6a-1-1}}{=} [\what{X}:X]·\deg_{X_η} (K_{X/S}+D) =
  \deg_{\what{X}_η} ℬ > 0.  \eqno\qed
  $$
\end{consequence}

The positivity of $\deg_{\what{X}_η} ℬ$ and the non-splitting of the sequence of
relative adapted differentials pointed out in Observation~\ref{obs:6-4a} allow
us to apply Theorem~\ref{thm:kimsquare} to the pair
$(\what{X}, \what{D}_{\log})$ over $S$.  In a nutshell, the following height
inequality holds for all $\what{γ} : \what{T} → \what{X}$ that are
$\cC$-integral points for the pair $(\what{X}, \what{D}_{\log})$,\CounterStep
\begin{equation}\label{eq:myhI2}
  h_ℬ (\what{γ}) \overset{\text{Thm.~\ref{thm:kimsquare}}}{≤} \max \bigl\{ \underbrace{\deg_{\what{X}_η} ℬ}_{\mathclap{=d' \text{ by Consequence~\ref{cons:59y}}}},\, 2+ε \bigr\}·\frac{\deg_{\what{T}} \Image d\,\what{γ}_𝒜}{[\what{T}:S]} + O \left( \textstyle{\sqrt{h_ℬ(\what{γ})}}\right),
\end{equation}
where, as before, $d\,\what{γ}_𝒜$ is the composed map
$$
\what{γ}^* 𝒜 %
→ \what{γ}^* Ω¹_{\what{X}} \bigl(\log \what{D}_{\log}\bigr) %
\overset{d\,\what{γ}}{→} ω_{\what{T}} \bigl(\log (\what{γ}^* \,
\what{D}_{\log})_{\red} \bigr).
$$

\subsection*{Step 2, orbifold integral points on $X$ and algebraic points on $\what{X}/S$}
\approvals{Arne & --- \\ Jorge & yes \\ Stefan & yes}

Given a $\cC$-integral point on $X$, we aim to bound its height using
\eqref{eq:myhI2}, by considering its preimage in $\what{X}$.  The following
notation will be used.

\begin{s-and-d}\label{setting:6-4}
  Given a $\cC$-integral point $γ : T → Σ_T ⊂ X$ of the pair $(X,D)$, we
  consider the preimage $c^{-1}Σ_T ⊆ \what{X}$ and choose a component
  $Σ_{\what{T}} ⊂ c^{-1}Σ_T$ that dominates $Σ_T$.  Denoting its normalisation
  by $\what{γ} : \what{T} → Σ_{\what{T}}$, we obtain a commutative diagram as
  follows:
  \begin{equation}\label{eq:jfgh}
    \begin{gathered}
      \xymatrix{ %
        \what{T} \ar[r]^{\what{γ}} \ar[d]_{α} & \what{X} \ar@/^3mm/[dr]^{\what{φ}} \ar[d]_{c} \\
        T \ar[r]_{γ} & X \ar[r]_{{φ}} & S.  %
      }
    \end{gathered}
  \end{equation}
  Since $Σ_{\what{T}}$ is not contained in $\supp \what{D}$ by assumption, we
  may consider the map $d\,\what{γ}_𝒜$ defined above.
\end{s-and-d}

\begin{rem}[The morphism $α$ is nearly always separable]\label{rem:mainas}
  Recall the assumption that the morphism $γ$ of Setting~\ref{setting:6-4} is
  birational onto its image and hence generically étale onto its image.  By base
  change, \cite[Prop.~3.3.c]{Milne80}, the same will hold for $\what{γ}$.  The
  morphism $c$ is separable and generically étale as well.  Unless $\img γ$ is
  contained in the closed set of $X$ over which $c$ fails to be étale, the
  composition $c◦\what{γ}$ thus is generically étale onto its image,
  \cite[Prop.~3.3.b]{Milne80}, and then so is $α$, \cite[Cor.~3.6]{Milne80}.  In
  particular, there are (up to reparametrisation) at most finitely many
  $\cC$-integral points $γ$ whose associated morphisms $α : \what{T} → T$ are
  inseparable.
\end{rem}

We aim to understand the image of the map $dα: α^* ω_T → ω_{\what{T}}$ and
to compare it to other sheaves of differentials on $\what{T}$.  The following
computation will be key.

\begin{computation}[Local computation]\label{comp:lcomp}
  In Setting~\ref{setting:6-4}, assume that the morphism $α$ is separable and
  let $p_{\what{T}} ∈ \what{T}$ be any given closed point, with image
  $p_T := α(p_{\what{T}}) ∈ T$.  Choose uniformising parameters
  $\what{t} ∈ 𝒪_{\what{T}, p_{\what{T}}}$ and $t ∈ 𝒪_{T, p_{T}}$ and write
  $a := \ord_{p_{\what{T}}} α^* t$, so that $α^* t = \what{t}^a·\what{u}$, where
  $\what{u} ∈ 𝒪_{\what{T}, p_{\what{T}}}$ is a unit.  Then
  \begin{align*}
    dα(dt) = a·\what{u}·\what{t}^{a-1}d\what{t}+\what{t}^ad\what{u} \quad\text{and}\quad %
    dα(d \log t) = a·d\log\what{t}+d\log\what{u}.
  \end{align*}
  Since $α$ is separable by assumption, both $dα(dt)$ and $d α (\log dt)$ are
  non-zero.  In particular, either $p$ does not divide $a$ or $d \what{u} ≠ 0$.
  In any case, observing that $d\what{u}$ and $d\log\what{u}$ either vanish
  simultaneously or differ only by a unit, we find that
  $\what{t}^a · dα(d \log t)$ is a section in $\Image(dα)$, and more generally
  that
  \begin{multline}\label{eq:lcomp}
      𝒥^a_{p_{\what{T}}} · \Image \bigl(dα_{\log}: α^* ω_T(\log p_T) → ω_{\what{T}}(\log p_{\what{T}})\bigr) \\
      ⊆ \Image(dα: α^* ω_T → ω_{\what{T}}).
  \end{multline}
  Observe that the number $a$ is bounded from above by the degree of $c$, that
  is, $a ≤ [\what{X}:X]$.
\end{computation}

\begin{claim}\label{claim:6-5}
  In Setting~\ref{setting:6-4}, assume that $α$ is separable.  Then, the
  following sequence of inclusions holds over $S^{◦◦}$,
  \begin{equation}\label{eq:6-5}
    \Image d\,\what{γ}_𝒜 ⊆ \Image \bigl( dα: α^* ω_T → ω_{\what{T}} \bigr) ⊆
    ω_{\what{T}}.
  \end{equation}
\end{claim}
\begin{proof}[Proof of Claim~\ref{claim:6-5}]
  We will prove Inequality~\eqref{eq:6-5} locally, in the neighbourhood of any
  given closed point $p_{\what{T}} ∈ \what{T}$ lying over $S^{◦◦}$.  We use the
  notation introduced in Computation~\ref{comp:lcomp} and write
  $$
  p_{\what{X}} := \what{γ}(p_{\what{T}}), \quad p_X := γ(p_T).
  $$
  The assumption that $γ$ is a $\cC$-integral point immediately implies that
  $p_{\what{X}}$ is not contained in the support of $\what{D}_{\log}$.  In
  particular, we see that $\Image d\,\what{γ}_𝒜 ⊆ ω_{\what{T}}$ near
  $p_{\what{T}}$.  If $p_X$ is not contained in $\supp D$,
  Inclusion~\eqref{eq:6-5} follows easily: by definition of
  $\what{Ω}¹_{X^{◦◦}}(\log \what{D}^{◦◦}_{\log})$, one sees that
  $𝒜 = c^* Ω¹_X ⊆ Ω¹_{\what{X}}$ near $p_{\what{X}}$.  Inclusion~\eqref{eq:6-5}
  will then follow immediately from the chain rule for taking derivatives.

  For the remainder of the proof, we consider the case where $p_X$ \emph{is}
  contained in $\supp D$.  The assumption that $φ^{◦◦}$ is snc for the pair
  $(X^{◦◦}, D^{◦◦})$ implies that $p_X$ is contained in a unique component of
  $D$, say $D_i ⊆ D$; the coefficient $m_i$ is then finite.  Choose local
  systems of parameters as follows.
  \begin{itemize}
  \item Choose parameters $x, y ∈ 𝒪_{X, p_X}$ so that $D_i = \{ y = 0\}$.

  \item Choose parameters $\what{x}, \what{y} ∈ 𝒪_{\what{X}, p_{\what{X}}}$ so
    that $\what{x} = c^*(x)$ and $\what{y}^{m_i} = (\unit)·c^*(y)$.
  \end{itemize}
  Near $p_{\what{T}}$, the sheaf $𝒜$ equals the sheaf
  $\what{Ω}¹_{X^{◦◦}}(\log \what{D}^{◦◦}_{\log})$ of adapted differentials.
  Recall from Definition~\ref{defn:soad} that $𝒜$ is a subsheaf of
  $c^* Ω¹_X\bigl(\log ⌈D⌉ \bigr)$, and can be generated as follows,
  $$
  𝒜_{p_{\what{T}}} = \bigl\langle c^* dx,\: \what{y}·c^* d\log y \bigr\rangle_{p_{\what{T}}} ⊆ \Bigl(c^* Ω¹_X\bigl(\log ⌈D⌉ \bigr) \Bigr)_{p_{\what{T}}}.
  $$
  To prove Inclusion~\eqref{eq:6-5}, it will therefore suffice to show that the
  following two logarithmic forms, which are \emph{a priori} sections in
  $ω_{\what{T}}\bigl(\log ((γ◦ α)^* ⌈D⌉)_{\red} \bigr)$, are in fact contained
  in the smaller sheaf $\Image(dα)$:
  \begin{equation}\label{eq:xfgh}
    \begin{aligned}
      (d\what{γ}◦dc)(dx) & = dα(dγ^*x) &&\text{and} \\
      d\what{γ}\bigl(\what{y}·dc(d\log y)\bigr) & = \what{γ}^*\what{y}·dα(d\log γ^*y).
    \end{aligned}
  \end{equation}
  Since there is nothing to show for $dα(dγ^*x)$, we concentrate on the second
  form.  There, the desired inclusion will follow from \eqref{eq:lcomp} once we
  show that $\ord_{p_{\what{T}}} \what{γ}^*\what{y} ≥ a$.  But then,
  \begin{align*}
    \ord_{p_{\what{T}}} \what{γ}^*\what{y} & = m_i^{-1}·\ord_{p_{\what{T}}} \what{γ}^*c^* y && \text{Choice of }\what{y}\\
                                           & = m_i^{-1}·\ord_{p_{\what{T}}} α^* γ^* y && \text{Diagram~\eqref{eq:jfgh}}\\
                                           & = m_i^{-1}·a·\ord_{p_T} γ^* y
  \end{align*}
  and the claim follows once we recall that $γ$ is a $\cC$-integral point, which
  implies in particular that $\ord_{p_T} γ^* y ≥ m_i$.  In summary, we have seen
  that the rational forms mentioned in \eqref{eq:xfgh} are both contained in
  $\Image(dα)$, which finishes the proof.  \qedhere~(Claim~\ref{claim:6-5})
\end{proof}

\begin{claim}\label{claim:6-6}
  In Setting~\ref{setting:6-4}, there exists a number $\const ∈ ℕ$ such that the
  following inequality holds for every $\cC$-integral point $γ : T → Σ_T$ and
  for every choice of a preimage component $Σ_{\what{T}} ⊂ c^{-1}Σ_T$ that
  dominates $Σ_T$:
  \begin{equation}\label{eq:5-15-1}
    \frac{\deg_{\what{T}} \Image d\,\what{γ}_𝒜}{[\what{T}:S]} %
    ≤ \frac{\deg_{\what{T}} \Image dα}{[\what{T}:S]} + \const %
    = δ(γ) + \const.
  \end{equation}
\end{claim}
\begin{proof}[Proof of Claim~\ref{claim:6-6}]
  Recalling from Remark~\ref{rem:mainas} that the morphism $α$ is nearly always
  separable, it suffices to consider the separable case only.  The equality in
  \eqref{eq:5-15-1} follows by a simple computation, using that $α$ is
  separable, so $\deg_{\what{T}} \Image dα = \deg_{\what{T}} α^* ω_T$.

  The inequality in \eqref{eq:5-15-1} would clearly follow if the
  Inclusion~\eqref{eq:6-5} would hold for all points of $\what{T}$.  This may,
  however, not always be the case.  If $p_{\what{T}} ∈ \what{T}$ is a closed
  point where the Inclusion~\eqref{eq:6-5} fails, then we know from
  Claim~\ref{claim:6-5} that $p_{\what{T}}$ cannot lie over $S^{◦◦}$.  In
  particular, there are at most $[\what{T}:S]·\#(S ∖ S^{◦◦})$ such points
  in $\what{T}$.  How badly can Inequality~\eqref{eq:6-5} fail at the point
  $p_{\what{T}}$?  By definition, we know that
  $𝒜 ⊆ Ω¹_{\what{X}} \bigl(\log (c^*D)_{\red} \bigr)$, which gives us an
  inclusion
  $$
  \Image d\what{γ}_𝒜 ⊆ \underbrace{\Image \Bigl(d(γ◦ α) : c^* Ω¹_X(\log ⌈ D ⌉) → ω_{\what{T}}\bigl(\log ((γ◦ α)^* ⌈ D ⌉)_{\red} \bigr) \Bigr)}_{=: \sf A}.
  $$
  Locally near $p_{\what{T}}$, Computation~\ref{comp:lcomp} shows
  $$
  𝒥^a_{p_{\what{t}}}·\Image d\what{γ}_𝒜 ⊆ 𝒥^a_{p_{\what{t}}}·{\sf A}
  \overset{\eqref{eq:lcomp}}{⊆} \Image(dα: α^* ω_T → ω_{\what{T}}).
  $$
  Recalling that $a ≤ [\what{X}:X]$, we can therefore finish the proof by
  setting
  $$
  \const := [\what{X}:X]·\#(S ∖ S^{◦◦}).  \eqno
  \qedhere~\text{(Claim~\ref{claim:6-6})}
  $$
\end{proof}

\subsection*{Step 3, end of proof}
\approvals{Arne & yes \\ Jorge & yes \\ Stefan & yes}

Combining the results obtained so far, we find that the following sequence of
inequalities holds for all orbifold integral points $γ$ in $X$, and all choices
of preimage components.  As before, we use the notation introduced in
\ref{setting:6-4} above.
\begin{align*}
  h(γ) & = h_{c^* (K_{X/S}+D)}(\what{γ}) ≤ h_{ℬ}(\what{γ}) + \const && \text{Consequence~\ref{cons:59x}}\\
       & ≤ \max \{ d', 2+ε \}·\frac{\deg_{\what{T}} \Image d\,\what{γ}_𝒜}{[\what{T}:S]} + O \left( \textstyle{\sqrt{h_ℬ(\what{γ})}}\right) && \text{Inequality~\eqref{eq:myhI2}} \\
       & ≤ \max \{ d', 2+ε \}·δ(γ) + O \left( \textstyle{\sqrt{h_ℬ(\what{γ})}}\right) && \text{Claim~\ref{claim:6-6}} \\
       & = \max \{ d', 2+ε \}·δ(γ) + O \left( \textstyle{\sqrt{h(γ)}}\right) && \text{Consequence~\ref{cons:59x}}
\end{align*}
This ends the proof of Theorem~\ref{thm:height1b}.  \qed

\subsection{Notes on the proof}

---

\subsubsection{Improved height bounds in characteristic zero}
\approvals{Arne & yes \\ Jorge & yes \\ Stefan & yes}
\label{sssec:dfg}

If $\operatorname{char}(k) = 0$, recall from Section~\ref{ssec:notksq} that the
height bound of Theorem~\ref{thm:kimsquare} can be improved.  In fact, replacing
\eqref{eq:myhI2} by the improved bound \eqref{eq:ccjk} when we apply
Theorem~\ref{thm:kimsquare} in Step~4 of our proof, we obtain the improved
result claimed in Section~\ref{ssec:impbhcz} above.

\subsubsection{Sharpness of Theorem~\ref*{thm:height1b} and the Kodaira-Spencer map}
\approvals{Arne & yes \\ Jorge & yes \\ Stefan & yes}
\label{ssec:doBetter}

As promised on Page~\pageref{sssec:B} at the end of Section~\ref{sssec:B}, we
add a few words concerning the sharpness of the bound given by
Theorem~\ref{thm:height1b}.  The bounds obtained in \cite[Thm.~2]{MR1436743} and
\cite[Claim~2.2]{MR1815399}, rely on a control of the degree of the restriction
of the foliation $𝒢$ to $X_η$, appearing in the proof of Proposition
\ref{prop:hbfol}.  More specifically, better bounds for $\deg_{X_η} 𝒢$ allow us
to extract improved height bounds from Inequality~\eqref{eq:keyinequality}.  We
do not know if one can replace the number
$$
d'= \deg_{X_η} (K_X+D)· \lcm\, \{ m_i \,|\, m_i \ne ∞ \} = \deg_{X_η} (K_X+D)·
\deg(c)
$$
by the smaller number $d = \deg_{X_η} (K_X+D)$.  For that, one would need to
show that $\deg_{X_η} 𝒢 ≥ [\what{X}:X]$ when $𝒢$ is the foliation on $\what{X}$
tangent to infinitely many degenerate $\what{X}/S$ algebraic points.  We could
not find a way to prove this stronger inequality, without imposing further
conditions on the logarithmic Kodaira-Spencer map of the pair $(X, ⌈D⌉)$.

%
%
\svnid{$Id: 07-bounds-orbi.tex 410 2019-04-26 19:23:16Z kebekus $}

\section{Geometric height bounds --- proof of Theorem~\ref*{thm:height1a}}
\subversionInfo
\label{sec:pf45b}
\approvals{Arne & --- \\ Jorge & yes \\ Stefan & yes}

We aim to apply Theorem~\ref{thm:height1b} to a pair $(X, D')$, where the
divisor $D'$ is obtained from $D$ by lowering the coefficients.  To be precise,
we set
\[
  D' := \sum_{i ∈ \fract} \frac{m'_i-1}{m'_i}·D_i + D_{\log},
\]
where
$$
m'_i :=
\begin{cases}
  m_i & \text{if } p² \nmid m_i \\
  3 & \text{if } p=2 \text{ and } p² | m_i \\
  \text{largest prime factor of $m_i$} & \text{otherwise.}
\end{cases}
$$

We assume in Item~\ref{il:3-6-2} of Theorem~\ref{thm:height1a} that
Sequence~\eqref{eq:log6-4} does not split when restricted to the generic fibre
$X_η$.  Recalling that the sequence always splits if $X_η = ℙ¹$ and $D$ contains
less than four points, the assumptions of Theorem~\ref{thm:height1a} imply, by
means of elementary computations, that the pair $(X, D')$ satisfies each of the
assertions below.
\begin{enumerate}
\item\label{il:A1} The pair $(X,D')$ is a $\cC$-pair.  We have $D' ≤ D$.

\item\label{il:A2} None of the numbers $(m_i)_{i∈ \fract}$ is a multiple of
  $p²$.

\item\label{il:A3} We have $\deg_{X_η} (K_X+D') ≥ \frac{1}{6}$.

\item\label{il:A4} We have $\supp D' = \supp D$.  The analogue of
  Sequence~\eqref{eq:log6-4} for the pair $(X, D')$ does not split when
  restricted to the generic fibre $X_η$.

\item\label{il:A5} We have
  $\lcm\{ m'_i \:|\: i ∈ \fract\} ≤ 3·\lcm\{ m_i \:|\: i ∈ \fract\}$.
\end{enumerate}

Item~\ref{il:A1} guarantees that any $\cC$-integral point of the pair $(X, D)$
is automatically $\cC$-integral for the pair $(X, D')$.  Items~\ref{il:A2} ---
\ref{il:A4} guarantee that the pair $(X, D')$ satisfies all assumptions made in
Theorem~\ref{thm:height1b}.  Now, given any number $ε > 1$ and choosing
\begin{align*}
  d & := \deg_{X_η} (K_X+D) & d_{K_X+D'} & := \deg_{X_η} (K_X+D') \\
  d' & := d·\lcm\{ m_i \:|\: i ∈ \fract\} & d'_{K_X+D'} & := d_{K_X+D'}·\lcm\{ m'_i \:|\: i ∈ \fract\},
\end{align*}
and applying Néron's theorem, \cite[Thm.~2.11]{MR1002324}, we obtain the
following
\begin{small}
  \begin{align*}
    d'_{K_X+D'} & ≤ d·\lcm\{ m'_i \:|\: i ∈ \fract\} ≤ 3·d' && \text{\ref{il:A1} and \ref{il:A5}}\\
    \intertext{and}
    h_{K_X+D}(γ) & ≤ \frac{d}{d_{K_X+D'}}·h_{K_X+D'}(γ) + O\left( \textstyle{\sqrt{h_{K_X+D}(γ)}} \right) && \text{Néron's theorem} \\
                & ≤ 6d·h_{K_X+D'}(γ) + O\left( \textstyle{\sqrt{h_{K_X+D}(γ)}} \right) && \text{Item~\ref{il:A3}} \\
                & ≤ 6d·\max\{d'_{K_X+D'}, 2+ε\}·δ(γ) + O\left( \textstyle{\sqrt{h_{K_X+D}(γ)}} \right) && \text{Theorem~\ref{thm:height1b} for $(X, D')$} \\
                & ≤ 6d·\max\{3d', 2+ε\}·δ(γ) + O\left( \textstyle{\sqrt{h_{K_X+D}(γ)}} \right)
  \end{align*}
\end{small}
Theorem~\ref{thm:height1a} is thus shown.  \qed

%
%
\svnid{$Id: 09-rigidity.tex 599 2021-08-23 12:06:26Z kebekus $}

\section{Rigidity theorem for $\cC$-integral points --- proof of Theorem~\ref*{thm:rigidity}}
\subversionInfo
\label{sec:orbirigidity}
\approvals{Arne & yes \\ Jorge & yes \\ Stefan & yes}

Our proof of Theorem~\ref{thm:rigidity} depends on the characteristic.  In
contrast to the proofs of Theorems~\ref{thm:height1a} and \ref{thm:height1b},
where inseparability is the main source of difficulties, here inseparability
greatly simplifies the argument.  No matter what the characteristic, we argue by
contradiction and assume that there exists a smooth, projective $k$-curve $T$
over $S$, a smooth, quasi-projective $k$-curve $H'$ and a non-constant morphism
$γ' : T⨯H' → X$ such that the induced morphisms $γ'_h : T → X$ are $S$-morphisms
and $\cC$-integral points over $S$, for all $h ∈ H'(k)$.

\begin{rem}
  Since $X$ is a surface and since the morphisms $γ'_h$ are $S$-morphisms, the
  assumption ``non-constant'' immediately implies that $γ'$ is dominant.
\end{rem}

\begin{notation}\label{n:diagram}
  Take $H$ to be the unique compactification of $H'$ to a smooth, projective
  $k$-curve.  The morphism $γ'$ extends to a rational map
  $γ : T⨯H \dasharrow X$, whose set of fundamental points is finite.  The
  following diagrams summarise the situation.
  \begin{equation}\label{eq:dt}
    \begin{aligned}
      \xymatrix{
        H & T⨯H \ar[l]_(.52){π_2} \ar@{-->>}[r]^(.55){γ} \ar[d]_{π_1} & X \ar[d]^{φ} \\
        & T \ar[r]_{ζ} & S }
    \end{aligned}
  \end{equation}
\end{notation}

\begin{rem}
  Since its indeterminacy locus is finite, $γ$ restricts to rational maps
  $γ_h : T \dasharrow X$, for every $h ∈ H(k)$.  But since $T$ is a $k$-smooth
  curve, these maps are in fact morphisms.
\end{rem}

\begin{asswlog}\label{ass:shrink}
  Shrinking $S°$, we may assume that $γ$ is well-defined over $S°$ and that for
  every $h ∈ H(k)$, the morphism $γ_h : T → X$ is either a $\cC$-integral point,
  or that its image is completely contained in the support of $D$.
\end{asswlog}

\subsection{Proof of Theorem~\ref*{thm:rigidity} in characteristic zero}
\approvals{Arne & --- \\ Jorge & yes \\ Stefan & yes}

In this section, we prove Theorem~\ref*{thm:rigidity} under the assumption that
the algebraically closed field $k$ has characteristic zero.  In particular, the
morphism $ζ$ is separable and generically étale.  Our goal is to construct a
splitting of Sequence~\eqref{eq:log6-4} over the generic fibre $X_η$ of $φ$,
contradicting Assumption~\ref{il:3-6-2a}.

\begin{asswlog}
  Shrinking $S°$, we assume that $ζ$ is étale over $S°$.
\end{asswlog}

\subsection*{Step 1.  Splitting of $γ^* Ω¹_X$}
\approvals{Arne & --- \\ Jorge & yes \\ Stefan & yes}

We begin with a discussion of $γ^* Ω¹_X$.  Over $S°$, where $γ$ is well-defined
and $ζ$ is étale, we show that the pull-back of the sequence of relative
differentials for the smooth morphism $φ$ splits.  In fact, a splitting over
$S°$,
\begin{equation}\label{eq:vbnv}
  \xymatrix{ %
    0 \ar[r] & %
    γ^* φ^* Ω¹_S \ar[r]_{γ^*dφ} & %
    γ^* Ω¹_X \ar[r]_{c} \ar@/_3mm/[l]_{a} & %
    γ^* ω_{X/S} \ar[r] \ar@/_3mm/[l]_{b} & %
    0, %
  }
\end{equation}
is easily constructed by taking $a$ to be the composition of the following
morphisms,
$$
\xymatrixcolsep{1.3cm}\xymatrix{ %
  γ^* Ω¹_X \ar[r]^{dγ} & %
  Ω¹_{T⨯H} \ar[r]^(.4){≅} & %
  π^*_1 Ω¹_T ⊕ π^*_2 Ω¹_H \ar[r]^(.6){\text{projection}} & %
  π^*_1 Ω¹_T \ar[r]^{≅ \text{ over } S°} & %
  γ^* φ^* Ω¹_S.  %
}
$$
The assumption that $ζ$ is étale guarantees that the morphisms $γ_h$ are
immersive over $S°$, because they remain immersive when composed with $φ$.  This
guarantees that the concatenation of the first three arrows,
$γ^* Ω¹_X → π^*_1 Ω¹_T$ is surjective.  It follows that $a$ is surjective.  We
can then take $b$ as the inverse of the restriction of $c$ to $\ker(a)$.  To
make use of this splitting, we need to relate it to the
Sequence~\eqref{eq:log6-4}, which involves log differentials rather than
differentials.  The following claim will turn out to be key.

\begin{claim}\label{claim:6-3}
  Over $S°$, we have
  $\Image(b) ⊆ \ker \bigl( γ^* Ω¹_X → γ^* Ω¹_{\supp D} \bigr)$.
\end{claim}
\begin{proof}[Proof of Claim~\ref{claim:6-3}]
  Since $c$ is surjective, we may as well prove that
  $$
  \Image \bigl( b◦c : γ^* Ω¹_X → γ^* Ω¹_X \bigr) ⊆ %
  \ker \bigl(γ^* Ω¹_X → γ^* Ω¹_{\supp D} \bigr).
  $$
  This can be shown locally, near any given point
  $\vec{p} ∈ (\supp γ^* D) ∩ T°⨯ H$.  In fact, given $\vec{p}$, write
  $\vec{x} = γ(\vec{p})$ and choose a local equation $f ∈ 𝒪_{X,\vec{x}}$ for
  $\supp D$, and a uniformising parameter $s ∈ 𝒪_{S,φ(\vec{x})}$.  Recall from
  Setting~\ref{setting:CKim} that $(X,D)$ is relatively snc over $S°$.  Near
  $\vec{x}$, the sheaf $Ω¹_X$ is therefore generated by the differential forms
  $df$ and $dφ(ds)$.  Near $\vec{p}$, the pull-back sheaf $γ^* Ω¹_X$ is then
  generated by the pull-back sections $γ^* df$ and $γ^* dφ(ds)$.  To prove the
  claim, it will then suffice to show that
  \begin{align}
    \label{eq:k1} (b◦c)\bigl(γ^* df\bigr) & ∈ \ker\bigl(γ^* Ω¹_X → γ^*
                                            Ω¹_{\supp D}\bigr) & \text{and}\\
    \label{eq:k2} (b◦c)\bigl(γ^* dφ(ds)\bigr) & ∈ \ker\bigl(γ^* Ω¹_X → γ^*
                                                Ω¹_{\supp D}\bigr).
  \end{align}
  Inclusion~\eqref{eq:k2} follows because $γ^* dφ(ds) ∈ \ker(c)$.
  Inclusion~\eqref{eq:k1} requires a little more thought.  To begin, observe
  that $df|_{\supp D} = 0 ∈ Ω¹_{\supp D}$, which implies that
  \begin{equation}\label{eq:p1a}
    γ^* df ∈ \ker \bigl(γ^* Ω¹_X → γ^* Ω¹_{\supp D} \bigr).
  \end{equation}
  Secondly, recall from Assumption~\ref{ass:shrink} that the divisor $γ^*(⌈D⌉)$
  intersects all horizontal curves $T ⨯ \{h\}$ with multiplicity at least $2$, or else
  contains these curves in its support.  Either way, it then follows from the
  construction of $a$ that near $\vec{p}$, the section $a(γ^* df)$ vanishes
  along the support of $γ^*D$.  By construction of $b$, this means that near
  $\vec{p}$, the sections $γ^* df$ and $(b◦c)\bigl(γ^* df\bigr)$ agree along
  $\supp γ^*D$.  Inclusion~\eqref{eq:p1a} will then show that the section of
  $γ^* Ω¹_{\supp D}$ that is induced by $(b◦c)\bigl(γ^* df\bigr)$, vanishes
  there, as desired.  \qedhere~(Claim~\ref{claim:6-3})
\end{proof}

\subsection*{Step 2.  Splitting of $γ^* Ω¹_X\bigl(\log ⌈D⌉\bigr)$}
\approvals{Arne & --- \\ Jorge & yes \\ Stefan & yes}

As a next step, we claim that Sequence~\eqref{eq:log6-4} splits over $S°$, once
one pulls it back via $γ$.  The following commutative diagram with exact rows
and columns summarises the situation, and relates Sequence~\eqref{eq:log6-4} to
Sequence~\eqref{eq:vbnv} discussed in the previous step.
$$
\small \xymatrix{ %
  γ^* \Bigl( 𝒥_{⌈D⌉} ⊗ φ^* Ω¹_S \Bigr) \ar@{^(->}[r] \ar@{^(->}[d] & %
  γ^* \Bigl( 𝒥_{⌈D⌉} ⊗ Ω¹_X\bigl(\log ⌈D⌉\bigr) \Bigr) \ar@{->>}[r] \ar@{^(->}[d]^{α} & %
  γ^* \Bigl( 𝒥_{⌈D⌉} ⊗ ω_{X/S}\bigl(\log ⌈D⌉\bigr) \Bigr) \ar@{^(->}[d]^{\text{isomorphic}} \\ %
  γ^* φ^* Ω¹_S \ar@{^(->}[r]^{γ^*dφ} \ar@{->>}[d] & %
  γ^* Ω¹_X \ar@{->>}[r]^{c} \ar@/^3mm/[l]^{a} \ar@{->>}[d]^{β} & %
  γ^* ω_{X/S} \ar@/^3mm/[l]^{b} \ar@{->>}[d] \\ %
  γ^* Ω¹_{\supp D} \ar@{^(->}_{\text{isomorphic}}[r] & %
  γ^* Ω¹_{\supp D} \ar@{->>}[r]& %
  0
}
$$

\begin{explanation}
  The first row in the diagram equals Sequence~\eqref{eq:log6-4} twisted with
  the ideal sheaf $𝒥_{⌈D⌉}$ and pulled back via $γ$.  The second row equals
  Sequence~\eqref{eq:vbnv}.  The middle column is a standard description of
  logarithmic differentials, as discussed for instance in \cite[Prop.~2.3(c) on
  p.~13]{EV92}.
\end{explanation}

We have seen in Claim~\ref{claim:6-3} that the composed map
$β ◦ b : γ^* ω_{X/S} → γ^* Ω¹_{\supp D}$ vanishes.  As a consequence, we obtain
a morphism $γ^* ω_{X/S} → γ^* \bigl( 𝒥_{⌈D⌉} ⊗ Ω¹_X\bigl(\log ⌈D⌉\bigr) \bigr)$,
and hence a splitting of the top row.  Since a sequence of coherent sheaves
splits if and only if it splits after tensoring with a locally free sheaf, we
can summarise our results so far as follows.

\begin{claim}\label{claim:gjkfhkj}
  The $γ$-pull-back of Sequence~\eqref{eq:log6-4},
  \begin{equation}\label{eq:hjkgh}
    0 → γ^* φ^* Ω¹_S → γ^* Ω¹_X\bigl(\log ⌈D⌉\bigr) → γ^* ω_{X/S}\bigl(\log
    ⌈D⌉\bigr) → 0
  \end{equation}
  splits over $T°$.  \qed
\end{claim}

\subsection*{Step 3.  Splitting of Sequence~\eqref{eq:log6-4}}
\approvals{Arne & --- \\ Jorge & yes \\ Stefan & yes}

We will now show that Sequence~\eqref{eq:log6-4} itself splits over the generic
fibre $X_η$ of $φ$.  Recall that $η$ denotes the generic point of $S$, and
denote the generic point of $T$ by $η_T$.  To begin, observe
that
\begin{align*}
  \deg_{η_T ⨯ H} γ^* φ^* Ω¹_S & = 0 \\
  \deg_{η_T ⨯ H} γ^* ω_{X/S}\bigl(\log ⌈D⌉\bigr) & > 0 && \text{Assumption~\ref{il:3-6-1a}.}
\end{align*}
It follows that the sheaf $γ^* Ω¹_X\bigl(\log ⌈D⌉\bigr)|_{η_T ⨯ H}$ is unstable.
Recalling that the pull-back of a semistable sheaf under separable morphisms of
curves remains semistable, \cite[Prop.~3.2]{Miyaoka87}, we infer that
$Ω¹_X\bigl(\log ⌈D⌉\bigr)|_{X_η}$ is likewise unstable.  Its maximal
destabilising subsheaf is then a saturated, invertible subsheaf
$𝒜 ⊂ Ω¹_X\bigl(\log ⌈D⌉\bigr)|_{X_η}$, of degree
\begin{equation}\label{eq:ghfg}
  \begin{aligned}
    \frac{1}{2}·\deg_{X_η} ω_{X/S} \bigl(\log ⌈D⌉\bigr) & < \deg_{X_η} 𝒜 && \text{unstable} \\
    & ≤ \deg_{X_η} ω_{X/S} \bigl(\log ⌈D⌉\bigr).  && \text{Sequence~\eqref{eq:hjkgh}}
  \end{aligned}
\end{equation}

\begin{claim}\label{cl:x5}
  The second inequality in \eqref{eq:ghfg} is indeed an equality.
\end{claim}
\begin{proof}[Proof of Claim~\ref{cl:x5}]
  Writing $ℬ$ for the quotient of $Ω¹_X\bigl(\log ⌈D⌉\bigr)|_{X_η}$ by $𝒜$, we
  obtain an exact sequence of locally free sheaves on $X_η$,
  $$
  0 → 𝒜 → Ω¹_X\bigl(\log ⌈D⌉\bigr)|_{X_η} → ℬ → 0.
  $$
  Assuming that we have strict inequalities, we obtain degree bounds
  \begin{align*}
    0 < \deg_{X_η} 𝒜 < \deg_{X_η} ω_{X/S} \bigl(\log ⌈D⌉\bigr) \\
    0 < \deg_{X_η} ℬ < \deg_{X_η} ω_{X/S} \bigl(\log ⌈D⌉\bigr)
  \end{align*}
  In particular, any morphism from $γ^* ω_{X/S}\bigl(\log ⌈D⌉\bigr)|_{η_T ⨯ H}$
  to the $γ$-pull-back of either $𝒜$ or $ℬ$ vanishes for degree reasons,
  contradicting Claim~\ref{claim:gjkfhkj}.  \qedhere~(Claim~\ref{cl:x5})
\end{proof}

Claim~\ref{cl:x5} implies in particular that the composed morphism
$$
𝒜 → Ω¹_X\bigl(\log ⌈D⌉\bigr)|_{X_η} → ω_{X/S}\bigl(\log ⌈D⌉\bigr)|_{X_η}
$$
is an isomorphism and that Sequence~\eqref{eq:log6-4} splits over $X_η$.  This
contradicts Assumption~\ref{il:3-6-2a} and therefore ends the proof of
Theorem~\ref*{thm:rigidity} in characteristic zero.  \qed

\subsection{Proof of Theorem~\ref*{thm:rigidity} over fields of positive characteristic}
\approvals{Arne & yes \\ Jorge & yes \\ Stefan & yes}

In this section, we prove Theorem~\ref{thm:rigidity} under the assumption that
$k$ is an algebraically closed field of positive characteristic.  The strategy
consists in deducing from Diagram~\eqref{eq:dt} the existence of an infinite
sequence of $\cC$-integral points over $S$ with bounded discriminant and
unbounded height, contradicting the geometric height inequalities for
$\cC$-integral points provided by Theorem~\ref{thm:height1a}.

\subsection*{Step 1.  Simplification}
\approvals{Arne & --- \\ Jorge & yes\\ Stefan & yes}

Given any number $m ∈ ℕ^{≥ 2}$, set
$$
D^m := D - {\textstyle\frac{1}{m}}·D_{\log}.
$$
Observe that $\supp D = \supp D^m$, that $(X, D^m)$ a $\cC$-pair and that all
$\cC$-integral points of $(X, D)$ are automatically $\cC$-integral points of
$(X, D^m)$.  Choosing $m$ sufficiently large, Assumption~\ref{il:3-6-1a} of
Theorem~\ref{thm:rigidity} is still satisfied for the pair $(X, D^m)$.
Replacing $D$ by $D^m$, we are thus safe to make the following assumption.

\begin{asswlog}\label{ass:no log}
  The divisor $D$ has no component with $\cC$-multiplicity equal to $∞$.  In
  other words, $D_{\log}= 0$.
\end{asswlog}

\subsection*{Step 2.  Iterated Frobenius morphisms}
\approvals{Arne & yes \\ Jorge & yes \\ Stefan & yes}

Given any number $n ∈ ℕ$, let $F_n : H_n → H$ be the $n$th iterated $k$-linear
Frobenius morphism, as discussed in \cite[IV, Rem.~2.4.1]{Ha77} or
\cite[\href{https://stacks.math.columbia.edu/tag/0CC9}{Tag
  0CC9}]{stacks-project}.  The $H_n$ are then $k$-algebraic curves, with genus
$g(H_n) = g(H)$, cf.~\cite[IV, Prop.~2.5]{Ha77}.

Choose divisors $A_H ∈ \Div(H)$ and $A_T ∈ \Div(T)$ with degrees
$\deg_• A_• ≥ 2g(•)+1$.  For every number $n$, set
$A_{H_n} := p^{-n}·F_n^* A_H ∈ \Div(H_n)$, which is a divisor of degree
$\deg_{H_n} A_{H_n} = \deg_H A_H$.  By \cite[IV, Cor.~3.2.b]{Ha77}, this
assumption implies that the divisors $A_H$, $A_T$ and $A_{H_n}$ are all very
ample, and then by \cite[Prop.~4.4.10.iv]{EGA4-2} so are the divisors
$$
A_{T⨯ H_n} := π^*_1 A_T ⊗ π^*_2 A_{H_n} ∈ \Div(T ⨯ H_n), \quad
\text{for every } n.
$$
For every number $n$, let $C_n ∈ |A_{T⨯ H_n}|$ be a general section.  We
summarise some of its main properties.
\begin{enumerate}
\item Following \cite[Cor.~12]{MR0360616} or \cite[Satz~5.2]{MR0460317}, the
  scheme $C_n$ is a regular curve, hence smooth over the perfect field $k$.

\item The curve $C_n$ avoids the (finite) indeterminacy locus of the composed
  morphism $γ◦(\Id_T⨯F_n) : T⨯H_n \dasharrow X$.

\item\label{il:g3} The natural morphism of curves $C_n → T$ is separable, since
  the very ample linear system $|A_{T⨯ H_n}|$ separates tangents and the
  projection map $T⨯H_n → T$ is smooth of relative dimension one.

\item\label{il:g4} By construction, the degree of $C_n$ over $S$ equals
  $$
  [C_n:S] = [C_n:T]·[T:S] = \deg_H A_H·[T:S]
  $$
  and is therefore independent of $n$.
\end{enumerate}

\begin{notation}
  The following diagram extends \eqref{eq:dt}, summarises the situation, and
  fixes the notation used to denote the relevant morphisms.
  \begin{equation*}
    \begin{aligned}
      \xymatrix@C=2.5cm{
        C_n \ar[r]^{ι_n}_{\text{inclusion}} \ar@/^8mm/[rrr]^{γ_n} \ar@/_8mm/[rr]_{ψ_n} & T⨯H_n \ar[r]^{η_n := \Id_T ⨯ F_n} & T⨯H \ar@{-->>}[r]_{γ} \ar[d]_{π_1} & X \ar[d]^{φ} \\
        && T \ar[r]_{ζ} & S }
    \end{aligned}
  \end{equation*}
  By minor abuse of notation, we denote the projections from $T⨯H_n$ by
  $\overline{π}_{•}$, as it will always be clear from the context what number
  $n$ is meant.
\end{notation}

\subsection*{Step 3.  The curves $γ_n$ as $\cC$-integral points}
\approvals{Arne & --- \\ Jorge & yes \\ Stefan & yes}

We will show that the curves $γ_n$ are $\cC$-integral points for the family
$φ : X → S$.  For the reader's convenience, we prove birationality first.

\begin{claim}\label{claim:8-2a}
  Given any number $n ∈ ℕ$, the morphism $γ_n$ is generically injective.
\end{claim}
\begin{proof}[Proof of Claim~\ref{claim:8-2a}]
  The very ample linear system $|A_{T ⨯ H_n}|$ separates points.  In particular,
  if $p_1 ∈ T ⨯ H_n$ is a general point with set-theoretic fibre
  \[
    (γ◦ η_n)^{-1}(γ◦ η_n)(p_1) = \{ p_1, …, p_{ℓ} \} ⊂ T ⨯ H_n,
  \]
  then there are sections
  $σ_2, …, σ_{ℓ} ∈ H⁰\bigl(T ⨯ H_n,\, 𝒪_{T ⨯ H_n}(A_{T ⨯ H_n}) \bigr)$ where
  $σ_i(p_1) = 0$ and $σ_i(p_i) ≠ 0$ for every $2 ≤ i ≤ ℓ$.  Since the
  algebraically closed base field $k$ is infinite, there exists a $k$-linear
  combination $σ ∈ H⁰\bigl(T ⨯ H_n,\, 𝒪_{T ⨯ H_n}(A_{T ⨯ H_n}) \bigr)$ that
  vanishes at $p_1$, but not at any of the remaining points $p_2, …, p_{ℓ}$.  If
  $C ∈ |A_{T ⨯ H_n}|$ is general among the elements of the linear system that
  contain $p_1$, then $C$ will not contain any of the points $p_2, …, p_{ℓ}$,
  and the $γ_n$ maps the curve $C$ generically injectively onto its image.  The
  claim follows.  \qedhere~(Claim~\ref{claim:8-2a})
\end{proof}

\begin{claim}\label{claim:8-2}
  Given any number $n ∈ ℕ$, the morphism $γ_n$ maps the curve $C_n$ birationally
  onto its image.
\end{claim}
\begin{proof}[Proof of Claim~\ref{claim:8-2}]
  In view of Claim~\ref{claim:8-2a}, it remains to show that the morphism
  $C_n → γ_n(C_n)$ is separable, or equivalently, that the pull-back map of
  differential forms, $d γ_n : γ_n^* Ω¹_X → Ω¹_{C_n}$, does not vanish
  identically.  The proof relies on two observations.

  \begin{itemize}
  \item First, recalling that the morphisms $γ'_h : T → X$ are assumed to be
    $\cC$-integral points over $S$ for all closed points $h ∈ H'$, the following
    composed morphism between invertible sheaves on
    $T⨯H ∖ (\text{indet.~locus of $γ$})$ does not vanish identically:
    $$
    \xymatrix@C=1.5cm{
      γ^* Ω¹_X \ar[r]^{dγ} & Ω¹_{T ⨯ H} \ar[r]^(.4){≅} & π_1^* Ω¹_T ⊕ π_2^* Ω¹_H \ar[r]^(.6){\text{projection}} & π_1^* Ω¹_T.
    }
    $$
    By general choice, the pull-back of this map to $C_n$, denoted by
    $α : γ_n^* Ω¹_X → ψ_n^* π_1^* Ω¹_T$, will then likewise not vanish.

  \item Second, recall from \ref{il:g3} that the map
    $\overline{π}_1◦ι_n : C_n → T$ is separable.  The composed morphism between
    invertibles on $C_n$,
    $$
    \xymatrix@C=1.5cm{ %
      ι_n^* \overline{π}_1^* Ω¹_T \ar[r]_(.4){\Id ⊕ 0} \ar@/^5mm/[rrr]^{β} &
      ι_n^* \overline{π}_1^* Ω¹_T ⊕ ι_n^* \overline{π}_2^* Ω¹_{H_n}
      \ar[r]_(.6){≅} & ι_n^* Ω¹_{T ⨯ H_n} \ar[r]_{dι_n} & Ω¹_{C_n} }
    $$
    will thus again not vanish identically.
  \end{itemize}
  With these two observations in place, the map $d γ_n$ can now be rewritten as
  $$
  \small \xymatrix@C=1.5cm{
    γ_n^* Ω¹_X \ar[r]_{ψ_n^* dγ} \ar@/^5mm/[rrr]^{dγ_n} \ar@{=}[d] & ψ_n^* Ω¹_{T ⨯ H} \ar[r]_{ι_n^* d η_n} \ar[d]^{≅} & ι_n^* Ω¹_{T⨯ H_n} \ar[r]_{dι_n} \ar[d]^{≅} & Ω¹_{C_n} \ar@{=}[d] \\
    γ_n^* Ω¹_X \ar[r]_(.37){α ⊕ \text{(other)}} & ψ_n^* π_1^* Ω¹_T ⊕ ψ_n^* π_2^*
    Ω¹_H \ar[r]_{\text{(Isom.)} ⊕ 0} & ι_n^* \overline{π}_1^* Ω¹_T ⊕ ι_n^*
    \overline{π}_2^* Ω¹_{H_n} \ar[r]_(.62){β + \text{(other)}} & Ω¹_{C_n}.  }
  $$
  Recalling that $ψ_n^* π_1^* Ω¹_T$, $ι_n^* π_1^* Ω¹_T$ and $Ω¹_{C_n}$ are
  invertible, so that any composition of non-vanishing morphisms is itself
  non-vanishing, a look at the bottom row will convince the reader.
  \qedhere~(Claim~\ref{claim:8-2})
\end{proof}

\begin{claim}\label{claim:8-3}
  For every sufficiently large number $n$, the curve $γ_n$ is a $\cC$-integral
  point for the family $φ : X → S$, in the standard sense of
  Definition~\ref{def:orbicurve}.
\end{claim}
\begin{proof}[Proof of Claim~\ref{claim:8-3}]
  Recall from Assumption~\ref{ass:no log} that no component of the divisor $D$
  has $\cC$-multiplicity equal to $∞$.  We will show the morphisms $γ_n$ are
  $\cC$-integral points whenever
  \begin{equation}\label{eq:izi}
    p^n ≥ \max \{ m_i\:|\: 1 ≤ i ≤ d \}.
  \end{equation}
  Assuming that one such $n$ is given, we need to show that $γ_n$ satisfies the
  conditions spelled out in Definition~\vref{def:orbicurve}.
  Condition~\ref{il:3-3-1} (``the curve $C_n$ dominates $S$'') is clear by
  construction, and Condition~\ref{il:3-3-3} (``the morphism $γ_n$ is birational
  onto is image'') has been verified in Claim~\ref{claim:8-2} above.  It remains
  to verify Condition~\ref{il:3-3-2}: the obvious restriction of $γ$ to the
  preimages of $S°$, which we write as $γ°_n : C°_n → X°$, is a $\cC$-curve for
  the pair $(X°, D°)$.

  To this end, let $x ∈ C°_n$ be any closed point such that $γ_n(x)$ belongs to
  the support of $D°$.  Using the assumption that $D°$ is relatively snc over
  $S°$, observe that the point $γ_n(x)$ is then contained in exactly one
  component of $D$.  Write $m$ for the $\cC$-multiplicity of that component and
  recall from Assumption~\ref{ass:no log} that $m < ∞$.  To show
  Condition~\ref{il:3-3-2}, we need to show that the multiplicity of $D$ at the
  point $x$ is at least $m$, that is, $\mult_x γ_n^*D ≥ m$.  This will be done by
  means of a local computation.  Write $(a, b) := (η_n◦ι_n)(x) ∈ T⨯H$ and choose
  uniformising parameters $t ∈ 𝒪_{T, a}$ and $h ∈ 𝒪_{H, b}$.  Abusing notation,
  we use the symbols $t$ and $h$ to also denote the associated elements of the
  local ring $𝒪_{T ⨯ H, (a, b)}$.  We claim that the following inclusion of
  ideals holds true:
  \begin{equation}\label{eq:gf}
    𝒥_{γ^* D} ⊂ 𝒥_{γ^* D \: ∩ \: (T⨯\{b\})} \overset{!}{⊆} (h, t^m) \quad \text{in } 𝒪_{T ⨯ H, (a, b)}.
  \end{equation}
  In fact, if $b ∈ H'$ and if $γ_n$ is not $\cC$-integral, then
  Inclusion~\eqref{eq:gf} follows from the assumption that $γ_b: T → X$ is a
  $\cC$-integral point.  If on the other hand $b ∉ H'$, then
  Inclusion~\eqref{eq:gf} follows from Assumption~\ref{ass:shrink}, which
  asserts that $(T⨯\{b\}) ⊆ \supp γ^* D$.  But now we are done: observing that
  $$
  \ord_x (η_n◦ι_n)^* h \:≥\: p^n \overset{\text{\eqref{eq:izi}}}{≥} m \qquad
  \text{and} \qquad \ord_x (η_n◦ι_n)^* t^m \:≥\: m,
  $$
  the claim follows directly from \eqref{eq:gf}.
  \qedhere~(Claim~\ref{claim:8-3})
\end{proof}

\subsection*{Step 4.  Conclusion}
\approvals{Arne & yes \\ Jorge & yes \\ Stefan & yes}

We conclude by showing that the family $γ_n$ of $\cC$-integral points has
bounded discriminant but unbounded height, contradicting
Theorem~\ref{thm:height1a}.

\begin{claim}\label{claim:8-1}
  The discriminants of the $(γ_n)_{n ∈ ℕ}$ are constant.
\end{claim}
\begin{proof}[Proof of Claim~\ref{claim:8-1}]
  Given $n ∈ ℕ$, choose closed points $t ∈ T(k)$ and $h_n ∈ H_n(k)$, with fibres
  $F_t$ and $F_{h_n}$ in $T⨯H_n$.  We have equalities of numerical classes,
  \begin{align*}
    [C_n] & \equiv_{\num} (\deg_T A_T)·[F_t] + \deg_H A_H·[F_{h_n}] \\
    [K_{T⨯H_n}] & \equiv_{\num} (2·g(T)-2)·[F_t] + (2·g(H)-2)·[F_{h_n}] \\
    \intertext{and therefore}
    \deg_{C_n} ω_{C_n} & = [C_n]·\bigl[C_n + K_{T ⨯ H_n} \bigr] = \const^+ ∈ ℕ^+.
  \end{align*}
  Item~\ref{il:g4} will thus conclude the proof.
  \qedhere~(Claim~\ref{claim:8-1})
\end{proof}

\begin{claim}\label{claim:8-4}
  The heights of the $(γ_n)_{n ∈ ℕ}$ are unbounded.  More precisely, we
  have
  $$
  \lim_{n → ∞} h(γ_n) = \lim_{n → ∞} \frac{\deg_{C_n} γ_n^*
    (K_{X/S}+D)}{[C_n:S]} = ∞.
  $$
\end{claim}
\begin{proof}[Proof of Claim~\ref{claim:8-4}]
  As before, choose $t ∈ T(k)$ and $h ∈ H(k)$, with fibres $F_t$ and $F_h$ in
  $T⨯H$ and write
  $$
  \bigl[γ^* (K_{X/S}+D)\bigr] \equiv_{\num} a·[F_t] + b·[F_h]
  $$
  where $a$ and $b$ are rational numbers.  Assumption~\ref{il:3-6-1a} guarantees
  that $b > 0$.  But then,
  \begin{align*}
    \deg_{C_n} γ_n^* (K_{X/S}+D) & = [A_{T ⨯ H_n}]·\bigl[η_n^*γ^*(K_{X/S}+D)\bigr] \\
    & = (\deg_T A_T)·p^n·b + (\deg_H A_H)·a
  \end{align*}
  As before, Item~\ref{il:g4} concludes the proof.
  \qedhere~(Claim~\ref{claim:8-4})
\end{proof}

To sum up, assuming that there exists a positive family of $\cC$-integral
points, we found a sequence of $\cC$-integral points of bounded discriminant but
unbounded height.  This contradicts the height bound found in
Theorem~\ref{thm:height1a} above and therefore ends the proof of
Theorem~\ref{thm:rigidity} in the last remaining case, when the characteristic
of the base field is positive.  \qed

%
%
\svnid{$Id: 10-mordell.tex 599 2021-08-23 12:06:26Z kebekus $}

\section{The Mordell conjecture for $\cC$-integral points --- proof of Theorem~\ref*{thm:Mordell}}
\subversionInfo
\label{sec:orbimordell}
\approvals{Arne & yes \\ Jorge & yes \\ Stefan & yes}

With Theorems~\ref{thm:height1a} (``Height bound'') and \ref{thm:rigidity}
(``Rigidity'') at our disposal, Theorem~\ref{thm:Mordell} follows quickly,
adapting standard arguments to our setting.  We argue by contradiction:
maintaining the setting and assumptions of Theorem~\ref{thm:Mordell}, we assume
that there exists a smooth $k$-curve $T$ over $S$ and an infinite number of
$\cC$-integral points $γ ∈ \Hom_S(T, X)(k)$.

\subsection{Step 1: the set of $\cC$-integral points}
\approvals{Arne & --- \\ Jorge & yes \\ Stefan & yes}

The following claim asserts that the set of $\cC$-integral points,
$$
\Hom_S(T, X, \cC) := \bigl\{ γ ∈ \Hom_S(T, X)(k) \:|\: γ \text{ is
  $\cC$-integral} \bigr\},
$$
is locally closed.  The proof uses little but the definition of
``$\cC$-integral'' and the standard fact that effective Cartier divisors on the
smooth curve $T$ are parameterised by the Hilbert scheme, as discussed in
\cite[\href{https://stacks.math.columbia.edu/tag/0B9C}{Sect.~0B9C}]{stacks-project}
or \cite[I.~Thm.~1.13]{K96}.

\begin{claim}\label{claim:9-1}
  The set $\Hom_S(T, X, \cC)$ is a locally closed subset of
  $\Hom_S(T, X)(k)$.
\end{claim}
\begin{proof}[Proof of Claim~\ref{claim:9-1}]
  Considering one component of $D$ at a time, we may assume without loss of
  generality that the divisor $D ∈ \Div(X)$ is irreducible, so
  $D = \frac{m_1-1}{m_1}·D_1$ with $m_1 ∈ ℕ^+∪ \{∞\}$.  Write
  $T := T° ∪ \{t_1, …, t_a \}$.  Now, given any component
  $H ⊆ \Hom_S(T, X)_{\red}$, we need to show that the set
  $$
  H_{\cC} := \{ γ ∈ H(k) \:|\: γ \text{ is $\cC$-integral} \}
  $$
  is locally closed in $H(k)$.  To begin, remark that
  $$
  H_{\cC} ⊆ \{ γ ∈ H \:|\: \Image(γ) ⊄ \supp D \} := H°,
  $$
  where $H° ⊆ H$ is open.  We will show that $H_{\cC}$ is a closed subset of
  $H°(k)$.

  Next, choose any morphism $γ ∈ H(k)$ and set $b_H := \deg_T γ^* D_1$.  This
  number is independent of the choice of $γ$.  Denoting the universal morphism
  by $u° : H°⨯T → X$, the pull-back $(u°)^* D_1$ is a relative effective Cartier
  divisor for the family $H°⨯T → T$; we refer the reader to
  \cite[\href{https://stacks.math.columbia.edu/tag/056P}{Section
    056P}]{stacks-project} for the definition of ``relative effective Cartier''
  and to
  \cite[\href{https://stacks.math.columbia.edu/tag/062Y}{Lem.~062Y}]{stacks-project}
  for the criterion used here.  By
  \cite[\href{https://stacks.math.columbia.edu/tag/0B9C}{Section
    0B9C}]{stacks-project}, this divisor yields a morphism
  $ν : H° → \Hilb^{b_H}_{T/k}$.  We aim to describe $H_{\cC}$ in terms of this
  morphism.  To this end, we consider sequences of numbers as follows,
  \begin{multline*}
    M := \Bigl\{ (n_1, …, n_{a+b_H}) ∈ ℕ^{a+b_H} \:\Bigl|\: \sum_j n_j = b_H \text{ and for all $i>a$, } \\
    \text{we have either $n_i = 0$ or $n_i > m_1$ } \Bigr\}.
  \end{multline*}
  Here, the inequality $n_i > m_1$ is understood to be never satisfied when
  $m_1 = ∞$.  If any sequence $\vec{n} = (n_1, …, n_{a+b_H}) ∈ M$ is given, we
  consider the relative effective Cartier divisor on
  $π_2 ⨯ ⋯ ⨯ π_{a+b_H+1} : T^{⨯ (a+b_H+1)} → T^{⨯ (a+b_H)}$ given as
  $$
  D_{\vec{n}} := \sum_{i=1}^{a+b_H} n_i·\Bigl[\bigl\{ (t_0, …, t_{a+b_H}) ∈ T^{⨯
    (a+b_H+1)} \:|\: t_0 = t_i \bigr\}\Bigr].
  $$
  As before, $D_{\vec{n}}$ defines a morphism
  $u_{\vec{n}} : T^{⨯ (a+b_H)} → \Hilb^{b_H}_{T/k}$.  The morphism $u_{\vec{n}}$
  is proper; its image is therefore closed and then so is the union
  $$
  \cC\Hilb^{b_H}_{T/k} := \bigcup_{\vec{n} ∈ M} \Image(u_{\vec{n}}).
  $$
  Returning to the original problem of describing $H_{\cC}$, it follows
  immediately from the construction that a point $γ ∈ H(k)$ is in $H_{\cC}$ if
  and only if $ν(γ) ∈ \cC\Hilb^{b_H}_{T/k}$.  The set $H_{\cC}$ is therefore
  closed in $H°(k)$.~\qedhere~(Claim~\ref{claim:9-1})
\end{proof}

\subsection{Step 2: boundedness, end of proof}
\approvals{Arne & yes \\ Jorge & yes \\ Stefan & yes}

Theorem~\ref{thm:height1a} implies that the height of all $\cC$-integral points
$γ : T → X$ is bounded:
$$
∃ \const^+ : \forall γ ∈ \Hom_S(T, X, \cC) : \deg_T γ^*(K_{X/S}+D) <
\const^+.
$$
Using that $K_{X/S}+D$ is relatively ample and that $\Hom_S(T, X)$ is an open
subscheme of $\Hilb_{T⨯_SX/S}$, we find finitely many irreducible components of
$\Hom_S(T, X)$ that contain all of $\Hom_S(T, X, \cC)$.  In particular, there
exists one irreducible component of $\Hom_S(T, X, \cC)$ that contains infinitely
many points and must therefore be of positive dimension.  Eventually, this
allows us to find a quasi-projective curve $C° ⊆ \Hom_S(T, X)$ and an $S$-morphism
$γ° : C° ⨯ T → X$ such that the morphisms $γ°_c : T → X$ are $\cC$-integral, for
all $c ∈ C°(k)$.  Theorem~\ref{thm:rigidity}, however, asserts that no such
morphism can possibly exist.  We obtain a contradiction, which ends the proof of
Theorem~\ref{thm:Mordell}.  \qed

\part{Insufficiency of the Brauer--Manin obstruction for a simply connected fourfold}
%
%
\svnid{$Id: 11-fibrations.tex 599 2021-08-23 12:06:26Z kebekus $}

\section{Fibrations of general type --- proof of Theorem~\ref*{maintheorem2}}
\label{section:fibrations}
\subversionInfo
\approvals{Arne & yes \\ Jorge & yes\\ Stefan & yes}

The goal of this section is to prove Theorem~\ref{maintheorem2}.  The case of
number fields has been treated by Campana in \cite[§5]{Cam05} conditionally on
the $abc$ conjecture.  Our construction proceeds along similar lines -- we use
our orbifold Mordell-type theorem in positive characteristic as a substitute for
the $abc$ conjecture, together with a delicate construction of genus-two
fibrations on certain simply connected surfaces with orbifold base of general
type, due to Stoppino \cite{Stoppino}.  Her construction is simpler than the one
used by Campana in \cite[§5]{Cam05}, and therefore has the advantage of being
more easily transportable to positive characteristic.

\subsection{The orbifold base}
\approvals{Arne & yes \\ Jorge & yes \\ Stefan & yes}

For fibrations between algebraic varieties over the complex numbers, Campana
defines the notion of \emph{orbifold base} in \cite[§1.1.4, §1.2.1]{Cam04}.  Let
us recall the definition.

\begin{defn}[Orbifold base]\label{definition:orbifoldbase}
  Let $k$ be an arbitrary field and let $π: Y → X$ be a surjective morphism of
  smooth, quasi-projective $k$-varieties.  Assume that $π$ is smooth over an
  open subset of $X$.  For each $P$ in $X^{(1)}$, the set of points of
  codimension $1$ on $X$, we define $m_P$ as the \emph{minimum} of the set of
  multiplicities of the irreducible components of the fibre $π^{-1}(P)$.  We
  then set
  $$
  Δ_π := \sum_{P ∈ X^{(1)}} \Bigl( 1-\frac{1}{m_P}\Bigr)·\overline{P}.
  $$
  where $\overline{P}$ denotes the Zariski closure of $P$.  We refer to
  $(X,Δ_π)$ as the \emph{orbifold base of $π$}.
\end{defn}

\begin{rem}
  The support of $Δ_π$ in Definition~\ref{definition:orbifoldbase} need not be SNC.
\end{rem}

\begin{rem}
  Definition~\ref{definition:orbifoldbase} considers only codimension-one points
  of $X$ and ignores the behaviour of $f$ over points of higher codimension.
  While more elaborate definitions of ``orbifold base'' have been suggested, the
  simple version described above is good enough for our purposes.
\end{rem}

\begin{lem}[Orbifold base and $\cC$-curves]\label{lem:orbibase}
  In the setting of Definition~\ref{definition:orbifoldbase}, assume that $k$ is
  algebraically closed.  Let $D ≤ Δ_π$ be any $ℚ$-divisor on $X$, such that
  $(X, D)$ forms an snc $\cC$-pair.  If $T$ is a smooth, quasi-projective
  $k$-curve equipped with a morphism $γ : T → Y$, then either
  $\Image (π◦γ) ⊂ \supp D$, or $π◦γ : T → X$ is a $\cC$-curve for the $\cC$-pair
  $(X,D)$.
\end{lem}
\begin{proof}
  If suffices to consider the case where the support of $D$ is irreducible.  Let
  us assume that $\Image(π◦γ) ⊄ \supp D$, and let $t ∈ T(k)$ be such that
  $(π◦γ)(t) ∈ \supp D$.  To show that $π◦γ$ is a $\cC$-curve, we need to check
  that
  $$
  \mult_t (π◦γ)^*⌈D⌉ ≥ \min(\text{multiplicities of irreducible components of
  }π^*⌈D⌉).
  $$
  This is however immediate from the definition of orbifold base once we observe
  that
  $$
  \mult_t (π◦γ)^*⌈D⌉ = \mult_t γ^*(π^* ⌈D⌉).  \eqno\qedhere
  $$
\end{proof}

\subsection{Multiple fibres}
\approvals{Arne & yes \\ Jorge & yes \\ Stefan & yes}

In order to guarantee that the variety constructed in Section~\ref{ssec:pf14} is
geometrically simply connected, we use an elementary criterion which generalises
\cite[Lem.~5.8]{Cam05}.  The following definition will be used.

\begin{defn}[Multiple fibres]\label{definition:multfibr}
  Let $k$ be an algebraically closed field.  Let $π : X → C$ be a projective,
  surjective $k$-morphism from a normal $k$-variety to a smooth $k$-curve.
  Assume that $π$ has connected fibres.  Given $P ∈ C(k)$, say that \emph{$π$
    has a \emph{multiple fibre} over $P$} if there exist an integer $m ≥ 2$ and
  a Weil divisor $D$ on $X$ such that $π^* P = m·D$.
\end{defn}

\begin{lem}[Multiple fibres and simple connectedness]\label{lem:fundamentalgroups} %
  In the setting of Definition~\ref{definition:multfibr}, fix a base point
  $x ∈ X(k)$.  Assume that $π$ does not have any multiple fibres and that at
  least one fibre of $π$ is simply connected.  Then, the natural
  homomorphism
  $$
  π_1^{\emph{ét}}(X,x) → π_1^{\emph{ét}}\bigl(C,π(x) \bigr)
  $$
  is an isomorphism.
\end{lem}
\begin{proof}
  Let $f: X' → X$ be a finite, étale cover.  We have to show that $f$ is the
  base change along $π$ of a finite étale cover of the base curve $C$.  Taking
  the Stein factorisation of $π◦f$, we obtain the commutative diagram,
  $$
  \xymatrix{%
    X' \ar[rr]^{f\text{, étale}} \ar[d]_{\txt{\scriptsize $π'$, connected\\\scriptsize fibres}} && X \ar[d]^{π} \\
    C' \ar[rr]_{g\text{, finite}} && C.
  }
  $$
  Since $π$ has a simply connected fibre, we see that $g$ must be étale above at
  least one point of $C$, and hence also generically.  We claim that $g$ is
  étale everywhere.  If fact, if $c' ∈ C'(k)$ is any closed point with image
  $c ∈ C$, then the coefficients of $π^* c ∈ \Div(X)$ are coprime; since $f$ is
  étale, so are the coefficients found in any connected component of
  $f^* π^* c$.  On the other hand, if $B'$ is any connected component of
  $(π')^*g^* c = f^*π^* c$ that maps to $c'$, then all coefficients in $B'$ are
  multiples of $\mult_{c'} g^* c$.  Therefore, $\mult_{c'} g^* c = 1$, so $g$ is
  étale at $c'$.  It is now easy to see that the natural morphism
  $X' → X ⨯_C C'$ is an isomorphism, as required.
\end{proof}

\subsection{Proof of Theorem~\ref*{maintheorem2}}
\label{ssec:pf14}
\CounterStep
\approvals{Arne & yes \\ Jorge & yes \\ Stefan & yes}

We begin directly with the construction of the example.  A second step will show
that the example does indeed satisfy all required properties.  Throughout, we
view $ℙ¹$ as an extension of $𝔸¹$ and write $t$ instead of $[t:1]$ and $∞$
instead of $[1:0]$.

\subsection*{Step 1.  Construction}
\approvals{Arne & --- \\ Jorge & yes \\ Stefan & yes}

We start by constructing a smooth, projective $ℚ$-surface $F_ℚ$ and a
$ℚ$-morphism $g_ℚ: F_ℚ → ℙ¹_ℚ$ with the following properties.
\begin{enumerate}
\item\label{il:x1a} The $ℚ$-surface $F_ℚ$ is smooth, of general type, and $g_ℚ$
  is smooth over an open subset of $ℙ¹_ℚ$.

\item\label{il:x2a} The natural morphism $𝒪_{ℙ¹_ℚ} → (g_ℚ)_* 𝒪_{F_ℚ}$ is an
  isomorphism.

\item\label{il:x3a} No geometric fibre of $g_ℚ$ is multiple.

\item\label{il:x4a} The fibre $g_ℚ^{-1}\bigl( [1:1] \bigr)$ has a $ℚ$-rational point $x_ℚ$.

\item\label{il:x5a} The geometric fibres of $g_ℚ$ over $[0:1]$ and $[1:0]$ are
  supported on trees of smooth rational curves with transverse intersections,
  and each of the irreducible components has multiplicity at least two.
\end{enumerate}

The construction outlined below is due to Stoppino.  We refer the reader to
Stoppino's paper \cite{Stoppino}, in particular to \cite[Fig.~2 as well as
Thm.~3.1, Rem.~3.6]{Stoppino} for full details, proofs and instructive sample
computations.

\begin{construction}[Special case of Stoppino's construction with ``fibres of type $1$'']
  Fix the rational number $α := -\frac{793}{216}$ and consider the following
  curve in $ℙ¹_ℚ⨯ℙ¹_ℚ$,
  \[
    B := \Bigl\{ \bigl([x:y],[t:s]\bigr) ∈ ℙ¹_ℚ⨯ℙ¹_ℚ \::\: st·\bigl(t²·x⁶ +
    α·st·x³y³ + s²·y⁶\bigr) = 0 \Bigr\}.
  \]
  Let $Y_ℚ$ be the canonical desingularisation of the double covering of
  $ℙ¹_ℚ⨯ℙ¹_ℚ$ ramified over $B$.  As shown in \cite[Section 2 and proof of
  Thm.~3.1]{Stoppino}, the relative minimal model of the composition
  \[
    Y_ℚ → ℙ¹_ℚ⨯ℙ¹_ℚ \xrightarrow{2^{\text{nd}} \text{ projection}} ℙ¹_ℚ
  \]
  is a smooth surface $Z_ℚ$, and the natural morphism $f_ℚ : Z_ℚ → ℙ¹_ℚ$ is a
  fibration satisfying Properties \ref{il:x2a}, \ref{il:x3a}, and \ref{il:x5a}.
\end{construction}

The specific numerical value of $α$ implies the following properties for the
$f_ℚ$-fibres over $[± 1:1]$.\CounterStep

\begin{enumerate}
\item\label{il:11.8.1} The point $\bigl( [2:3], [1:1] \bigr) ∈ ℙ¹_ℚ⨯ℙ¹_ℚ$ lies
  on the curve $B$.  This implies in particular that the $f_ℚ$-fibre over
  $[1:1]$ contains a $ℚ$-rational point.  In other words, the fibration
  $f_ℚ : Z_ℚ →ℙ¹_ℚ$ satisfies Property~\ref{il:x4a} as well.

\item\label{il:11.8.2} The curves $B$ and $ℙ¹_ℚ ⨯ \bigl\{[± 1:1]\bigr\}$
  intersect transversely $ℙ¹_ℚ⨯ℙ¹_ℚ$.  In particular, we find that the
  $f_ℚ$-fibres over $[± 1:1]$ are smooth, so that $f_ℚ$ is smooth over
  neighbourhoods of the points $[± 1:1]$.
\end{enumerate}

To construct a surface $F_ℚ$ and a $ℚ$-morphism $g_ℚ: F_ℚ → ℙ¹_ℚ$ which also
satisfy Property~\ref{il:x1a}, consider the $ℚ$-morphism
\[
  γ: ℙ¹_ℚ → ℙ¹_ℚ, \quad [t:s] ↦ \bigl[(t+s)³ + (t-s)³ : (t+s)³ -
  (t-s)³\bigr],
\]
and perform the base change
\[
  \xymatrixcolsep{2.5cm}\xymatrix{ %
    F_ℚ \ar[r] \ar[d]_{g_ℚ} & Z_ℚ \ar[d]^{f_ℚ} \\
    ℙ¹_ℚ \ar[r]_{γ} & ℙ¹_ℚ.  %
  }
\]
We check that $g_ℚ : F_ℚ → ℙ¹_ℚ$ satisfies all required properties.
\begin{description}
\item[Property~\ref{il:x1a}] The morphism $γ$ is totally ramified over the
  points $[± 1:1]$ and étale elsewhere.  Smoothness of $Z_ℚ$ and
  Item~\ref{il:11.8.2} will therefore imply that $F_ℚ$ is smooth, and that $f_ℚ$
  is smooth near $γ^{-1}([± 1:1])$.  We refer the reader to
  \cite[Prop.~3.2]{Stoppino} or \cite[Prop.~1.7]{Cam05} for the fact that $F_ℚ$
  is of general type.

\item[Property~\ref{il:x2a}] This follows from flatness-and-base change, since
  $γ$ is flat and since $𝒪_{ℙ¹_ℚ} → (f_ℚ)_* 𝒪_{Z_ℚ}$ is isomorphic.

\item[Property~\ref{il:x3a}] This holds because the geometric fibres of $f_ℚ$
  and $g_ℚ$ agree.

\item[Property~\ref{il:x4a}] Follows from \ref{il:11.8.1}, given that $γ$ is
  totally ramified over $[1:1]$ and that $γ\bigl( [1:1]\bigr) = [1:1]$.

\item[Property~\ref{il:x5a}] This property holds because it holds for $g_ℚ$, and
  because
  \[
    γ\bigl( [0:1] \bigr) = [0:1] \quad\text{and}\quad γ\bigl( [1:0] \bigr) =
    [1:0].
  \]
\end{description}

\subsubsection*{Step 1a.  Reduction mod $p$}
\CounterStep
\approvals{Arne & yes \\ Jorge & yes \\ Stefan & yes}

The fibration $g_ℚ: F_ℚ → ℙ¹_ℚ$ described above can be defined over a suitable
localisation of the ring of integers.  Let us therefore choose a proper model
$\wtilde{g}: ℱ → ℙ¹_{ℤ_T}$, where $T$ is a finite set of primes containing $2$
and $3$.  We now choose $p ∉ T$ and write $F := ℱ ⊗_{ℤ_T} 𝔽_p$.  In analogy, we
consider the morphism $g : F → ℙ¹_{𝔽_p}$ and the $𝔽_p$-point $x ∈ g^{-1}(1)$.
The following properties will then hold if $p$ is sufficiently large:
\begin{enumerate}
\item\label{il:x2b} The natural morphism $𝒪_{ℙ¹_{𝔽_p}} → g_* 𝒪_F$ is an
  isomorphism; this follows from \ref{il:x2a} and \cite[Prop.~9.4.4]{EGA4-3}.
  In particular, $g$ has geometrically connected fibres, by \cite[Thm.~12.69 on
  p.~348]{MR2675155}.

\item\label{il:x3b} No geometric fibre of $g$ is multiple: this follows from
  \ref{il:x3a} and the fact that multiplicities ``spread out well over an open
  on the base'', as follows for example from the arguments in
  \cite[Lem.~3.12]{MR4157111} and \cite[Lem.~3.9]{MR4216591}.

\item\label{il:x5b} The geometric fibres of $g$ over $0$ and $∞$ are simply
  connected and each irreducible component of these fibres has multiplicity at
  least two: this is a consequence of \ref{il:x5a} and the classical
  specialisation theorem for the étale fundamental group, see \cite[Exp.~X,
  Cor~2.4]{SGA1}.
\end{enumerate}
By \cite[Prop.~17.7.11]{EGA4-4}, we may choose a dense open set
$1 ∈ S° ⊆ ℙ¹_{𝔽_p} ∖ \{ 0, ∞ \}$ over which the morphism $g$ is smooth.

\subsubsection*{Step 1b.  Products and covers}
\approvals{Arne & yes \\ Jorge & yes \\ Stefan & yes}

Spreading out and using a suitable version of the cyclic covering trick -- which
features already in the proof of \cite[Thm.~3.1]{Stoppino}, and in Campana's
construction \cite[§5]{Cam05} -- we can now construct a commutative diagram of
morphisms between $𝔽_p$-varieties as follows,
$$
\xymatrixcolsep{2.5cm}\xymatrix{ %
  Y° \ar[r]_{π°} \ar[d]_{f} & X° \ar[r]_{φ°} \ar[d]^{h} & S° \ar@{=}[d] \ar@/_5mm/[ll]_{σ°}\\
  F ⨯ S° \ar[r]_{g ⨯ \Id} & ℙ¹_{𝔽_p} ⨯ S° \ar[r]_{pr_2} & S°.  %
}
$$
Set $X° := ℙ¹_{𝔽_p} ⨯ S°$ and let $φ°$ be the projection to the second factor.
Writing
$$
B_t := \{t\} ⨯ S° \quad\text{and}\quad %
B_Δ := \bigl\{ (t,s) ∈ ℙ¹_{𝔽_p} ⨯ S° \:|\: t = s \bigr\},
$$
the $𝔽_p$-morphism $h$ is the (separable!) triple cover which ramifies totally
over the curves $B_1$ and $B_Δ$,
\begin{align*}
  h : X° & → ℙ¹_{𝔽_p} ⨯ S°, & h (x, t) & := \Bigl( {\textstyle \frac{(x-t)³+t·(x-1)³}{(x-t)³+\hphantom{t·}(x-1)³}}, t\Bigr).\\
  \intertext{Take $Y°$ as the fibre product, and $σ°$ as the fibre product of the following morphisms,}
  α : S° & → X° = ℙ¹_{𝔽_p} ⨯ S°, & α(t) & := \bigl(1, t \bigr) \\
  β : S° & → F ⨯ S°, & β(t) & := (t, t).
\end{align*}
Finally, set $D° := \frac{1}{2}·h^* \bigl( B_0 + B_∞ \bigr)$.

\subsubsection*{Step 1c.  Summary}
\CounterStep
\approvals{Arne & yes \\ Jorge & yes \\ Stefan & yes}

The main properties of this construction are summarised as follows.
\begin{enumerate}
\item\label{il:l0} Since $F$ is smooth over $𝔽_p$, \cite[Prop.~17.7.11]{EGA4-4},
  we find that $φ°◦ π° : Y° → S°$ is smooth.  The geometric fibres of $π°$ are
  those of $g$, and therefore geometrically connected by \ref{il:x2b}.  The
  morphism $π°$ is smooth over an open subset of $X°$.

\item\label{il:l1} The choice $S°$ guarantees that $h$ is étale along $B_0$ and
  $B_∞$.  The pair $\bigl(X°, D°\bigr)$ is therefore relatively snc over $S°$.
  The Sequence~\eqref{eq:log6-4} does not split when restricted to the generic
  fibre $X_η$ of $φ°$.  Indeed, the existence of such splitting is equivalent to
  the existence of a vector field of the form
  \[
    v = \frac{∂}{∂ t} + \left( a(t) + b(t)x + c(t)x² \right) \frac{∂}{∂ x}
  \]
  which is tangent to $\supp D°$.  In other words, the derivation of $𝔽_p(t)[x]$
  determined by $v$ must preserve the ideal generated by
  $n(x,t)=(x-t)³+t·(x-1)³$, which is the numerator of the first component of
  $h(x,t)$, as well as the ideal generated by $d(x,t)=(x-t)³+ (x-1)³$, which is
  the denominator of the first component of $h(x,t)$.  An elementary computation
  shows that the ideal generated by $d(x,t)$ is preserved by $v$ if, and only
  if,
  \[
    (a(t), b(t), c(t)) = \left( \frac{-1}{t-1}, \frac{1}{t-1}, 0 \right) \, .
  \]
  Likewise, the ideal generated by $n(x,t)$ is preserved by $v$ if, and only if,
  \[
    (a(t), b(t), c(t)) = \left( \frac{-2}{3(t-1)},\frac{2t-1}{3t(t-1)}, \frac{1}{3t(t-1)} \right) \, .
  \]
  Therefore, no vector field $v$ of the form above is tangent to $\supp D°$.  We
  conclude that Sequence~\eqref{eq:log6-4} does not split, as claimed.

\item\label{il:l2} Using that $X°_η$ is rational and that $\supp D° → S°$ is
  six-to-one, we find that the degree of $K_{X°}+D°$ on the generic fibre $X_η$
  of $φ°$ equals one.

\item\label{il:l4} We claim that the support of $(π°)^* D°$ is isomorphic to the
  product of $\supp D°$ with a geometrically simply connected curve, and that
  every component $(π°)^* ⌈D°⌉$ has multiplicity at least $2$.  This follows
  from \ref{il:x5b} since
  \begin{align*}
    (π°)^* ⌈D⌉ & = (π°)^* h^* \bigl(B_0 + B_∞ \bigr) = f^* (g ⨯ \Id)^* \bigl(B_0 + B_∞ \bigr) \\
               & = f^* \Bigl( g^*\bigl(\{0\} + \{∞\} \bigr) ⨯ S° \Bigr).
  \end{align*}

\item\label{il:l5} If $E° ⊂ X°$ is any prime divisor which dominates $S°$, then
  the coefficients of $(π°)^* E°$ are coprime.  This is again a consequence of
  \ref{il:x3b}, using the fact that the fibres of $π°$ are those of $g$.
\end{enumerate}

\subsection*{Step 2.  Verification of properties}
\approvals{Arne & yes \\ Jorge & yes \\ Stefan & yes}

Writing $η$ for the generic point of $S°$, with residue field $K := 𝔽_p(t)$, we
view $Y_η$ is a smooth projective surface over $K$, equipped with a surjection
$π°_η : Y_η → X°_η = ℙ¹_K$.  It remains to show that the surface $Y_η$ is of
general type, geometrically simply connected, and that $π°_η\bigl( Y_η(K)\bigr)$
is finite and not empty.

\subsubsection*{Step 2a.  General type}
\approvals{Arne & yes \\ Jorge & yes \\ Stefan & yes}

Since $F_ℚ$ is of general type over $ℚ$, the surface $F_{𝔽_p}$ is of general
type over $𝔽_p$, and $(F⨯S°)_η$ is of general type over $K$.  Since $h$ is
separable, so is the induced morphism of $K$-varieties, $f_η : Y_η → (F⨯S°)_η$.
As a separable cover of a surface of general type, $Y_η$ is then itself of
general type.

\subsubsection*{Step 2b.  Simple connectedness}
\approvals{Arne & yes \\ Jorge & yes \\ Stefan & yes}

Writing $\overline{K}$ for an algebraic closure of $K$, we need to show that
$Y_{\overline{K}}$ is simply connected.  Item~\ref{il:l0} equips us with a
proper map $π_{\overline{K}} : Y_{\overline{K}} → ℙ¹_{\overline{K}}$ with
connected fibres.  Better still, Item~\ref{il:l4} asserts that
$π_{\overline{K}}$ admits at least one simply connected fibre, while
Item~\ref{il:l5} asserts that $π_{\overline{K}}$ has no multiple fibre.
Lemma~\ref{lem:fundamentalgroups} therefore applies to show the simple
connectedness.

\subsubsection*{Step 2c.  Rational points}
\approvals{Arne & yes \\ Jorge & yes \\ Stefan & yes}

The existence of $σ°$ shows that $Y°(K)$ is not empty.  To prove finiteness, we
pass to the algebraic closure, $k := \overline{𝔽_p}$, consider the sequence of
morphisms over $k$,
$$
\xymatrixcolsep{2.5cm}\xymatrix{ %
  Y°_k \ar[r]^{π°_k\text{, dominant}} & X°_k \ar[r]^(.4){φ°_k\text{, dominant}} & S°_k ⊆ ℙ¹_k, %
}
$$
and use the Mordell-type theorem for $\cC$-integral points,
Theorem~\ref{thm:Mordell}, to show the stronger statement that
$π°_k\Bigl(Y°_k\bigl(k(t)\bigr) \Bigr) ⊆ X°_k\bigl(k(t)\bigr)$ is finite.  To
this end, choose compactifications,
$$
\xymatrixcolsep{2.5cm}\xymatrix{ %
  Y_k \ar[r]^{π_k\text{, dominant}} & X_k \ar[r]^(.4){φ_k\text{, dominant}} & S_k = ℙ¹_k, %
}
$$
such that $X_k$ is $k$-smooth and $(X_k, D_k)$ is snc, where $D_k$ is the
Zariski-closure of the divisor $D°_k$.  Recalling that smoothness and
non-vanishing of the Kodaira-Spencer map remain invariant when passing from
$𝔽_p$ to $k$, we find that the morphism $φ_k : X_k → S_k$ and the pair
$(X_k, D_k)$ satisfy all assumptions made in Theorem~\ref{thm:Mordell}.  In
particular, there are at most finitely many $\cC$-integral points $S_k → X_k$.
To conclude, we need to check that given any $k(t)$-valued point
$γ ∈ Y°_k\bigl(k(t)\bigr)$, or equivalently any section $γ : S_k → Y_k$, then
either $(π◦γ)$ is $\cC$-integral, or $\Image(π◦γ) ⊆ \supp D$.  Recalling from
\ref{il:l0} that $π°_k$ is smooth over an open of $X°_k$, this follows from
Lemma~\ref{lem:orbibase}; hence we are done.  \qed

%
%
\svnid{$Id: 12-construction.tex 584 2021-08-09 10:08:58Z kebekus $}

\section{Construction of the example --- proof of Theorem~\ref*{maintheorem1}}
\label{sec:construction}
\subversionInfo
\approvals{Arne & yes \\ Jorge & yes \\ Stefan & yes}

The goal of this paragraph is to prove Theorem~\ref{maintheorem1}.  We will
explain how to construct examples of simply connected fourfolds over global
fields for which the failure of the local-global principle is \emph{not}
explained by a Brauer--Manin obstruction.  We will mimic the construction
presented in \cite[Prop.~3.2]{Smeets}.  While most of the arguments given there
carry over to our setting, the construction in \cite{Smeets} is done over number
fields.  A few adjustments are required.

\subsection{Comparison of Brauer groups in conic bundles}
\approvals{Arne & yes \\ Jorge & yes \\ Stefan & yes}

The following preliminary lemma extends \cite[Prop.~2.2.(i)]{CTPS} to positive
characteristic.  For the sake of simplicity, we restrict ourselves to the case
where the characteristic is odd.  We do expect, however, that the result should
remain true even in characteristic two.

\begin{lem}[Comparison of Brauer groups in conic bundles]\label{lem:Brauersurjective}
  Let $B$ be a smooth, projective, geometrically integral variety over a field
  $K$ of characteristic different from two.  Let $f: W → B$ be a conic bundle,
  i.e., a surjective, flat morphism from a smooth, projective, geometrically
  integral $K$-variety $W$ to $B$, the generic fibre of which is a conic.
  Assume that there exists a codimension-one point $P$ on $B$ such that for any
  other codimension-one point $Q ≠ P$ on $B$, the fibre $f^{-1}(Q)$ is a smooth
  conic.  Then, the induced morphism $f^*: \Brauer(B) → \Brauer(W)$ is
  surjective.
\end{lem}
\begin{proof}
  The argument is based on the proof of \cite[Prop.~2.2.(i)]{CTPS}, with some
  modifications needed to deal with the typical subtleties in positive
  characteristic.  Let $η = \Spec K(B)$ be the generic point of $B$, the generic
  fibre $W_η$ is then a smooth conic over $K(B)$.  The following diagram
  summarises the pull-back morphisms between the Brauer groups that will be
  relevant for us,
  $$
  \xymatrix{ %
    \Brauer(B) \ar@{^(->}[r]^{ι_η^*} \ar[d]_{f^*} & \Brauer(η) \ar@{->>}[d]^{\txt{\scriptsize $f_η^*$, surjective\\\scriptsize since conic}} \\
    \Brauer(W) \ar@{^(->}[r]_{ι_{W_η}^*} & \Brauer(W_η).
  }
  $$

  Fix a prime $ℓ$, possibly equal to the characteristic $p$ of $K$, and let
  $α ∈ \Brauer(W)[ℓ^∞]$.  We need to show that there exists a class
  $β ∈ \Brauer(B)$ with $f^*(β) = α$.  If $ℓ$ is prime to $p$ (in particular, if
  $ℓ = 2$) then the argument from \cite[Prop.~2.2.(i)]{CTPS}, which uses
  residues, works verbatim.  Let us therefore assume that $ℓ = p$.  Set
  $α_η := ι_{W_η}^*α$.  Since $f_η^*$ is surjective, there exists an element
  $β_η ∈ \Brauer(η)$ which maps to $α_η$.  We claim that $β_η ∈ \Image ι_η^*$.
  A well-known purity result for the Brauer group, \cite[2.5]{Gabber93}, implies
  that it suffices to prove the following: if $Q ∈ B$ is a point of codimension
  $1$ on the base, inducing a discrete valuation on $K(B)$ with valuation ring
  $R = 𝒪_{B,Q}$, then $β_η ∈ \Brauer(R)$.

  Assume that one such $Q$ is given, and write $v$ for the associated valuation
  on $K(B)$.  The discussion in \cite[§3.5]{CTCetraro} implies that the smooth
  $K(B)$-conic $W_η$ admits a diagonal model over $R$ given by an equation of
  the form $x² - ay² - bz² = 0$, where $a ∈ R^⨯$, and $v(b) ∈ \{0,1\}$.  In
  fact, if $Q$ is not the point $P$ from the statement of the lemma, then both
  $a$ and $b$ can be chosen in $R_Q^⨯$.  In any case, the ring
  $R' = R[\sqrt{a}]$ is a quadratic unramified extension of $R$, and if $η'$
  denotes the generic point of $\Spec R'$, then $W_η$ acquires an $η'$-point --
  equivalently, $W_R$ acquires an $R'$-point.  Since $β_η$ maps to $α_η$, the
  restriction $β_{η'} ∈ \Brauer(η')$ maps to $α_{η'} ∈ \Brauer(W_η')$.  But
  $α_{η'} ∈ \Brauer(W_{R'})$ since $α ∈ \Brauer(W)$; evaluating on the
  $R'$-point mentioned above shows that $β_{η'} ∈ \Brauer(R')$.  Corestricting
  from $R'$ to $R$ (see \cite[\S 3.8]{CTSko}) then shows that $2β ∈ \Brauer(R)$.
  Since $p ≠ 2$, we conclude that $β ∈ \Brauer(R)$, as required.
\end{proof}

\subsection{Families of Châtelet surfaces}
\approvals{Arne & yes \\ Jorge & yes \\ Stefan & yes}

The second main ingredient in our proof of Theorem~\ref{maintheorem1} is the
following explicit construction of families of Châtelet surfaces.  The idea of
using Châtelet surfaces to address the problem at hand goes back to Poonen.

\begin{prop}[Families of Châtelet surfaces]\label{prop:Chateletconstruction}
  Let $K$ be a global field of odd characteristic.  Let $Y$ be a smooth,
  projective, geometrically integral $K$-variety.  Assume there exists a
  dominant, proper morphism $f: Y → ℙ¹_K$ with geometrically integral generic
  fibre, such that $f\bigl(Y(K)\bigr)$ is finite.  Then, there exists a family
  of Châtelet surfaces $g: S → ℙ¹_K$ such that the fibre product
  $Z := Y ⨯_{ℙ¹_K} S$ is a smooth, projective and geometrically integral
  $K$-variety for which $Z(K) = ∅$, even though $Z(𝔸_K)^{\emph{ét,Br}} ≠ ∅$.
\end{prop}
\begin{proof}
  Let us recall the construction from \cite[Prop.~3.2]{Smeets}.  Let
  $γ: ℙ¹_K → ℙ¹_K$ be a morphism which maps $f(Y(K))$ to $\{∞\}$.  We can then
  find an integer $n ≫ 0$ and a section
  $s ∈ Γ\bigl(ℙ¹_K ⨯_K ℙ¹_K,\, 𝒪(n,2)^{⊗ 2}\bigr)$, such that the following
  conditions are all satisfied.
  \begin{enumerate}
  \item\label{il:s1} The vanishing locus of $s$ is a smooth, proper,
    geometrically integral $K$-curve $C$.

  \item\label{il:s2} The composition $C ↪ ℙ¹_K ⨯_K ℙ¹_K \xrightarrow{π_1} ℙ¹_K$
    is étale at points where the composed morphism $h := γ◦f$ is not smooth.

  \item\label{il:s3} For suitable $a ∈ K$, the equation $y² - az² = s$ defines a
    conic bundle $\wtilde{S} → ℙ¹_K ⨯_K ℙ¹_K$ such that its restriction
    $\wtilde{S}_∞ := \wtilde{S}|_{\{∞\} ⨯ ℙ¹_K}$ is a Châtelet surface with
    $\wtilde{S}_∞(K) = ∅$, even though $\wtilde{S}_∞(𝔸_K) ≠ ∅$.
  \end{enumerate}
  Over number fields, the existence of a Châtelet surface $\wtilde{S}_∞$ as in
  Item~\ref{il:s3} follows from \cite[Prop.~5.1]{Poonen}.  Over global function
  fields of odd characteristic, the existence of such an $\wtilde{S}_∞$ is
  proven in \cite[§11]{Poonen}.  Once such a surface $\wtilde{S}_∞$ has been
  chosen, the proof of \cite[Prop.  3.2]{Smeets} (heavily based on \cite[§~6 and
  Lem.~7.1]{Poonen}) yields the existence of $s$ as above, and the composition
  $$
  \wtilde{S} → ℙ¹_K ⨯_K ℙ¹_K \xrightarrow{π_1}ℙ¹_K
  $$
  is a family of Châtelet surfaces.  Let $g: S → ℙ¹_K$ be the pullback of
  this family along $γ$; we claim that this yields the desired family for the
  statement of Proposition~\ref{prop:Chateletconstruction}.

  \begin{claim}\label{claim:torsors}
    Writing $Z := Y ⨯_{ℙ¹_K} S$, the morphism $Z → Y$ induces an equivalence of
    between the categories of finite étale covers, $\Etale(Y) → \Etale(Z)$.  If,
    moreover, $G$ is a finite étale group scheme over $K$ and if $Z' → Z$
    denotes any right $G$-torsor, there exists a right $G$-torsor $Y' → Y$ such
    that $Z' → Z$ and $Z ⨯_Y Y' → Z$ are isomorphic as $G$-torsors.
  \end{claim}
  \begin{proof}[Proof of Claim~\ref{claim:torsors}]
    Both statements follow from the fact that all geometric fibres of $Z → Y$
    are reduced and have trivial étale fundamental group.  For more details, we
    refer to the proofs of \cite[Lem.~8.1]{Poonen0} and \cite[Prop.~2.3]{CTPS};
    the arguments given there are purely algebraic and work in any
    characteristic.  \qedhere~(Claim~\ref{claim:torsors})
  \end{proof}

  The arguments in the proof of \cite[Prop.~3.2]{Smeets} now work verbatim to
  show that the family of surfaces $g: S → ℙ¹_K$ constructed above satisfies all
  requirements.  To be precise, the geometric properties of $Z$ are a
  consequence of \ref{il:s1} and \ref{il:s2}.  Next, $Z(K) = ∅$ follows from the
  fact all fibres of $Z → Y$ over $K$-rational points of $Y$ are isomorphic to
  $\wtilde{S}_∞$, which satisfies $\wtilde{S}_∞(K) = ∅$ by \ref{il:s3}.
  Finally, $Z(𝔸_K)^{\textrm{ét,Br}} ≠ ∅$ is a consequence of
  Lemma~\ref{lem:Brauersurjective} and the fact that $\wtilde{S}_∞(𝔸_K) ≠ ∅$,
  again by \ref{il:s3}.  \end{proof}

\begin{rem}\label{remark:pi1}
  As in \cite[Rem.~3.3]{Smeets}, one sees that
  $π_1^{\text{ét}}(Z) ≅ π_1^{\text{ét}}(Y)$ in the situation of
  Proposition~\ref{prop:Chateletconstruction}: this follows from the arguments
  given in \cite{Smeets} since conics, whether they are smooth or singular, are
  geometrically simply connected in any characteristic.
\end{rem}

\begin{rem}
  The proof of Proposition~\ref{prop:Chateletconstruction} starts by
  constructing a section $s$ of the bundle $𝒪(n,2)^{⊗ 2}$.  If the
  characteristic of $K$ is equal to two, one can still find $s$ such that the
  surface defined by the modified equation $y² + yz + az² = s$ satisfies
  Properties~\ref{il:s1}--\ref{il:s3}.  The existence of a suitable Châtelet
  surface $\wtilde{S}_∞$ in characteristic two follows from a result of Viray,
  \cite[Thm.~1.1]{Viray}.  We chose to restrict ourselves to odd characteristic
  because our previous result, Lemma~\ref{lem:Brauersurjective}, has been stated
  only in this setup.
\end{rem}

\subsection{Proof of Theorem~\ref*{maintheorem1}}
\approvals{Arne & yes \\ Jorge & yes \\ Stefan & yes}

Theorem~\ref{maintheorem2} yields the existence of a global field $K$ and a
smooth, projective, geometrically simply connected $K$-surface $S$, which comes
equipped with a dominant, proper morphism $f: S → ℙ¹_K$, such that
$f\bigl(S(K)\bigr)$ is finite and non-empty.
Proposition~\ref{prop:Chateletconstruction} then yields a smooth, projective
fourfold over $K$ such that $Z(K) = ∅$, whereas $Z(𝔸_K)^{\text{ét,Br}} ≠ ∅$.  We
have seen in Remark~\ref{remark:pi1} that $Z$ is geometrically simply connected,
so that $Z(𝔸_K)^{\text{ét,Br}} = Z(𝔸_K)^{\text{Br}}$.
Theorem~\ref{maintheorem1} is therefore shown.  \qed

\end{document}